\documentclass{article}

\usepackage[utf8]{inputenc} 
\usepackage[T1]{fontenc}    
\usepackage{url}            
\usepackage{booktabs}       
\usepackage{amsfonts}       
\usepackage{nicefrac}       
\usepackage{microtype}      

\usepackage{mathtools}
\usepackage{xcolor}
\newcommand*{\colorboxed}{}
\def\colorboxed#1#{%
	\colorboxedAux{#1}%
}
\newcommand*{\colorboxedAux}[3]{%
	\begingroup
	\colorlet{cb@saved}{.}%
	\color#1{#2}%
	\boxed{%
		\color{cb@saved}%
		#3%
	}%
	\endgroup
}

\usepackage{amssymb}
\usepackage{amsmath} 
\usepackage{amsthm} 
\newtheorem{Theorem}{Theorem}
\newtheorem{proposition}{Proposition} 
\newtheorem{lemma}{Lemma} 
\newtheorem{Assumption}{Assumption} 
\newtheorem{corollary}{Corollary} 
\newtheorem{Definition}{Definition}

\usepackage{enumitem}

\usepackage{algorithm}
\usepackage{algorithmic}

\usepackage{natbib}
\usepackage{graphicx}
\usepackage{subfigure}
\usepackage{adjustbox}

\DeclareMathOperator*{\E}{\mathbb{E}} 
\DeclareMathOperator*{\Var}{\mathbb{V}ar}
  
\newcommand{\floor}[1]{\left\lfloor  #1 \right\rfloor}
\newcommand{\ceil}[1]{\left\lceil  #1 \right\rceil}

\usepackage{cleveref}
\crefname{Assumption}{Assumption}{Assumptions}
\crefname{Theorem}{Theorem}{Theorems}

\usepackage[accepted]{icml2020}

\icmltitlerunning{History-Gradient Aided Batch Size Adaptation for Variance Reduced Algorithms}
\begin{document}

\twocolumn[
\icmltitle{History-Gradient Aided Batch Size Adaptation for Variance Reduced Algorithms}




\begin{icmlauthorlist}
\icmlauthor{Kaiyi Ji}{to}
\icmlauthor{Zhe Wang}{to}
\icmlauthor{Bowen Weng}{to}
\icmlauthor{Yi Zhou}{ed}
\icmlauthor{Wei Zhang}{goo}
\icmlauthor{Yingbin Liang}{to}
\end{icmlauthorlist}

\icmlaffiliation{to}{Department of Electrical and Computer Engineering, The Ohio State University}
\icmlaffiliation{ed}{Department of Electrical and Computer Engineering, University of Utah}
\icmlaffiliation{goo}{Department of Mechanical Engineering, Southern University of Science and Technology}

\icmlcorrespondingauthor{Kaiyi Ji}{ji.367@osu.edu}

\icmlkeywords{Machine Learning, ICML}

\vskip 0.3in
]



\printAffiliationsAndNotice{}  

\begin{abstract}
Variance-reduced algorithms, although achieve great theoretical performance, can run slowly in practice due to the periodic gradient estimation with a large batch of data. Batch-size adaptation thus arises as a promising approach to accelerate such algorithms. However, existing schemes either apply prescribed batch-size adaption rule or exploit the information along optimization path via additional backtracking and condition verification steps. In this paper, we propose a novel scheme, which eliminates backtracking line search but still exploits the information along optimization path by adapting the batch size via history stochastic gradients. We further theoretically show that such a scheme substantially reduces the overall complexity for popular variance-reduced algorithms SVRG and SARAH/SPIDER for both conventional nonconvex optimization and reinforcement learning problems. To this end, we develop a new convergence analysis framework to handle the dependence of the batch size on history stochastic gradients. Extensive experiments validate the effectiveness of the proposed batch-size adaptation scheme. 
\vspace{-0.2cm}
\end{abstract}

\section{Introduction}\label{sec:intro}
Stochastic gradient descent (SGD)~\cite{ghadimi2013stochastic} algorithms have been extensively used to efficiently solve large-scale optimization problems recently. Furthermore, various variance reduced algorithms such as SAGA~\cite{Defazio2014}, SVRG~\cite{johnson2013accelerating,reddi2016stochastic}, SARAH~\cite{nguyen2017sarah,nguyen2017stochastic}, and SPIDER~\cite{fang2018spider}/SpiderBoost~\cite{wang2019spiderboost}, have been proposed to reduce the variance of SGD. Such variance reduction techniques have also been applied to policy gradient algorithms to develop SVRPG~\cite{Papini2018}, SRVR-PG~\cite{xu2019sample} and SARAPO~\cite{yuan2018policy} in reinforcement learning (RL).  
Though variance reduced algorithms have been shown to have order-level lower computational complexity than SGD (and than vanilla policy gradient in RL), they often do not perform as well as SGD in practice, largely due to the periodic  
 large-batch gradient estimation. In fact, variance-reduced gradient estimation plays an important role only towards the later stage of the algorithm execution, and hence a promising way to accelerate variance reduced algorithms is to adaptively increase the batch size.


Two types of batch-size adaptation schemes have been proposed so far to accelerate stochastic algorithms~\cite{smith2017don,friedlander2012hybrid,devarakonda2017adabatch}. The first approach follows a {\em prescribed} rule to adapt the batch size, which can be {\em exponential} and {\em polynomial} increase of batch size as in hybrid SGD (HSGD) \cite{friedlander2012hybrid,zhou2018new} and {\em linear} increase of batch size \cite{zhou2018new}. 
Moreover,~\citealt{harikandeh2015stopwasting} proposed to use exponential increase of batch size at each outer-loop iteration for SVRG. 

The second approach adapts the batch size based on the information along the optimization path. 
 For example, \citealt{de2016big,de2017automated} proposed Big Batch SGD, which adapts the batch size so that the resulting gradient and variance satisfy certain optimization properties. Since the batch size needs to be chosen {\em even before} its resulting gradient is calculated, the algorithm 
adopts  the backtracking line search to iteratively check that the chosen batch size ensures the resulting gradient to satisfy a variance bound. Clearly, the backtracking step adds undesired complexity, but seems to be unavoidable, 
 because the convergence analysis exploits the {\em instantaneous} variance bound. 

Our contribution lies in designing an easy-to-implement scheme to adapt the batch size, which  incorporates the information along the optimization path, but  
does not involve backtracking and condition verification. We further show by both theory and experiments that such a scheme achieves much better performance than vanilla variance reduced algorithms in both conventional optimization and RL problems.

\vspace{-0.2cm}
\subsection{Our Contributions}\label{sec:contribution}
\vspace{-0.1cm}

{\bf New batch-size adaptation scheme via history gradients:}  We propose to adapt the batch size of each epoch (i.e., each outer loop) of variance reduced algorithms SVRG and SPIDER inversely proportional to the average of stochastic gradients over each epoch, and call the algorithms as {\bf A}daptive {\bf ba}tch-size SVRG (AbaSVRG) and AbaSPIDER. We further apply the scheme to the variance reduced policy gradient algorithms SVRPG and SPIDER-PG (which refers to SRVR-PG in~\citealt{xu2019sample}) in RL, and call the resulting algorithms as AbaSVRPG and AbaSPIDER-PG. These algorithms initially use small batch size (due to large gradients) and enjoy fast iteration, and then gradually increase the batch size (due to the reduced gradients) and enjoy reduced variance and stable convergence. Further technical justification is provided in \Cref{sec:algorithm}. 

Since the batch size should be set at the beginning of the epoch at which point the gradients in that epoch has not been calculated yet. It is a similar situation as in~\citealt{de2016big,de2017automated}, which introduced the backtracking line search to guarantee the variance bound. Here, we propose to use the average of stochastic gradients over the {\em preceding} epoch as an approximation of the present gradient information to avoid the complexity of backtracking line search, which we further show theoretically to still achieve guaranteed improved performance.

{\bf New convergence analysis:} Since the updates in our algorithms depend on the past gradients, it becomes much more challenging to establish the provable convergence guarantee. The technical novelty of our analysis mainly lies in the following two aspects.
\begin{list}{$\bullet$}{\topsep=-0.2ex \leftmargin=0.12in \rightmargin=0.in \itemsep =-0.1in}
\vspace{0.16cm}

\item We develop a new framework to analyze the convergence of variance reduced algorithms with  batch size adapted to history gradients 
for nonconvex optimization. In particular, we bound the function values for each epoch by the average gradient in the preceding epoch due to the batch size dependence, which further facilitates the bounding of the accumulative change of the objective value over the entire execution. Such an analysis is different from the existing analysis of SVRG in  \citealt{reddi2016stochastic,li2018simple} and SPIDER/SpiderBoost in \citealt{fang2018spider,wang2019spiderboost}, which
 are based on guaranteeing
  the decrease of the objective value iterationwisely or epochwisely. 

\vspace{0.16cm}

\item We develop a simpler convergence analysis for SVRG-type algorithms than previous studies \cite{reddi2016stochastic,li2018simple} for nonconvex optimization, which allows more  flexible choices of parameters. More importantly, such an analysis fits well to the analysis framework we develop to handle the dependence of the adaptive batch size on the stochastic gradients in the previous epoch.
\end{list}


Based on the new analysis framework, we show that both AbaSVRG and AbaSPIDER for conventional nonconvex optimization and AbaSVRPG and AbaSPIDER-PG for policy optimization in RL achieve improved complexity over their corresponding vanilla counterpart (without batch-size adaptation). The worst-case complexity of these algorithms all match the best known complexity. We also provide the convergence analysis of AbaSVRG and AbaSPIDER for nonconvex problems under the {P\L} condition in Appendix~\ref{se:5}.  

{\bf Experiments:} We provide extensive experiments on both supervised learning and RL problems and demonstrate that the proposed adaptive batch-size scheme substantially speeds up the convergence of variance reduced algorithms.

\begin{figure*}[t]
	\begin{minipage}[t]{8.5cm}
		\null 
		\begin{algorithm}[H]
				\small
			\caption{AbaSVRG}
			\label{alg:svrg}
			\begin{algorithmic}[1]
				\STATE {\bfseries Input:} ${x}_0$,  $m$, 
				$B$, $\eta$, $c_\beta,c_\epsilon, \beta_1>0$ 
				\STATE{${\tilde x}^0 = {x}_0$}
				\FOR{$s=1, 2, ..., S$}
				\STATE{${x}_0^s = {\tilde x}^{s-1}$}
				\STATE{Sample $\mathcal{N}_s$ from $[n]$ without replacement, 
					\\ where $\color{blue}{N_s =\min\{c_\beta \sigma^2 \beta_s^{-1}, c_\epsilon\sigma^2\epsilon^{-1},n\} }$ } 
				\STATE{${g}^s = \nabla f_{\mathcal{N}_s}({\tilde x}^{s-1})$}
				\STATE{ $\color{blue}{\text{Set}\; \beta_{s+1} = 0}$}
				\FOR{$t = 1,2,...,m$ }
				\STATE {Sample  $\mathcal{B}$ from $[n]$ with replacement }
				\STATE{ ${v}_{t-1}^s = \nabla f_{\mathcal{B}}({x}_{t-1}^s) - \nabla f_{\mathcal{B}}({\tilde x}^{s-1})+{g}^s$}
				\STATE{${x}_t^{s}={x}_{t-1}^s-\eta{v}_{t-1}^s$}
				\STATE{$\color{blue}{\beta_{s+1} \leftarrow \beta_{s+1}  +\|{v}^s_{t-1}\|^2/m}$}
				\ENDFOR
				\STATE{  ${\tilde x}^s = {x}_m^s$ }
				\ENDFOR
				\STATE {\bfseries Output:} ${x}_{\zeta}$ from $\{ {x}_{t-1}^s\}_{s\in[S], t\in[m]}$ uniformly at random
			\end{algorithmic}
		\end{algorithm}
	\end{minipage}
	\hspace{0.1cm}
	\begin{minipage}[t]{8.5cm}
		\null 
		\begin{algorithm}[H]
				\small
			\caption{AbaSPIDER }
			\label{alg:spider}
			\begin{algorithmic}[1]
				\STATE {\bfseries Input:} ${x}_0$,  $m$, 
				$B$, $\eta$, $c_\beta,c_\epsilon, \beta_1>0$
				\STATE{${\tilde x}^0 = {x}_0$}
				\FOR{$s=1, 2, ..., S$}
				\STATE{${x}_0^s = {\tilde x}^{s-1}$}
				\STATE{Sample $\mathcal{N}_s$ from $[n]$ without replacement, 
					\\ where $\color{blue}{N_s =\min\{c_\beta\sigma^2 \beta_s^{-1}, c_\epsilon \sigma^2\epsilon^{-1},n\} }$ } 
				\STATE{${v}_0^s = \nabla f_{\mathcal{N}_s}({\tilde x}^{s-1})$}
				\STATE{${x}_1^{s}={x}_{0}^s-\eta{v}_{0}^s$}
				\STATE{ $\color{blue}{\text{Set}\; \beta_{s+1} = \|{v}_0^s\|^2/m}$}
				\FOR{$t = 1,2,...,m-1$ }
				\STATE {Sample $\mathcal{B}$ from $[n]$ with replacement }
				\STATE{ ${v}_{t}^s = \nabla f_{\mathcal{B}}({x}_{t}^s) - \nabla f_{\mathcal{B}}({x}_{t-1}^s)+{v}_{t-1}^s$}
				\STATE{${x}_{t+1}^{s}={x}_{t}^s-\eta{v}_{t}^s$}
				\STATE{$\color{blue}{ \beta_{s+1} \leftarrow \beta_{s+1}  +\|{v}^s_{t}\|^2 /m}$}
				\ENDFOR
				\STATE{  ${\tilde x}^s = {x}_m^s$ }
				\ENDFOR
				\STATE {\bfseries Output:} ${x}_{\zeta}$ from $\{ {x}_{t-1}^s\}_{s\in[S], t\in[m]}$ uniformly at random
			\end{algorithmic}
		\end{algorithm}
	\end{minipage}
\vspace{-0.2cm}
\end{figure*}

\vspace{-0.2cm}
\subsection{Related Work}
\vspace{-0.2cm}
{\bf Variance reduced algorithms for conventional optimization.}
In order to improve the performance of SGD~\cite{robbins1951}, various variance reduced algorithms have been proposed such as SAG \cite{Nicolas2012}, SAGA \cite{Defazio2014}, SVRG \cite{Allen_Zhu2016, johnson2013accelerating},  SARAH~\cite{nguyen2017sarah,nguyen2017stochastic,nguyen2019finite}, SNVRG~\cite{zhou2018stochastic}, SPIDER~\cite{fang2018spider}, SpiderBoost~\cite{wang2019spiderboost}. 
This paper shows that two representative algorithms SVRG and SPIDER can be equipped with the proposed adaptive batch size  and attain substantial performance gain.  

{\bf Variance reduced policy gradient  for RL:} Variance reduction methods have also been applied to policy gradient methods \citep{Richard2000} in RL. One way is to incorporate a baseline in the gradient estimator, e.g., \citealt{Williams1992,Weaver2001,wu2018variance}. Optimization techniques have also been applied. For example, \citealt{Papini2018,xu2019improved}  applied SVRG to develop stochastic variance reduced policy gradient (SVRPG) algorithm. \citealt{yuan2018policy} and \citealt{xu2019sample} applied SARAH/SPIDER to develop stochastic recursive gradient policy optimization (SARAPO) and
stochastic recursive variance reduced policy gradient (SRVR-PG), respectively. \citealt{shen2019hessian} developed Hessian aided policy gradient (HAPG). This paper shows that the batch size adaptation scheme can also be applied to  variance reduced policy gradient algorithms
to significantly improve their performance. 


{\bf Stochastic algorithms with adaptive batch size.} Adaptively changing the batch size emerges as a powerful approach for accelerating stochastic algorithms~\cite{smith2017don,friedlander2012hybrid,devarakonda2017adabatch}. Hybrid SGD (HSGD) applies {\em exponential} and {\em polynomial} increase of batch size \cite{friedlander2012hybrid,zhou2018new} and {\em linear} increase of batch size \cite{zhou2018new}. \citealt{de2016big,de2017automated} proposed Big Batch SGD with the batch size adaptive to the \textit{instantaneous} gradient and variance information (which needs to be ensured, e.g., by backtracking line search) at each iteration. Moreover,~\citealt{harikandeh2015stopwasting} and \citealt{lei2019adaptivity} proposed to use exponential increase of batch size at each outer-loop iteration respectively for SVRG and for an adaptively sampled variance reduced algorithm SCSG~\cite{lei2017non}. Our algorithms adapt the batch size to history gradients, which differs from the prescribed adaptive schemes, and is easier to implement than Big Batch SGD by eliminating backtraking line search and still guarantees the convergence.


We note that a concurrent work~\cite{sievert2019improving} also proposed an improved SGD algorithm by adapting the batch size to history gradients, but only as a practice without convergence proof. Our analysis framework is applicable to their algorithm as we show in Appendix~\ref{apen:sgd}. 







{\bf Notations.} Let $\wedge$ and $\vee$ denote the minimum and the maximum.
Let $[n]:=\{1,....,n\}$. For a set $\mathcal{S}$, let $S$ be its cardinality and define 
 $\nabla f_{\mathcal{S}}(\cdot):=\frac{1}{S}\sum_{i\in\mathcal{S} }\nabla f_i(\cdot)$.

\vspace{-0.2cm}
\section{Batch Size Adaptation for Nonconvex Optimization}\label{se:4}
\vspace{-0.1cm}

In this section, we consider the following finite-sum optimization problem
\begin{align}
\min_{{x}\in\mathbb{R}^d}\, f({x}) : = \frac{1}{n}\sum_{i=1}^n f_i({x}).  \tag{P}
\end{align}
In the context of 
machine learning problems, 
each  function $f_i(\cdot)$ evaluates the loss on a particular $i$-th data sample, and is generally nonconvex due to the complex models.



\subsection{Proposed Algorithms with Batch-Size Adaptation}\label{sec:algorithm}

Two popular variance reduced algorithms to solve the optimization problem~(P) are SVRG \cite{johnson2013accelerating} and SARAH~\cite{nguyen2017sarah}/SPIDER~\cite{fang2018spider}, which have been shown to outperform SGD. 
However, SVRG and SARAH/SPIDER often run slowly in practice due to the full/large-batch gradient evaluation at the beginning of each epoch. We propose a batch-size adaptation scheme to mitigate such an issue for these algorithms, and we call the corresponding algorithms as AbaSVRG and AbaSPIDER (see Algorithms~\ref{alg:svrg} and~\ref{alg:spider}). Note that AbaSPIDER adopts the improved version SpiderBoost~\cite{wang2019spiderboost} of the original SPIDER~\cite{fang2018spider}.

We take SVRG as an example to briefly explain our idea. Our analysis of SVRG shows that the decrease of the average function value over an epoch $s$ with length $m$ satisfies
\begin{align*}
\frac{\mathbb{E}(f({\tilde x}^{s+1}) - f({\tilde x}^{s})) }{m}\leq -\phi  \frac{\sum_{t=0}^{m-1} \mathbb{E}\|{v}_t\|^2}{m}+ \frac{\psi I_{(N_s<n)}}{N_s} 
\end{align*} 
where $\phi, \psi>0$ are constants, the indicator function $I_{(A)}$ equals $1$ if the event $A$ is true and $0$ otherwise, ${\tilde x}^s$ is  the snapshot in epoch $s$, ${v}_t$ is a stochastic estimation of $\nabla f({x}_t)$ within epoch $s$, and $N_s$ is the batch size used at the outer-loop iteration.  
The above bound naturally suggests that $N_s$ should be chosen such that the second term is at the same level as the first term, i.e., the batch size should adapt to the average stochastic gradient over the epoch, in which case the convergence guarantee follows easily. However, this is not feasible in practice,  because the batch size should be chosen at the beginning of each epoch, at which point the gradients in the same loop have not been calculated yet. Such an issue was previously solved in~\citealt{de2016big,de2017automated} via backtracking line search, which adds significantly additional complexity. Our main idea here is to use the stochastic gradients calculated in the previous epoch for adapting the batch size of the coming loop, and we show that such a scheme still retains the convergence guarantee and achieves improved computational complexity. 

More precisely, AbaSVRG/AbaSPIDER chooses the batch size $N_s$ at epoch $s$ adaptively to the average $\beta_s$ of stochastic gradients in the \textit{preceding} epoch $s-1$ as given below 
\begin{align*}
N_s=\min\{c_\beta \sigma^2\beta_s^{-1}, c_\epsilon\sigma^2 \epsilon^{-1},n\},\beta_s = \frac{\sum_{t=1}^m \|{v}_{t-1}^{s-1}\|^2}{m},
\end{align*}
where $c_\beta,c_\epsilon>0$ are constants and $\sigma^2$ is the variance bound. As a comparison, the vanilla SVRG and SPIDER pick a {\em fixed} batch size to be either $n$ or $\min\{c_\epsilon \sigma^2\epsilon^{-1},n\}$.

\vspace{-0.2cm}
\subsection{Assumptions and Definitions} 
\vspace{-0.1cm}
We adopt the following standard assumptions~\cite{lei2017non, reddi2016stochastic} for convergence analysis.
\begin{Assumption}\label{assum1}
The objective function in (P) satisfies:
\begin{list}{$\bullet$}{\topsep=0.ex \leftmargin=0.06in \rightmargin=0in \itemsep =-0in}
\item[(1)]$\nabla f_i(\cdot)$ is $L$-smooth for $ i\in[n]$, i.e.,  for any ${x,y}\in\mathbb{R}^d$, 
$\|\nabla f_i({x})- \nabla f_i({y})\| \leq L\|{x-y}\|$.
\item[(2)]$f(\cdot)$ is bounded below, i.e.,  $f^*=\inf_{{x}\in\mathbb{R}^d} f({x})>-\infty$.
\item[(3)]$\nabla f_i(\cdot)$ (with the index $i$ uniformly randomly chosen) has bounded variance, i.e.,  there exists a constant $\sigma>0$ such that for any ${x}\in\mathbb{R}^d$, 
$\frac{1}{n}\sum_{i=1}^n\|\nabla f_i({x})-\nabla f({x})\|^2 \leq \sigma^2.$
\end{list}
\end{Assumption}
 The item (3) of the bounded variance assumption is commonly adopted for proving the convergence of SGD-type algorithms  (e.g., SGD~\cite{ghadimi2013stochastic}) and stochastic variance reduced methods  (e.g., SCSG~\cite{lei2017non}) that draw a sample batch with size less than $n$ for gradient estimation at each outer-loop iteration. 
 
 In this paper, we use the gradient norm as the convergence criterion for nonconvex optimization. 
 \begin{Definition}
 We say that ${x}^\zeta$ is an $\epsilon$-accurate solution for the optimization problem (P) if $\mathbb{E}\|\nabla f({x}^\zeta)\|^2\leq \epsilon$, where  ${x}^\zeta$ is an output returned by a stochastic algorithm.
 \end{Definition}
 To compare the efficiency of different stochastic algorithms, we adopt the following stochastic first-order oracle (SFO) for the analysis of the computational complexity.
 \begin{Definition}
Given an input ${x}\in\mathbb{R}^d$, SFO randomly takes an index $i \in [n]$ and returns a stochastic gradient $\nabla f_i({x})$ such that $\mathbb{E}_i\big[\nabla f_i({x})\big]=\nabla f({x})$.
 \end{Definition}
 

\vspace{-0.2cm}
\subsection{Convergence Analysis for AbaSVRG}\label{sec:abasvrgss}
\vspace{-0.1cm}
Since the batch size of AbaSVRG is adaptive to the {\em history} gradients due to the component $c_\beta\sigma^2 \beta_s^{-1}$, the existing convergence analysis for SVRG type of algorithms in~\citealt{li2018simple, reddi2016stochastic} does not extend easily. Here, we develop a simpler analysis for SVRG algorithm than that in~\citealt{li2018simple, reddi2016stochastic} (and can be of independent interest), and enables to handle the dependence of the  batch size on the stochastic gradients in the past epoch in the convergence analysis for AbaSVRG. To compare more specifically, \citealt{reddi2016stochastic} introduced a Lyapunov function $R^{s}_{t} = \mathbb{E} [f({x}^{s}_t)+ c_t\|{x}_t^s-{\tilde x}^{s-1}\|^2]$ and proves that $R^s$ decreases by the accumulated gradient norms $\sum_{t=0}^{m-1} \mathbb{E}\|\nabla f({x}_{t}^s)\|^2$ within an epoch $s$, and \citealt{li2018simple} directly showed that $ \mathbb{E} f({x}^{s})$ decreases by $\sum_{t=0}^{m-1} \mathbb{E}\|\nabla f({x}_{t}^s)\|^2$ using tighter bounds. 
As a comparison, our analysis shows that $ \mathbb{E} f({x}^{s})$ decreases by the accumulated \textit{stochastic gradient} norms $\sum_{t=0}^{m-1} \mathbb{E}\|{v}_{t}^s\|^2$. More details  about our proof can be found in Appendix~\ref{apen:svrg}.
The following theorem provides a general convergence result for AbaSVRG.
\begin{Theorem}\label{th_svrg} 
Suppose Assumption~\ref{assum1} is satisfied. Let $\epsilon >0$ and $c_\beta,c_\epsilon\geq \alpha$ for certain constant $\alpha>0$. Let $\psi=\frac{2\eta^2 L^2 m^2}{B} + \frac{2}{\alpha}+2 $ and choose $\beta_1$, $\alpha$ and $\eta$ such that $\beta_1\leq  \epsilon S$ and $\phi = \frac{1}{2}-\frac{1}{2\alpha}-\frac{L\eta}{2}- \frac{\eta^2L^2m^2}{2B}>0$, where $S$ denotes the number of epochs.  Then, the output ${x}_\zeta$ returned by AbaSVRG satisfies 
\begin{align*}
\mathbb{E}\|\nabla f({x}_\zeta)\|^2   \leq  \frac{\psi(f({x}_0)- f^*)}{\phi \eta K} + \frac{\psi\epsilon}{\phi \alpha} + \frac{4\epsilon}{\alpha},
\end{align*} 
where $f^*=\inf_{{x}\in\mathbb{R}^d} f({x})$ and $K=Sm$ represents the total number of iterations.
\end{Theorem} 
Theorem~\ref{th_svrg} guarantees the convergence of AbaSVRG as long as $\phi$ is positive, i.e. $\frac{L\eta}{2}+ \frac{\eta^2L^2m^2}{B} \leq \frac{1}{2}-\frac{1}{2\alpha}$, and thus allows very flexible choices of the stepsize $\eta$, the epoch length $m$ and the mini-batch size $B$. Such flexibility and generality are also due to the aforementioned simpler proof that we develop for SVRG-type algorithms.

In the following corollary,  we provide the complexity performance of AbaSVRG under certain choices of parameters.
\begin{corollary}\label{co:svrg}
Under the setting of Theorem~\ref{th_svrg}, we choose 
the constant stepsize $\eta = \frac{1}{4L}$, the epoch length $m=\sqrt{B}$ (which $B$ denotes the mini-batch size) and  $c_\beta, c_\epsilon \geq 16$. 
Then, 
to achieve $\mathbb{E}\|\nabla f({x}^\zeta)\|^2\leq \epsilon$,  the total SFO complexity of  AbaSVRG is given by 
\vspace{-0cm}
\begin{align*}
&\underbrace{\sum_{s=1}^S\min\bigg\{{\color{blue}{  \frac{c_\beta\sigma^2 }{\sum_{t=1}^m \|{v}_{t-1}^{s-1}\|^2/m}      }}, c_\epsilon\sigma^2\epsilon^{-1},n\bigg\} + KB }_{\text{ \normalfont complexity of AbaSVRG}} \nonumber\\
&\quad {\color{blue} <} \underbrace{S\min\big\{c_\epsilon\sigma^2\epsilon^{-1},  n\big\}+  KB}_{\text{\normalfont complexity of  vanilla SVRG }} =\mathcal{O}\Big(\frac{n  \wedge \epsilon^{-1}}{\sqrt{B}\epsilon} + \frac{B}{\epsilon}\Big).
\end{align*}
If we choose $B=n^{2/3} \wedge \epsilon^{-2/3}$, then the worst-case  complexity is  given by $\mathcal{O}\left( \epsilon^{-1}(n \wedge \epsilon^{-1})^{2/3}\right).$
\end{corollary}
We make the following  remarks on Corollary~\ref{co:svrg}.

First, the worst-case SFO complexity  under the specific choice of $B=n^{2/3} \wedge \epsilon^{-2/3}$ matches the best known result for SVRG-type algorithms. More importantly, since the adaptive component $\frac{c_\beta\sigma^2}{\sum_{t=1}^m \|{v}_{t-1}^{s-1}\|^2/m}$ can be much smaller than $\min\{c_\epsilon \sigma^2\epsilon^{-1}, n\}$ during the optimization process particularly in the initial stage, the actual  SFO complexity of AbaSVRG can be much lower than that of SVRG with fixed batch size as well as the worst-case complexity of $\mathcal{O}\left(\frac{1}{\epsilon}(n \wedge \frac{1}{\epsilon})^{2/3}\right)$, as demonstrated in  our experiments. 

Second, our convergence and complexity results hold for any choice of mini-batch size $B$, and thus we can safely choose a small mini-batch size  rather than the large one  $n^{2/3} \wedge \epsilon^{-2/3}$ in the regime with large $n$ and $\epsilon^{-1}$. 
In addition, for a given $B$, the resulting worst-case complexity $\mathcal{O}\big(\frac{n  \wedge \epsilon^{-1}}{\sqrt{B}\epsilon} + \frac{B}{\epsilon}\big)$ still matches the best known order given by ProxSVRG+~\cite{li2018simple} for SVRG-type algorithms.

Third,  Corollary~\ref{co:svrg} sets the mini-batch size $B=m^2$ to obtain the best complexity order. However, our experiments suggest that $B=m$ has better performance. Hence, we also provide analysis for this case in Appendix~\ref{apen:a3}.

Note that although Corollary~\ref{co:svrg} requires $c_\beta, c_\epsilon \geq 16$, the same complexity result can be achieved by 
  more flexible choices for $c_\beta$ and $c_\epsilon$. More specifically, Theorem~\ref{th_svrg} only requires $c_\beta$ and $c_\epsilon$ to be greater than $\alpha$ for any positive $\alpha$. Hence, $\alpha = 0.5, \eta = 1/(4L), B = m^2, c_\beta>0.5$, and $c_\epsilon > 0.5$ are also valid parameters, which can be checked to yield the same complexity order in Corollary~\ref{co:svrg}. Hence, the choices of $c_\beta,c_\epsilon$ in the experiments are consistent with our theory for AbaSVRG. 



\vspace{-0.1cm}
\subsection{Convergence Analysis for AbaSPIDER}
\vspace{-0.1cm}
In this subsection, we study AbaSPIDER, and compare our results with that for AbaSVRG. Note that AbaSPIDER in~\Cref{alg:spider} adopts the improved version SpiderBoost~\cite{wang2019spiderboost} of the original SPIDER~\cite{fang2018spider}.

The following theorem provides a general convergence result for AbaSPIDER. 
\begin{Theorem}\label{th_spider}
	Suppose Assumption~\ref{assum1} holds. Let $\epsilon>0$ and $c_\beta,c_\epsilon\geq \alpha $ for certain constant $\alpha>0$. Let  $\psi=\frac{2\eta^2 L^2 m}{B} + \frac{2}{\alpha}+2 $ and choose $\beta_1$, $\alpha$ and $\eta$ such that  $\beta_1 \leq S\epsilon$ and $\phi = \frac{1}{2}-\frac{1}{2\alpha}-\frac{L\eta}{2}- \frac{\eta ^2 L^2 m}{2B}>0$. Then, the output ${x}_\zeta$ returned by AbaSPIDER satisfies 
	\begin{align*}
	\mathbb{E}\|\nabla f({x}_\zeta)\|^2  \leq	\frac{\psi(f({x}_0)- f^*)}{\phi \eta K} + \frac{\psi\epsilon}{2\phi \alpha} +\frac{4\epsilon}{\alpha},
	\end{align*}
	where $f^*=\inf_{{x}\in\mathbb{R}^d} f({x})$ and $K=Sm$.
\end{Theorem}
To guarantee the convergence,  AbaSPIDER allows a smaller mini-batch size $B$ than AbaSVRG, because AbaSPIDER requires $B\geq \Theta( m\eta^2)$ (see Theorem~\ref{th_spider}) to guarantee $\phi$ to be positive, whereas
AbaSVRG requires $B\geq \Theta( m^2\eta^2)$ (see Theorem~\ref{th_svrg}). 
Thus, to achieve the same-level of target accuracy, AbaSPIDER uses fewer mini-batch samples than AbaSVRG, and thus achieves a lower worst-case SFO complexity, as can be seen in the following corollary. 
\begin{corollary}\label{co:spider}
	Under the setting of Theorem~\ref{th_spider}, for any mini-batch size $B\leq n^{1/2}\wedge \epsilon^{-1/2}$, if we set the epoch length $m = (n \wedge \frac{1}{\epsilon})B^{-1}$,  the stepsize 
	$\eta =\frac{1}{4L}\sqrt{\frac{B}{m}}$ and  $c_\beta, c_\epsilon \geq 16$, then 
	to obtain an $\epsilon$-accurate solution ${x}_\zeta$,  the total  SFO complexity of AbaSPIDER is given by
	\begin{align*}
	&\underbrace{\sum_{s=1}^S\min\bigg\{{\color{blue}{  \frac{c_\beta\sigma^2 }{\sum_{t=1}^m \|{v}_{t-1}^{s-1}\|^2/m}      }}, c_\epsilon\sigma^2\epsilon^{-1},n\bigg\} + KB }_{\text{ \normalfont complexity of AbaSPIDER}}
	 \nonumber \\
	 &\quad {\color{blue}< } \underbrace{S\min\big\{ c_\epsilon\sigma^2\epsilon^{-1},n\big\} + KB }_{\text{ \normalfont complexity of vanilla SPIDER}} 
	=\mathcal{O}\big(\epsilon^{-1}\big(n \wedge \epsilon^{-1}\big)^{1/2}\big).
	\end{align*}
\end{corollary}
Corollary~\ref{co:spider} shows that for  a \textit{wide range} of mini-batch size $B$ (as long as {\small $B\leq n^{1/2}\wedge \epsilon^{-1/2}$}), AbaSPIDER achieves the near-optimal worst-case complexity $\mathcal{O}\big(\frac{1}{\epsilon}(n \wedge \frac{1}{\epsilon})^{1/2}\big)$ under a proper selection of $m$ and $\eta$. Thus, our choice of mini-batch size is much less restrictive than $B=n^{1/2} \wedge \epsilon^{-1/2}$ used by SpiderBoost~\cite{wang2019spiderboost}  to achieve the optimal complexity, which can be very large in practice. 
More importantly, our practical complexity can be much better than those of SPIDER~\cite{fang2018spider} and SpiderBoost~\cite{wang2019spiderboost} with a fixed batch size due to the  batch size adaptation. 

Similarly to the argument at the end of Section~\ref{sec:abasvrgss}, the parameters $c_\beta, c_\epsilon$ for AbaSPIDER   
in Corollary~\ref{co:spider} can be chosen more flexibly, e.g., $c_\beta, c_\epsilon>0.5$. Hence, the choices of $c_\beta, c_\epsilon$ in the experiments are consistent with our theory for AbaSPIDER.

\allowdisplaybreaks
\newcommand\numberthis{\addtocounter{equation}{1}\tag{\theequation}}

\newcommand{\numleq}[1]{\overset{\text{(#1)}}{\leq}}
\newcommand{\numequ}[1]{\overset{\text{(#1)}}{=}}
\newcommand{\numgeq}[1]{\overset{\text{(#1)}}{\geq}} 
\def\defeq{\mathrel{\mathop:}=}

\newcommand{\norml}[1]{\left\|#1\right\|} 
\newcommand{\inner}[2]{\left\langle #1, #2 \right\rangle}
\section{Batch Size Adaptation for Policy Gradient}\label{rlyoudiandiao}
 In this section, we demonstrate an important application of our proposed batch-size adaptation scheme to variance reduced policy gradient algorithms in RL. 

 \subsection{Problem Formulation}
 Consider a discrete-time Markov decision process (MDP) $\mathcal{M} = \{\mathcal{S}, \mathcal{A}, \mathcal{P}, \mathcal{R}, \gamma, \rho\}$, where $\mathcal{S}$ denotes the state space; $\mathcal{A}$ denotes the action space; $\mathcal{P}$ denotes the Markovian transition model, $\mathcal{P}(s^\prime |s,a)$ denotes the transition probability from state-action pair $(s,a)$ to state $s^\prime$; $\mathcal{R} \in [-R,R]$ denotes the reward function, $\mathcal{R}(s,a)$ denotes the reward at state-action pair $(s,a)$; $\gamma \in \left[0,1\right)$ denotes the discount factor; and $\rho$ denotes the initial state distribution. The agent's decision strategy  is captured by the policy $\pi:=\pi(\cdot|s)$, which represents the density function over space $\mathcal{A}$ at state $s$. Assume that the policy is parameterized by $\theta\in \mathbb{R}^d$. Then, the policy class can be represented as $\Pi = \{\pi_{\theta} | \theta \in \mathbb{R}^d\}$. 
 
We consider a MDP problem with a finite horizon $H$. Then, a trajectory $\tau$ consists of a sequence of states and actions $(s_0, a_0, \cdots, s_{H-1}, a_{H-1} )$ observed by following a policy $\pi_{\theta}$ and $s_0 \sim \rho$. The total reward of such a trajectory $\tau$ is given by $\mathcal{R}(\tau) = \sum_{t=0}^{H} \gamma^t \mathcal{R}(s_t,a_t)$. For a give policy $\pi_{\theta}$, the corresponding expected reward is given by $J(\theta) = \E_{\tau  \sim p(\cdot| \theta)} \mathcal{R}(\tau)$, where $p(\cdot| \theta)$ represents the probability distribution of the trajectory $\tau$ by following the policy $\pi_{\theta}$. The goal of the problem is to find a policy that achieves the maximum accumulative reward by solving 
 \begin{align}
 \max_{\theta \in \mathbb{R}^d} J(\theta), \quad \text{where } J(\theta) =   \E_{\tau  \sim p(\cdot| \theta)}\left[ \mathcal{R}(\tau)\right]. \tag{Q}
 \end{align}
 
  \begin{figure*}[h]  
	\begin{minipage}{0.5\linewidth}
		\begin{algorithm}[H]
			\small
			\caption{AbaSVRPG} \label{SVRPG} 
			\begin{algorithmic}[1]
				\STATE 	{\bf Input:}  $\eta, \theta_0, \epsilon, m, \alpha,  \beta>0$ 
				\FOR{ $k=0, 1, \ldots, K$}
				\IF { $\textrm{mod}(k,m) = 0$ } 
				\STATE Sample  $\{\tau_i\}_{i=1}^{N}$  from $p(\cdot|\theta_{k})$, {\color{blue} where $N$ is given by~\eqref{Batch_size} }
				\STATE  $v_k = \frac{1}{N} \sum_{i=1}^{N} g(\tau_i| \theta_{k})$
				\STATE $\tilde{\theta} = \theta_{k}$ and  $ \tilde{v} = v_k$
				\ELSE 
				\STATE Draw $\{\tau_i\}_{i=1}^{B}$ samples from $p(\cdot| \theta_k)$
				\STATE $v_k = \frac{1}{B} \sum_{i=1}^{B}  \big( g(\tau_i|\theta_k) -   \omega(\tau_i| \theta_{k}, \tilde{\theta}) g(\tau_i|\tilde{\theta} ) \big) + \tilde{v} $
				\ENDIF
				\STATE {$\theta_{k+1} = \theta_k + \eta  v_k$}
				\ENDFOR
				\STATE 	{ {\bf Output:} $\theta_\xi$   from $\{\theta_0, \cdots, \theta_{K}\}$ uniformly at random.}
			\end{algorithmic}
		\end{algorithm} 
	\end{minipage}  
	\begin{minipage}{0.5\linewidth} 
		\begin{algorithm}[H]
				\small
			\caption{AbaSPIDER-PG}\label{SpiderPG}
			\begin{algorithmic}[1]
				\STATE {\bf Input:}  $\eta, \theta_0, \epsilon, m, \alpha,  \beta>0$ 
				\FOR{ $k=0, 1, \ldots, K $}
				\IF { $\textrm{mod}(k,m) = 0$ }
				\STATE Sample $\{\tau_i\}_{i=1}^{N}$  from $p(\cdot| \theta_k)$, {\color{blue} where $N$ is given by~\eqref{Batch_size} } 
				\STATE   $v_{k} = \frac{1}{N} \sum_{i=1}^{N} g(\tau_i|\theta_k)$
				\ELSE 
				\STATE   Draw $\{\tau_i\}_{i=1}^{B}$ samples from $p(\cdot| \theta_k)$.
				\STATE {$v_k = \frac{1}{B} \sum_{i=1}^{B} \big( g(\tau_i| \theta_k)-$ 
			\\	\quad \quad \quad $ \omega(\tau_i| \theta_{k}, \theta_{k-1}) g(\tau_i| \theta_{k-1}) \big) + v_{k-1}$}
				\ENDIF
				\STATE {$\theta_{k+1} = \theta_k + \eta  v_k$}
				\ENDFOR
				\STATE 	{ {\bf Output:} $\theta_\xi$   from $\{\theta_0, \cdots, \theta_{K }\}$ uniformly at random.}
			\end{algorithmic}
		\end{algorithm}
	\end{minipage} 
\end{figure*} 

\vspace{-0.4cm}
\subsection{Preliminaries of Policy Gradient and Variance Reduction}
\vspace{-0.1cm}
Policy gradient is a popular approach to solve the problem (Q), which iteratively updates the value of $\theta$ based on the trajectory gradient of the above objective function. Since the distribution $p(\cdot| \theta)$ of $\tau$ is unknown because the MDP is unknown, policy gradient adopts the trajectory gradient (denoted by $g(\tau|\theta)$) based on the {\em sampled} trajectory $\tau$ for its update. Two types of commonly used trajectory gradients, namely REINFORCE \citep{Williams1992} and G(PO)MDP \citep{Baxter2001}, are introduced in \Cref{app:pg}. A key difference of such policy gradient algorithms from SGD in conventional optimization is that the sampling distribution $p(\cdot| \theta)$ changes as the policy parameter $\theta$ is iteratively updated, and hence trajectories here are sampled by a varying distributions during the policy gradient iteration. 
 
This paper focuses on the following two variance reduced policy gradient algorithms, which were developed recently to improve the computational efficiency of policy gradient algorithms. First, \citealt{Papini2018} proposed a stochastic variance reduced policy gradient (SVRPG) algorithm by adopting the SVRG structure in conventional optimization. In particular, SVRPG continuously adjusts the gradient estimator by {\em importance sampling} due to the iteratively changing trajectory distribution. 
 Furthermore,  \citealt{xu2019sample} applied the SARAH/SPIDER estimator in conventional optimization to develop 
a stochastic recursive variance reduced policy gradient (SRVR-PG) algorithm, which we refer to as SPIDER-PG in this paper.

\vspace{-0.2cm}
\subsection{Proposed Algorithms with Batch-size Adaptation}

 

Both  variance reduced policy gradient algorithms SVRPG and SPIDER-PG 
choose a large batch size $N$ for estimating the policy gradient at the beginning of each epoch. As the result, they  often  run slowly in practice, and do not show significant advantage over the vanilla policy gradient algorithms. This motivates us to apply our developed batch-size adaptation scheme to reduce their computational complexity.

Thus, we propose AbaSVRPG and AbaSPIDER-PG algorithms, as outlined in Algorithms~\ref{SVRPG} and~\ref{SpiderPG}. More specifically, we adapt the batch size $N$ 
based on  the average of the trajectory gradients in the preceding epoch as 
\begin{align}\label{Batch_size}
N = \frac{\alpha \sigma^2}{\frac{\beta}{m} \sum_{i=n_k-m}^{n_k-1} \norml{v_i}^2 + \epsilon},   
\end{align}
where  $k$ denotes the iteration number, $n_k = \floor{k/m}\times m$, and $\norml{v_{-1}} = \cdots =  \norml{v_{-m}} = 0$. 


\vspace{-0.2cm}
\subsection{Assumptions and Definitions}\label{assum:rl}
We take the following standard assumptions, as also adopted by~\citealt{xu2019improved,xu2019sample,Papini2018}.
\begin{Assumption} \label{g_assumption}
	The trajectory gradient $g(\tau | \theta)$ is an unbiased gradient estimator, i.e., $\E_{\tau \sim p(\cdot|\theta) } [g(\tau|\theta)] = \nabla J(\theta)$.
\end{Assumption}
Note that the commonly used trajectory gradients $g(\tau | \theta)$ such as REINFORCE and G(PO)MDP as given in \Cref{app:pg} satisfy \Cref{g_assumption}.

\begin{Assumption} \label{assumption}
	For any state-action pair $(s,a)$, at any value of $\theta$, and for $1 \leq
	i,j \leq d$, there exist constants $0 \le G,H, R < \infty$ such that $|\nabla_{\theta_i} \log \pi_{\theta} (a|s)| \leq  G$ and 
	\begin{align*}
	|\mathcal{R}(s,a)| \leq  R,\;\Big| \frac{\partial^2}{\partial \theta_i \partial \theta_j} \log \pi_{\theta} (a|s)\Big|  \leq H.
	\end{align*}
\end{Assumption}  
	\vspace{-0.3cm}

\Cref{assumption} assumes that the reward function $\mathcal{R}$, and the gradient and Hessian of $\log \pi_{\theta} (a|s)$ are bounded. 
\begin{Assumption} \label{assumption_bound_variance}
	The estimation variance of the trajectory gradient is bounded, i.e., 
	there exists a constant $\sigma^2 < \infty$ such that, for any $\theta \in \mathbb{R}^d$:
	\begin{align*}
	\Var[g(\tau|\theta)] = \E_{\tau \sim p(\cdot|\theta) }   \norml{g(\tau|\theta) - \nabla J(\theta) }^2   \leq \sigma^2.
	\end{align*}
\end{Assumption}
	\vspace{-0.3cm}

Since the problem (Q) in general is nonconvex, we take the following standard convergence criterion.
 \begin{Definition}
	We say that $\bar{\theta}$ is an $\epsilon$-accurate stationary point for the problem (Q) if $\E \norml{\nabla J(\bar\theta)}^2 \leq \epsilon$.
\end{Definition}
To measure the computational complexity of various  policy gradient methods, we take the stochastic trajectory-gradient oracle (STO) complexity as the metric, which meansures the number of trajectory-gradient computations to attain an $\epsilon$-accurate stationary point.

\subsection{Convergence Analysis for AbaSVRPG} 
In this subsection, we provide the convergence and complexity analysis for AbaSVRPG algorithm. Since the batch size of AbaSVRPG is adaptive to the trajectory gradients calculated in the previous epoch, we will adopt our new analysis framework to bound the change of the function value for each epoch by the trajectory gradients in the preceding epoch due to the batch size dependence. The challenge here arises due to the fact that the sampling distribution is time varying as the policy parameter is updated due to the iteration. Hence, the bound should be tightly developed in order to ensure the decrease of the accumulative change of the objective value over the entire execution of the algorithm. Such bounding procedure is very different from the convergence proofs for vanilla SVRPG in \citealt{Papini2018,xu2019improved}.

%


The following theorem characterizes the  convergence of AbaSVRPG. Let $\theta^* \defeq \arg  \max_{\theta \in \mathbb{R}^d}J(\theta) $.
\begin{Theorem}\label{SVRPG_conv}  
	Suppose \Cref{g_assumption}, \ref{assumption}, and \ref{assumption_bound_variance} hold.
Choose 
$\eta = \frac{1}{2L},  m =  \big( \frac{L^2\sigma^2}{Q\epsilon} \big)^{\frac{1}{3}} , B= \big( \frac{Q\sigma^4}{L^2\epsilon^2} \big)^{\frac{1}{3}}, \alpha = 48 \text{ and }  \beta = 6$, where $L>0$ is a Lipschitz constant given in Lemma~\ref{bounded_from_papini} in \Cref{em_verify}, and $Q$ is the constant given in \Cref{lispchitz_g} in \Cref{em_verify}. 
	Then, the output $\theta_\xi$ of AbaSVRPG  satisfies 
	\begin{align}
	\E \norml{\nabla J(\theta_{\xi})}^2 \leq   \frac{88L}{K+1} \left( J(\theta^*) -  J(\theta_{0})\right)    +  \frac{\epsilon}{2},
	\end{align}
	where $K$ denotes the total number of iterations.  
\end{Theorem}
\vspace{-0.2cm}

\Cref{SVRPG_conv} shows that the output of AbaSVRPG converges at a rate of $\mathcal{O}(1/K)$. Furthermore, the following corollary
captures the overall STO complexity of AbaSVRPG and its comparison with the vanilla SVRPG.
\begin{corollary}\label{cor:SVRPG_sto}
Under the setting of \Cref{SVRPG_conv},	 the overall STO complexity of {\small AbaSVRPG}  to achieve {\small $\E \norml{\nabla J(\theta_{\xi})}^2 \leq \epsilon $} is
{\small \begin{align*}
	 \underbrace{2KB + \sum_{k=0}^{n_K} \frac{\alpha \sigma^2}{\frac{\beta}{m} \sum_{i=km-m}^{km-1} \norml{v_i}^2 + \epsilon} }_{\text{\normalfont compelxity of AbaSVRPG}}  &{\color{blue} <} \underbrace{ 2KB + \sum_{k=0}^{n_K}   \frac{\alpha \sigma^2}{\epsilon}}_{\text{\normalfont compelxity of vanilla SVRPG}} \\  
	&= \mathcal{O} \left(  \epsilon^{-5/3} + \epsilon^{-1} \right).
	\end{align*} 
	}
\end{corollary}
\vspace{-0.3cm}

The STO complexity improves the state-of-the-art complexity of vanilla SVRPG characterized in \citealt{xu2019improved}, especially due to the saving samples at the initial stage. The worst-case STO complexity of AbaSVRPG is $\mathcal{O}(  \epsilon^{-5/3} + \epsilon^{-1} )$, which matches that of \citealt{xu2019improved}. 

\subsection{Convergence Analysis for AbaSPIDER-PG}
In this section, we provide the convergence and complexity analysis for AbaSPIDER-PG algorithm. 
\begin{Theorem} \label{SpiderPG_conv} 
	Suppose Assumptions~\ref{g_assumption}, \ref{assumption}, and \ref{assumption_bound_variance} hold. Choose $
	\eta = \frac{1}{2L},  m = \frac{L\sigma}{\sqrt{ Q\epsilon}}  , B= \frac{\sigma \sqrt{Q}}{L\sqrt{\epsilon}}, \alpha = 48 \text{ and } \beta = 6. 
	$ 	where $L>0$ is a Lipschitz constant given in Lemma~\ref{bounded_from_papini} in \Cref{em_verify}, and $Q$ is the constant given in \Cref{lispchitz_g} in \Cref{em_verify}. Then, the  output $\theta_\xi$ of AbaSPIDER-PG satisfies 
	\begin{align*}
	\E \norml{\nabla J(\theta_{\xi})}^2 \leq  \frac{40L}{K+1} \left( J(\theta^*) -  J(\theta_{0})\right) + \frac{\epsilon}{2}.
	\end{align*}
\end{Theorem}

\begin{corollary}\label{cor:SpiderPG_sto}
Under the setting of \Cref{SpiderPG_conv}, the total STO complexity of {\small AbaSPIDER-PG} to achieve {\small $\mathbb{E}\|\nabla J(\theta_\xi)\|^2 \le \epsilon$} is 
{\small
	\begin{align*}
	\underbrace{2KB + \sum_{k=0}^{n_K} \frac{\alpha \sigma^2}{\frac{\beta}{m} \sum_{i=km-m}^{km-1} \norml{v_i}^2 + \epsilon} }_{\text{\normalfont compelxity of AbaSPIDER-PG}}  &< \underbrace{ 2KB + \sum_{k=0}^{n_K}   \frac{\alpha \sigma^2}{\epsilon}}_{\text{\normalfont compelxity of vanilla SPIDER-PG}} \\  
	&= \mathcal{O}(  \epsilon^{-3/2} + \epsilon^{-1} ).
	\end{align*}  }
\end{corollary}

\vspace{-0.3cm}
\Cref{cor:SpiderPG_sto} shows that the worst-case STO complexity of AbaSPIDER-PG is $\mathcal{O}(  \epsilon^{-3/2} + \epsilon^{-1} )$, which orderwisely outperforms that of AbaSVRPG in \Cref{cor:SVRPG_sto}, by a factor  of $\mathcal{O}(  \epsilon^{-1/6})$. This is due to the fact that AbaSPIDER-PG avoids the variance accumulation problem of AbaSVRPG by continuously using the gradient information from the immediate preceding step (see Appendix~\ref{sec_variance_bound_SpiderPG}  for more details).

    \begin{figure*}[h]
   	\centering     
   	\subfigure{\includegraphics[width=41mm]{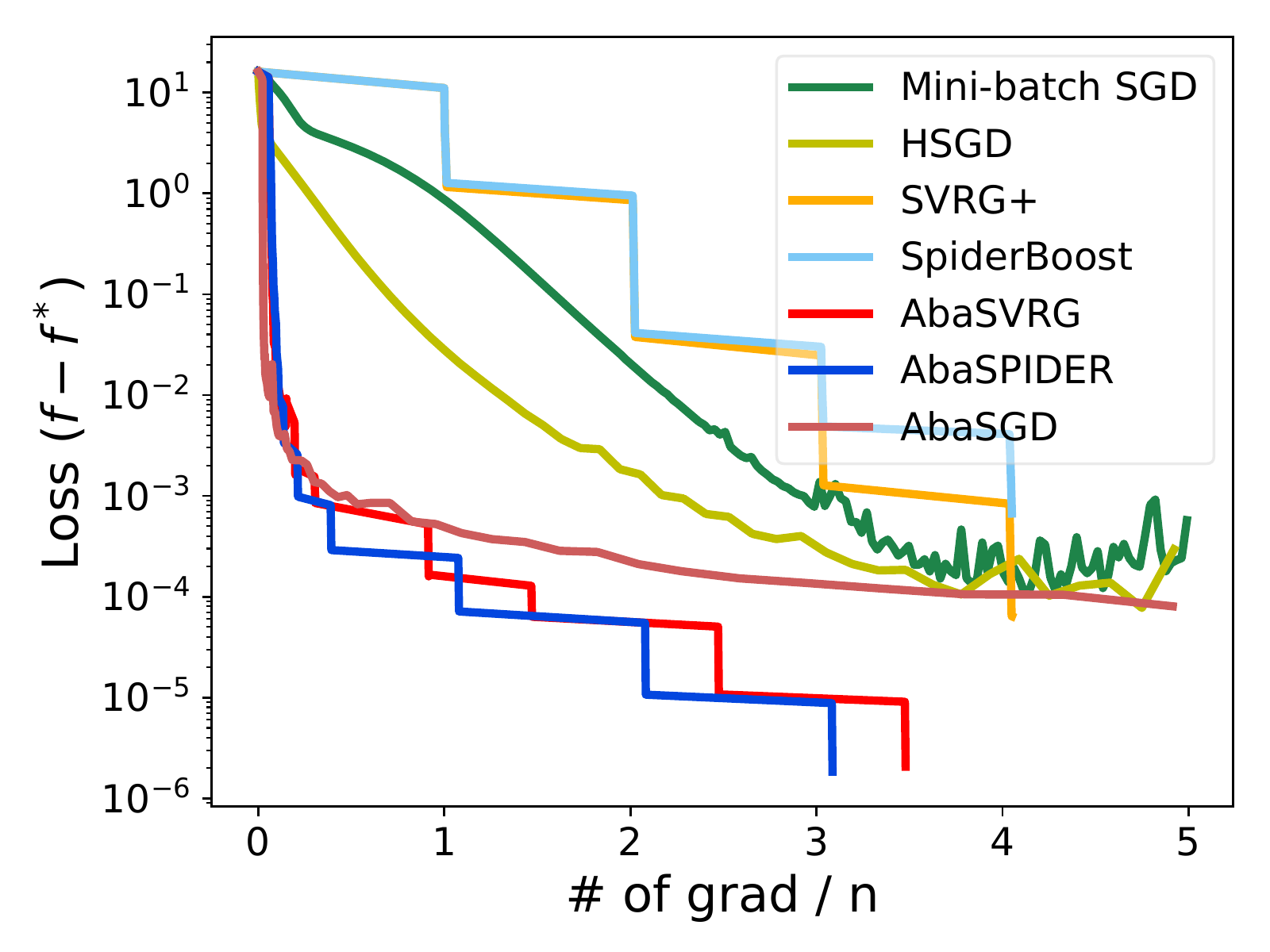}} 
   	\subfigure{\includegraphics[width=41mm]{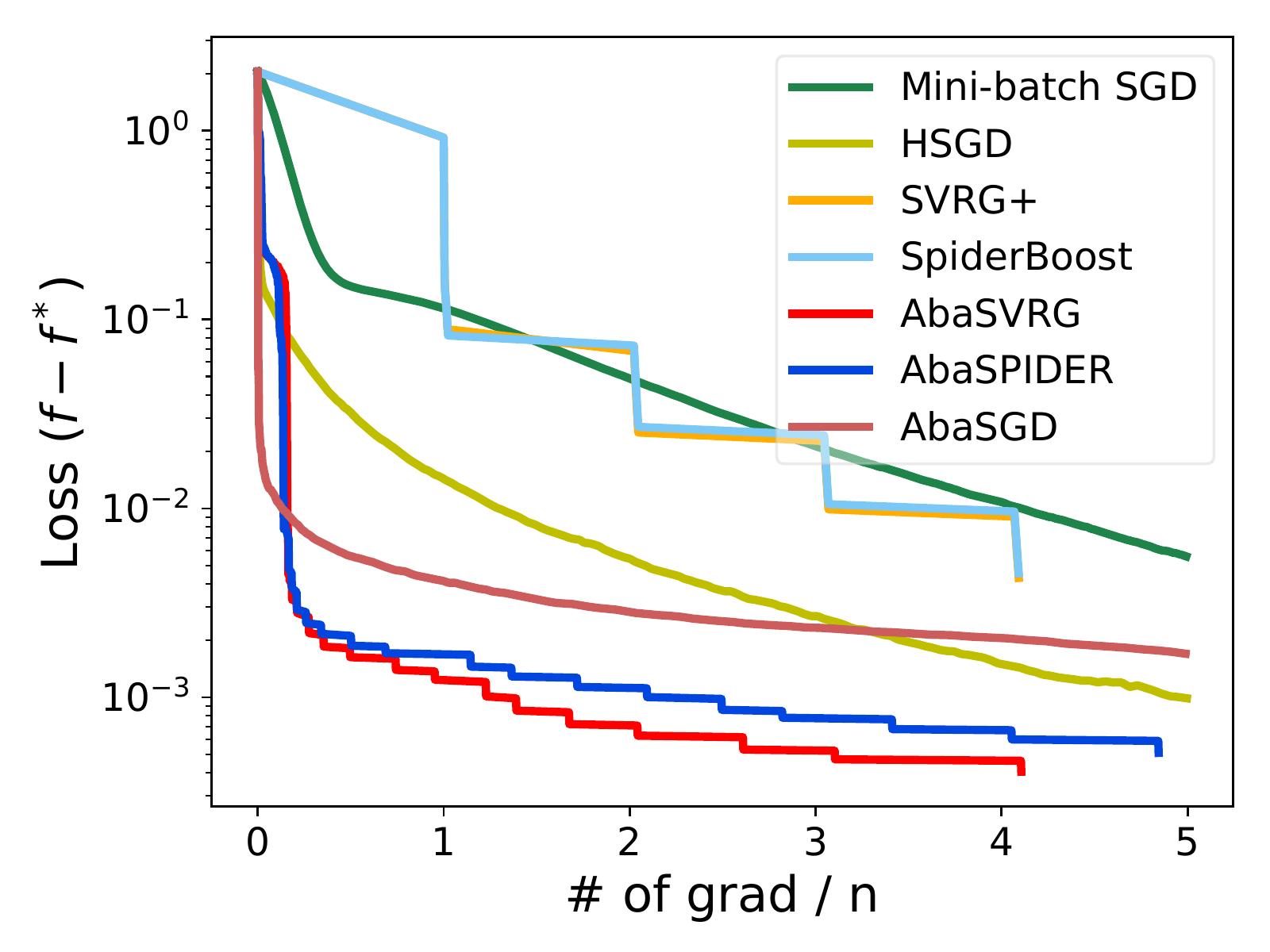}}
   	\subfigure{\includegraphics[width=41mm]{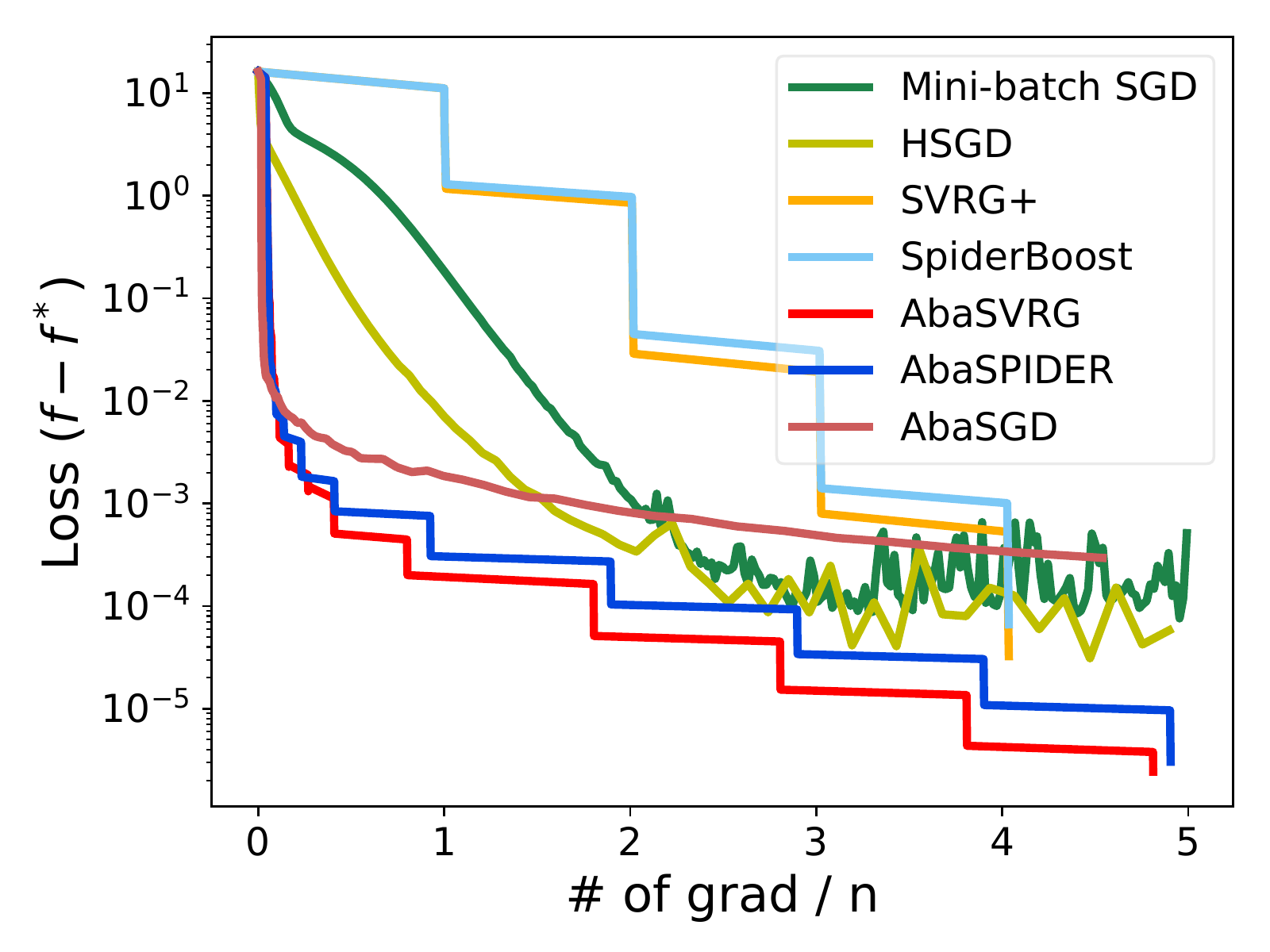}} 
   	\subfigure{\includegraphics[width=41mm]{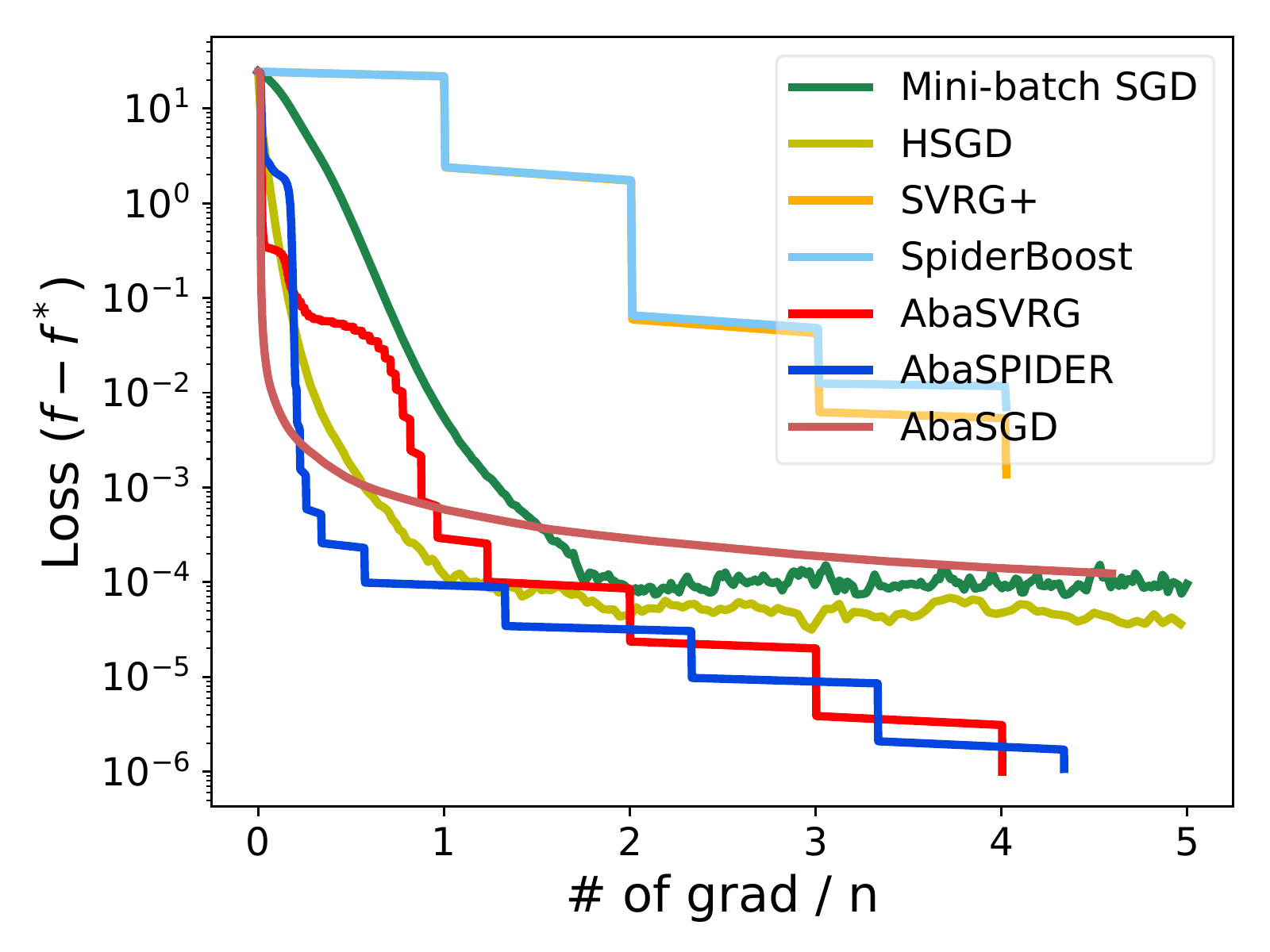}}
   	\caption{Comparison of various algorithms for  logistic regression problem over four datasets.  From left to right: a8a, ijcnn1,  a9a, w8a. }\label{figure:result1}
   \end{figure*}
 
 \begin{figure*}[h]
	\centering    
	\subfigure[dataset: a9a ]{\label{fig2:a}\includegraphics[width=41mm]{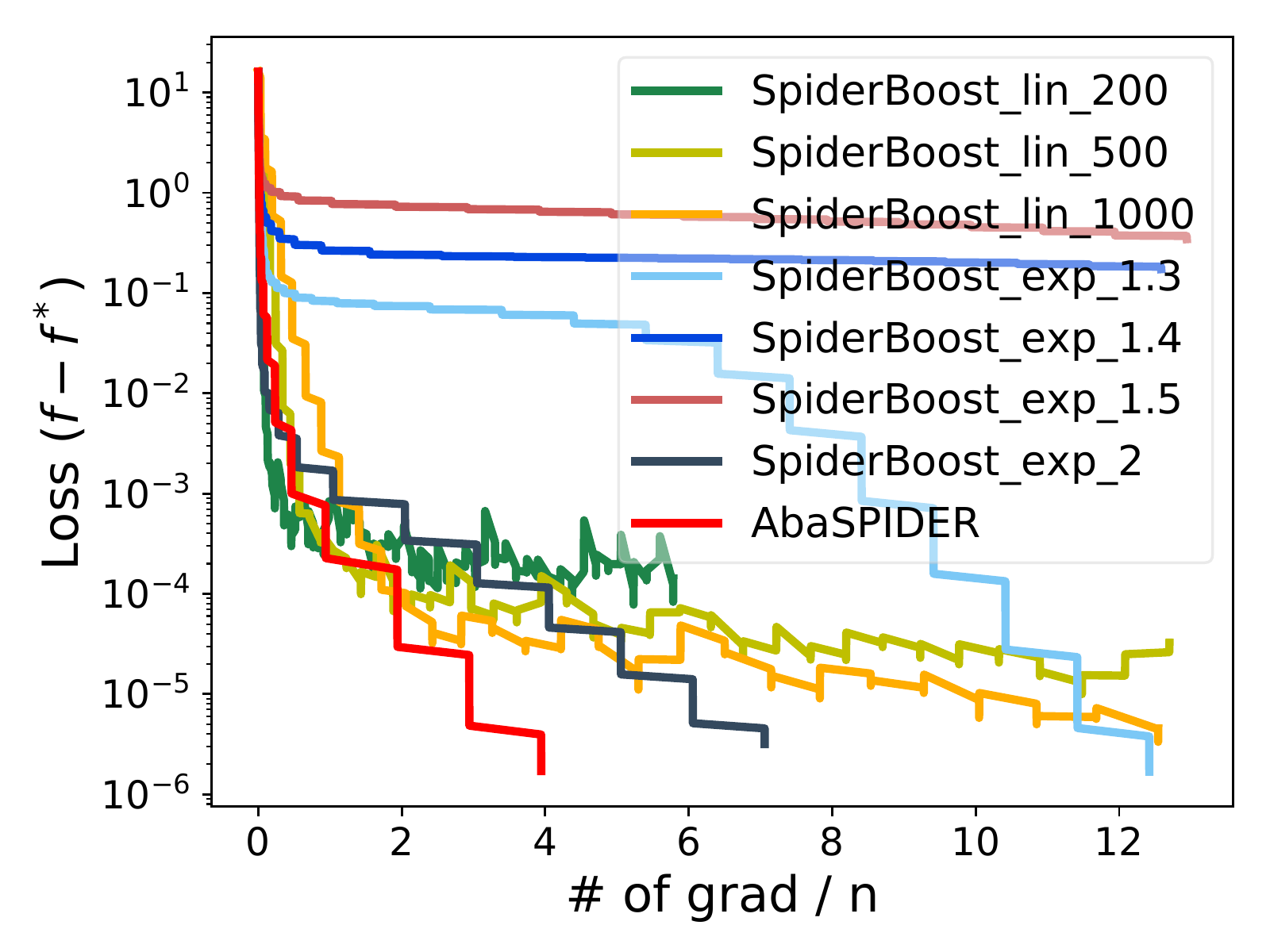}\includegraphics[width=41mm]{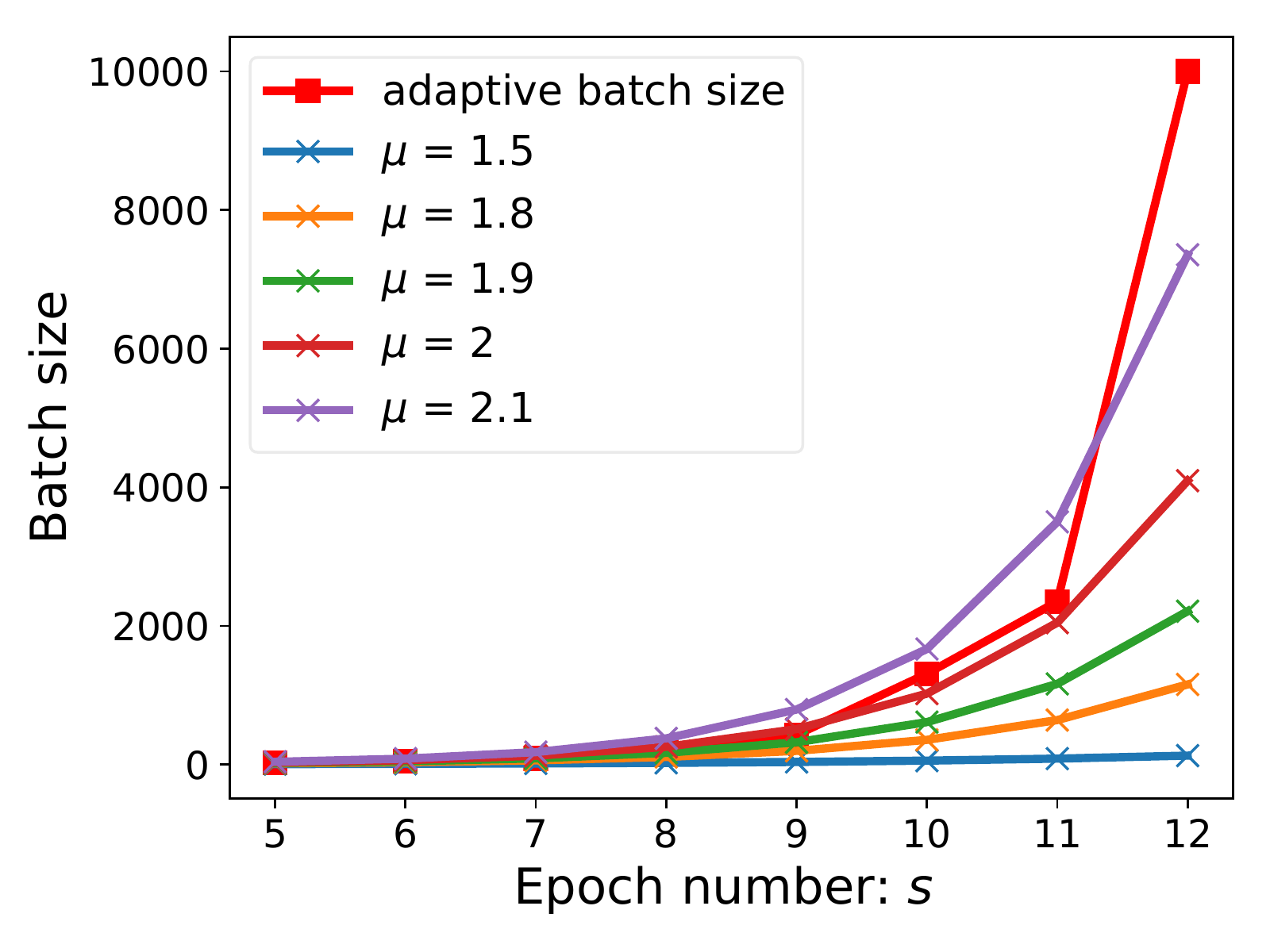}} 
	\subfigure[dataset: w8a ]{\label{fig2:b}\includegraphics[width=41mm]{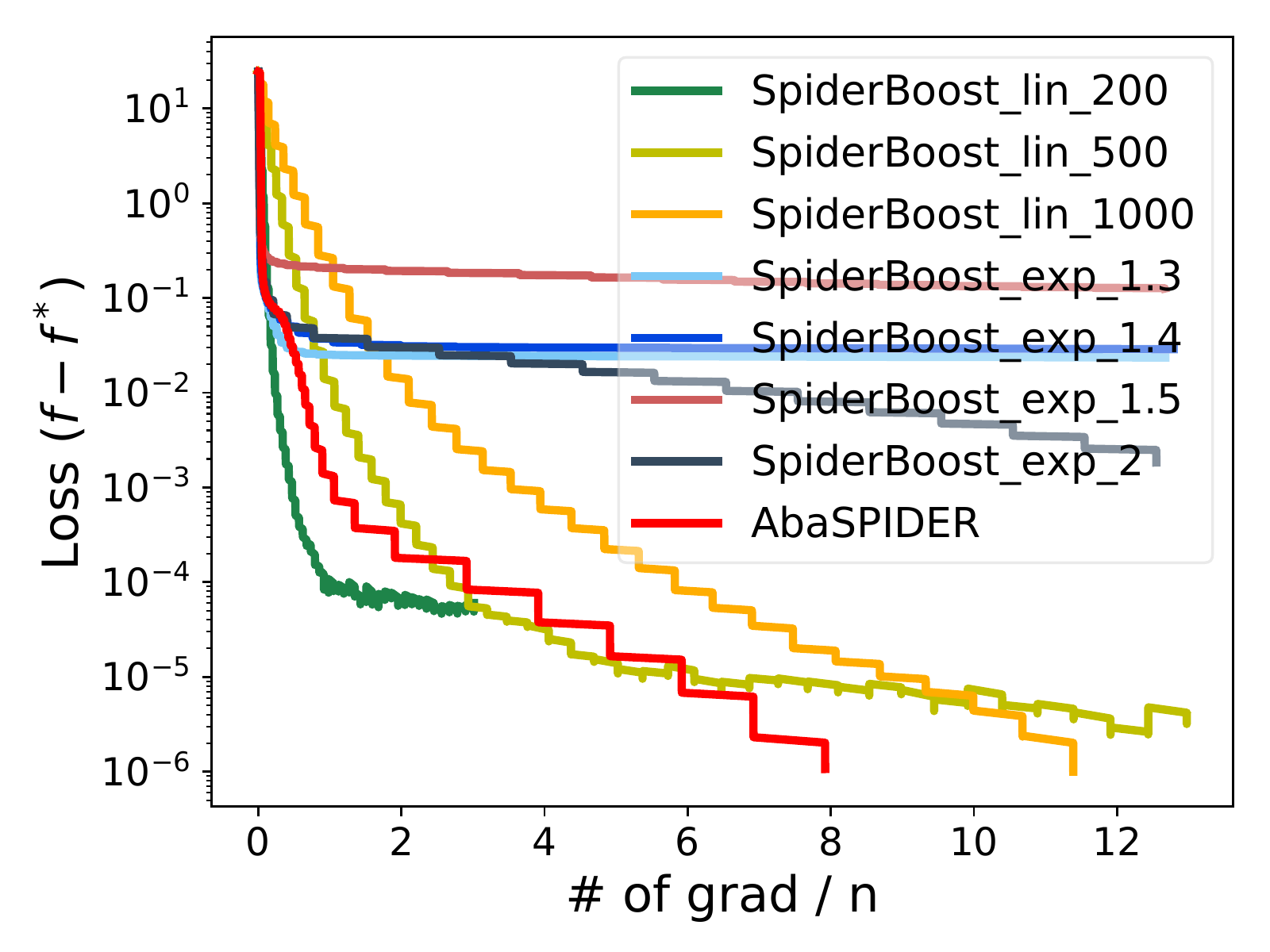}\includegraphics[width=41mm]{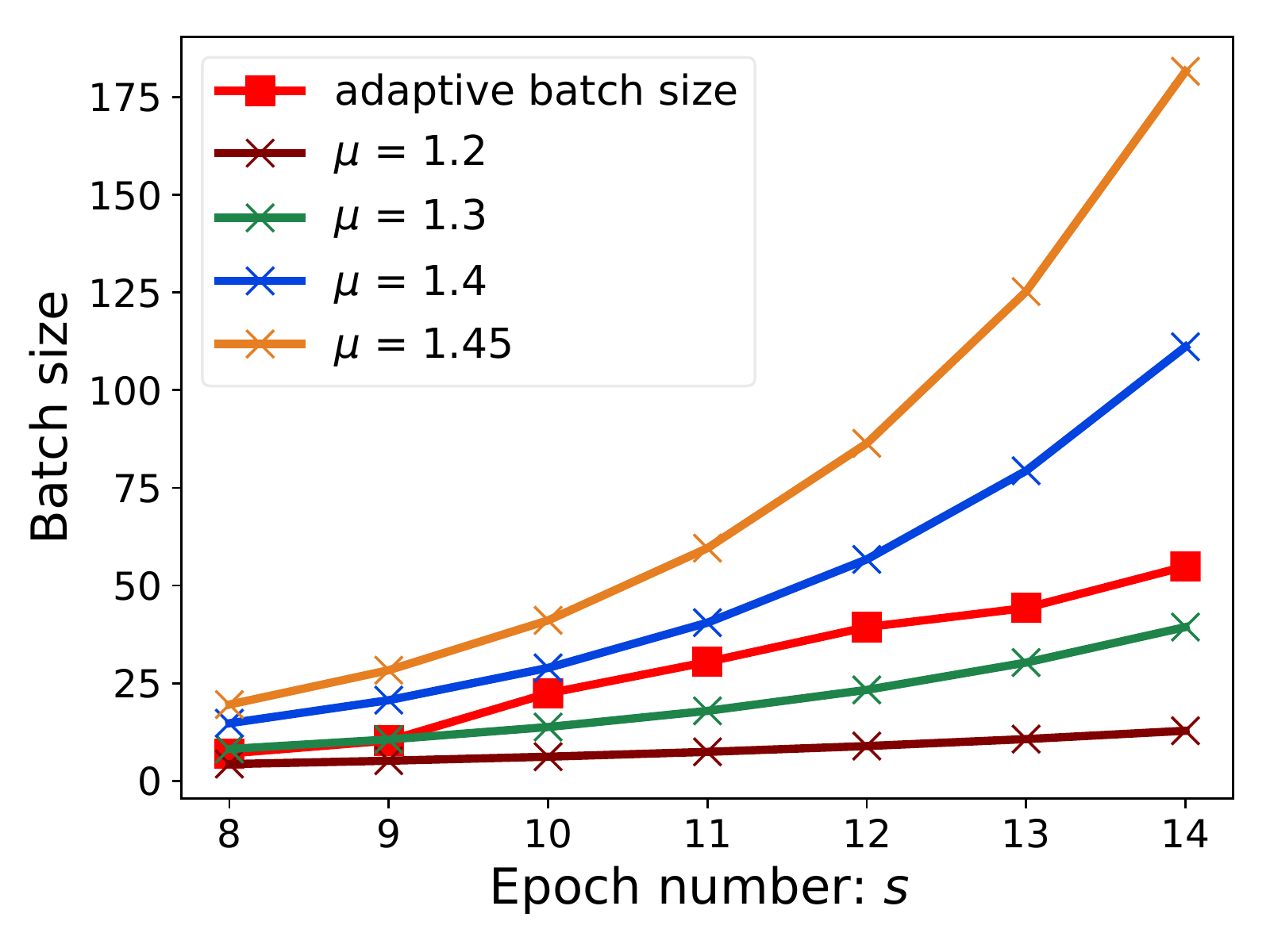}}  
	\caption{Comparison of our gradient-based
		 adaptive batch size and  exponentially and linearly increased  batch sizes. 
		For each dataset, the left figure plots  loss v.s.  $\#$ of gradient evaluations and the right figure plots adaptive batch size and exponentially increasing batch sizes v.s. epoch number $s$.
	}\label{figure:result2}
\end{figure*}
 \section{Experiments} 

In this section, we compare our proposed batch-size adaptation algorithms with their corresponding vanilla algorithms in both conventional nonconvex optimization and reinforcement learning problems.

 
 \subsection{Nonconvex Optimization}\label{diaodehen}

We compare our proposed AbaSVRG and AbaSPIDER with the state-of-the-art algorithms including mini-batch SGD \cite{ghadimi2013stochastic}, HSGD~\cite{zhou2018new}, AbaSGD\footnote{An improved SGD algorithm with batch size adapting to the history gradients.}~\cite{sievert2019improving}, SVRG+~\cite{li2018simple}, and SpiderBoost~\cite{wang2019spiderboost} for two nonconvex optimization problems, i.e., nonconvex logistic regression and training multi-layer neural networks. Due to the space limitations, the detailed hyper-parameter settings for all algorithms and the results on training neural networks are relegated to Appendix~\ref{ex:additional}.

     \begin{figure*}[h]
	\vspace{-0cm}
	\centering     
	\subfigure{\includegraphics[width=42.2mm]{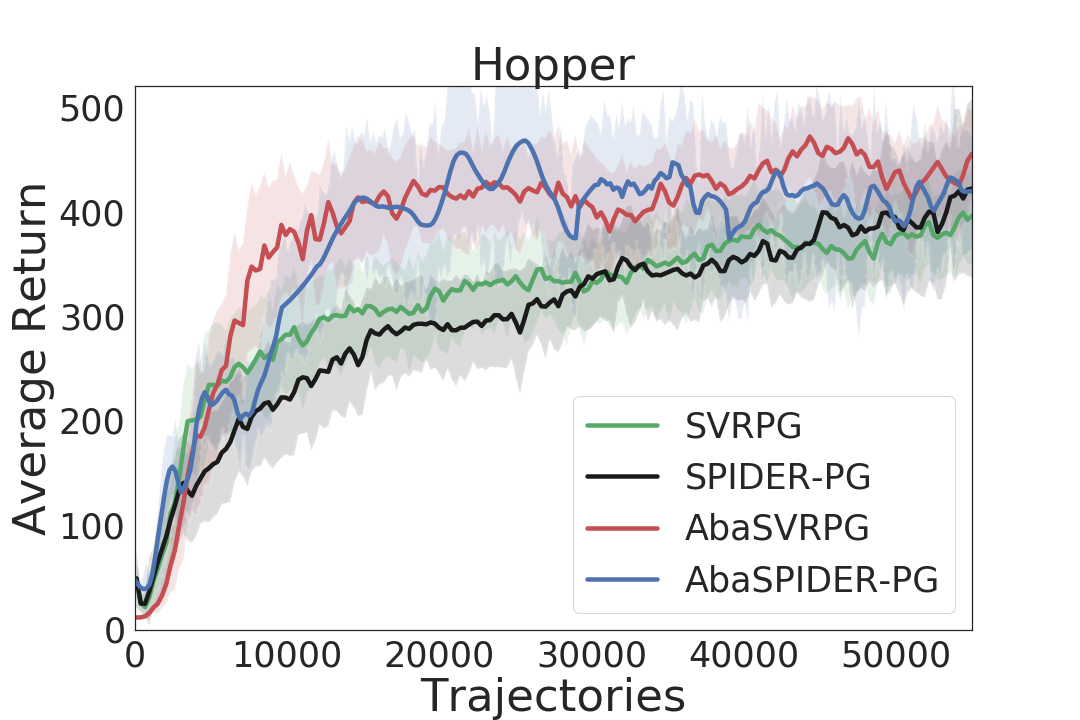}} 
	\subfigure{\includegraphics[width=42.2mm]{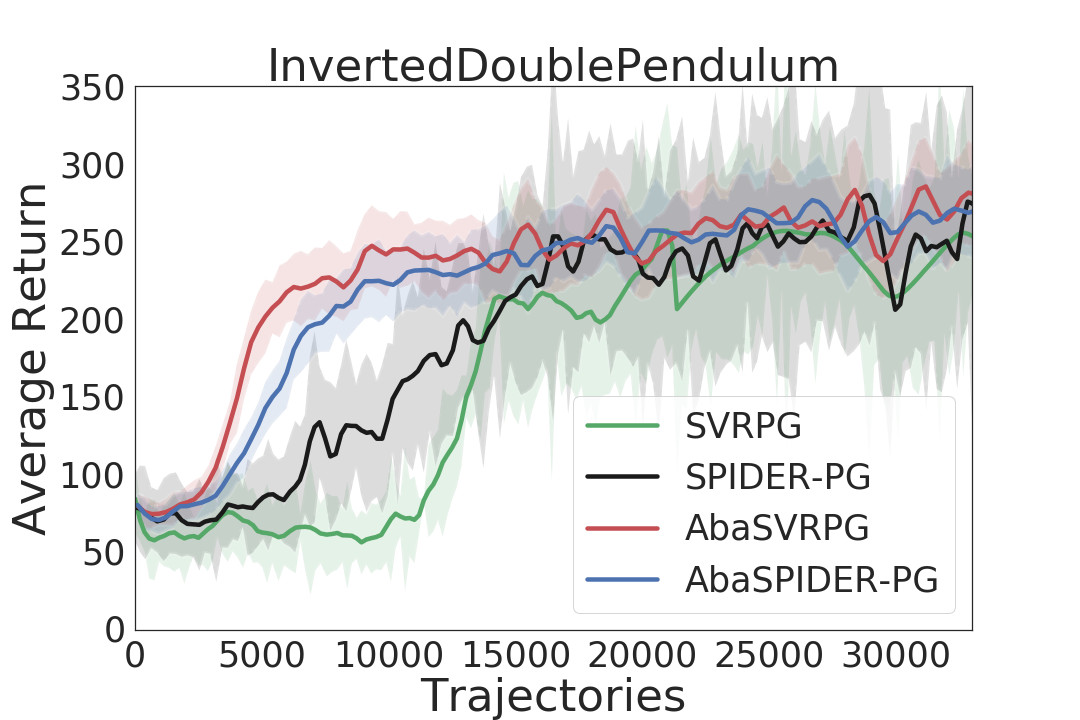}}
	\subfigure{\includegraphics[width=42.2mm]{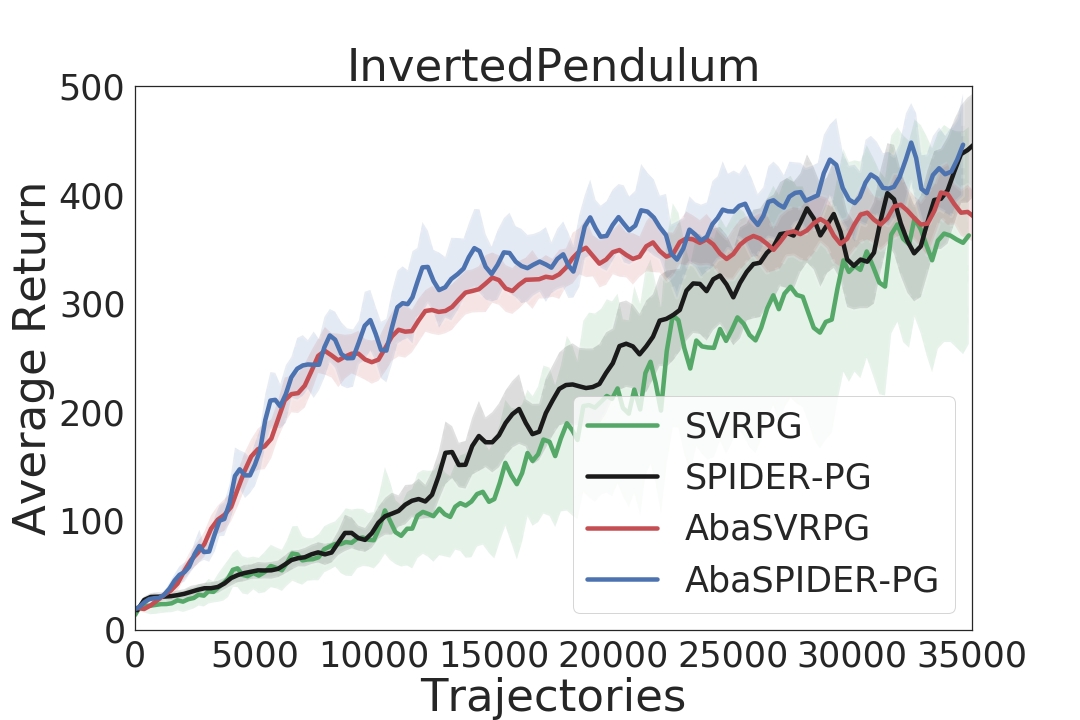}} 
	\subfigure{\includegraphics[width=42.2mm]{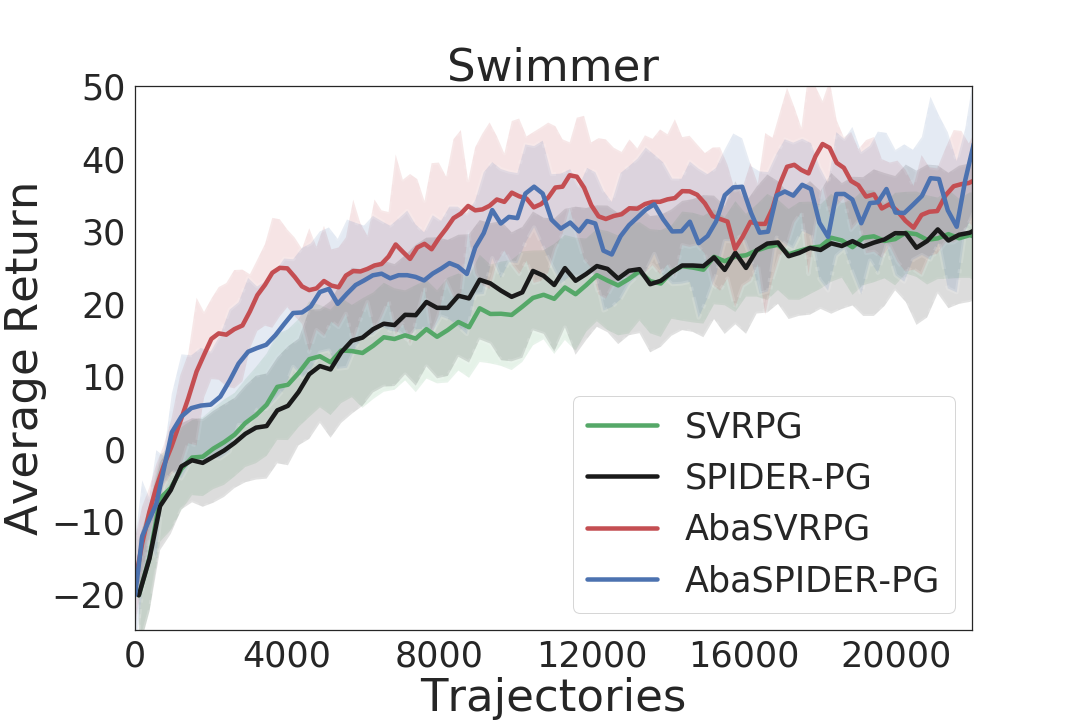}}
	\caption{Comparison of various algorithms for reinforcement learning on four tasks. }\label{figure:rl}
\end{figure*}

We consider the following nonconvex logistic regression problem with two classes
$\min_{{w}\in \mathbb{R}^d} \frac{1}{n} \sum_{i=1}^{n} \ell({w}^T {x}_i, y_i)+ \alpha \sum_{i=1}^{d} \frac{w_i^2}{1  +w_i^2 },$
where ${x}_i\in \mathbb{R}^d$ denote the features, $y_i\in \{\pm 1 \}$ are  labels,  $\ell$ is the cross-entropy loss, and we set $\alpha=0.1$. For this problem, we use four datasets from LIBSVM \cite{chang2011libsvm}: a8a, w8a, a9a, ijcnn1. 
    
As can be seen from Fig.~\ref{figure:result1} and Fig.~\ref{figure:result11} (in Appendix~\ref{apen:b2}),  AbaSVRG and AbaSPIDER  converge much faster than all other algorithms in terms of the total number of gradient evaluations (i.e., SFO complexity) on all four datasets. It can be seen that both of them take the advantage of sample-efficient SGD-like updates (due to the small batch size) at the initial stage and attain high accuracy provided by variance-reduced methods at the final stage. This is consistent with the choice of our batch-size adaptation  scheme. 

We then evaluate the performance of our history-gradient based batch-size adaptation scheme with the other two commonly used prescribed adaptation schemes, i.e., exponential increase of batch size $N_s=\mu^s$ and linear increase of batch size $N_s = \nu(s+1)$.  
 Let SpiderBoost$\_$exp$\_ \mu$ and SpiderBoost$\_$lin$\_ \nu$ denote SpiderBoost algorithms with exponentially and linearly increasing batch sizes under parameters $\mu$ and $\nu$, respectively. As shown in Fig.~\ref{figure:result2}, our adaptive batch size scheme achieves the best performance for a9a dataset, and performs better than all other algorithms for w8a dataset except SpiderBoost\_lin\_200, which, however, does not converge in the high-accuracy regime. Furthermore, the performance of prescribed batch-size adaptation can be problem specific. For example, exponential increase of batch size (with $\mu=2$ and $\mu =2.1$) performs better than linear increase of batch size for a9a dataset, but worse for w8a dataset, whereas our scheme adapts to the optimization path, and hence performs the best in both cases.

\subsection{Reinforcement Learning}

We compare our proposed AbaSVRPG and AbaSPIDER-PG with vanilla SVRPG~\cite{Papini2018} and SPIDER-PG~\cite{xu2019sample} on four benchmark tasks in reinforcement learning, i.e., InvertedPendulum, InvertedDoublePendulum, Swimmer and Hopper. We apply the Gaussian policy which is constructed using a two-layer neural network (NN) with the number of hidden weights being task-dependent. We also include the setup of adaptive standard deviation. The experimental results are averaged over $20$ trails with different random seeds, and selections of random seeds are consistent for different algorithms within each task for a fair comparison. Further details about the hyper-parameter setting and task environments are provided in Appendix~\ref{append:rl}. 

It can be seen from \Cref{figure:rl} that the proposed AbaSVRPG and AbaSPIDER-PG converge much faster than the vanilla SVRPG and SPIDER-PG (without batch size adaptation) on all four tasks. Such an acceleration is more significant at the initial stage of optimization procedure due to the large trajectory gradient that suggests small batch size. 
\section{Conclusion}
In this paper, we propose a novel scheme for adapting the batch size via history gradients, based on which we develop  AbaSVRG and AbaSPIDER for conventional optimization and AbaSVRPG and AbaSPIDER-PG for reinforcement learning. We show by theory and experiments that the proposed algorithms achieve improved computational complexity than their vanilla counterparts (without batch size adaptation). Extensive experiments demonstrate the promising performances of proposed algorithms. 
We anticipate that such a scheme can be applied to a wide range of other stochastic algorithms to accelerate their theoretical and practical performances.  

\section*{Acknowledgements}
The work was supported in part by the U.S. National Science Foundation under the grants CCF-1761506, ECCS-1818904, CCF-1900145, and CCF-1909291.


\bibliographystyle{icml2020}
\bibliography{refs}

\newpage
\onecolumn
\appendix
{\Large{\bf Supplementary Materials}} 
\section{Further Details on Policy Gradient}\label{app:pg}

This paper considers the policy gradient algorithms that can adopt the following two types of trajectory gradients, namely REINFORCE \citep{Williams1992} and G(PO)MDP \citep{Baxter2001}. We note that
 \begin{align*}
 \nabla J(\theta) &= \nabla \E_{\tau  \sim p(\cdot| \theta)}\left[ \mathcal{R}(\tau)\right] =  \E_{\tau  \sim p(\cdot| \theta)}\left[\mathcal{R}(\tau)  \nabla \log p(\tau| \theta)  \right],
 \end{align*}
where $p(\tau| \theta) = \rho (s_0) \pi_{\theta}(a_0|s_0) \prod_{i=0}^{H-1} \mathcal{P}(s_{i+1}|s_{i},a_i) \pi_{\theta}(a_{i+1}|s_{i+1})$. 
REINFORCE constructs the trajectory gradient as
 \begin{small} 
 \begin{align*}
g(\tau | \theta) =  \underbrace{\Big(\sum_{t=0}^{H} \gamma^t \mathcal{R}(s_t,a_t) - b(s_t,a_t) \Big)  \Big(\sum_{t=0}^{H} \nabla \log  \pi_{\theta}(a_t|s_t) \Big)}_{ \mathcal{R}(\tau) \nabla \log p(\tau| \theta) },
 \end{align*}
 \end{small}
 \hspace{-0.13cm}where $b: \mathcal{S} \times \mathcal{A} \rightarrow \mathcal{R}$ is a bias. G(PO)MDP  enhances the  trajectory gradient of REINFORCE by  further utilizing  the fact that the reward at time $t$ does not depend  on the action implemented after time $t$. Thus, G(PO)MDP constructs the trajectory gradient as
 \begin{align*}
g(\tau | \theta) =    \sum_{t=0}^{H} \left( \gamma^t \mathcal{R}(s_t,a_t) - b(s_t,a_t)\right)    \sum_{i=0}^{t} \nabla \log  \pi_{\theta}(a_i|s_i) . 
 \end{align*}
 Note that REINFORCE and G(PO)MDP are both unbiased gradient estimators, i.e., $\E_{\tau \sim p(\cdot|\theta) } [g(\tau|\theta)] = \nabla J(\theta)$. 
 
%

\vspace{-0.1cm}
\section{Further Specification of Experiments and Additional Results}\label{ex:additional}
\subsection{ Hyper-parameter Configuration of Algorithms for Nonconvex Optimization}\label{apen:b1} 
To implement HSGD, we follow~\citealt{zhou2018new} and  choose the linearly increasing mini-batch size at the $t^{th}$ iteration to be $c_b(t+1)$, and tune $c_b$ to the  best. We set the epoch length $m=10$ for all variance-reduced algorithms, because $m=10$ works best for all variance-reduced algorithms for a fair comparison. 
$\epsilon$ is the target accuracy predetermined by users, typically dependent on specific applications. 
Specifically, we choose $\epsilon = 1e^{-3}$ for the logistic regression and $\epsilon = 1e^{-2}$  for the neural network training, respectively.
We choose the batch size to be $\min\{n,c_1\epsilon^{-1}\}$ for SVRG+ and SpiderBoost, and $\min\{n,c_1\epsilon^{-1},c_2\beta_s^{-1}\}$ for AbaSPIDER and AbaSVRG,  where $\beta_s= \frac{1}{m}\sum_{t=1}^m \|{v}_{t-1}^{s-1}\|^2$ as given in Subsection~\ref{sec:algorithm}.

\subsection{Additional Results for Nonconvex Logistic Regression}\label{apen:b2}
For logistic regression, we use four datasets: a8a  ($n=22696, d=123$),  a9a ($n=32561, d=123$), w8a  ($n=43793, d=300$) and ijcnn1 ($n=49990, d=22$).
 We select the stepsize $\eta$ from $\{0.1k, k=1,2,...,15 \}$ and the mini-batch size $B$ from $\{10, 28, 64, 128, 256, 512, 1024\}$ for all algorithms, and we present the best performance among these parameters. For all variance-reduced algorithms, we select constants $c_1$ and $c_2$ from $\{1,2,3,...,10\}$, and present the best performance among these parameters. 
 For HSGD algorithm, we select $c_b$ in its linearly increasing batch size $c_b(t+1)$ from $\{1, 5, 10, 40, 100, 400, 1000\}$, and present the best performance among these parameters. For AbaSGD, we set its batch size as $\min \Big\{ \frac{c_\beta}{\sum_{i=1}^{5}\|{v}_{t-i}\|^2/5},\frac{ c_\epsilon}{\epsilon} , n\Big\} $, and select the best $c_\beta$ and $c_\epsilon$ from $\{1,2,3,...,10\}$. 
  \begin{figure*}[h]
	\centering     
	\subfigure[dataset: a8a ]{\includegraphics[width=42mm]{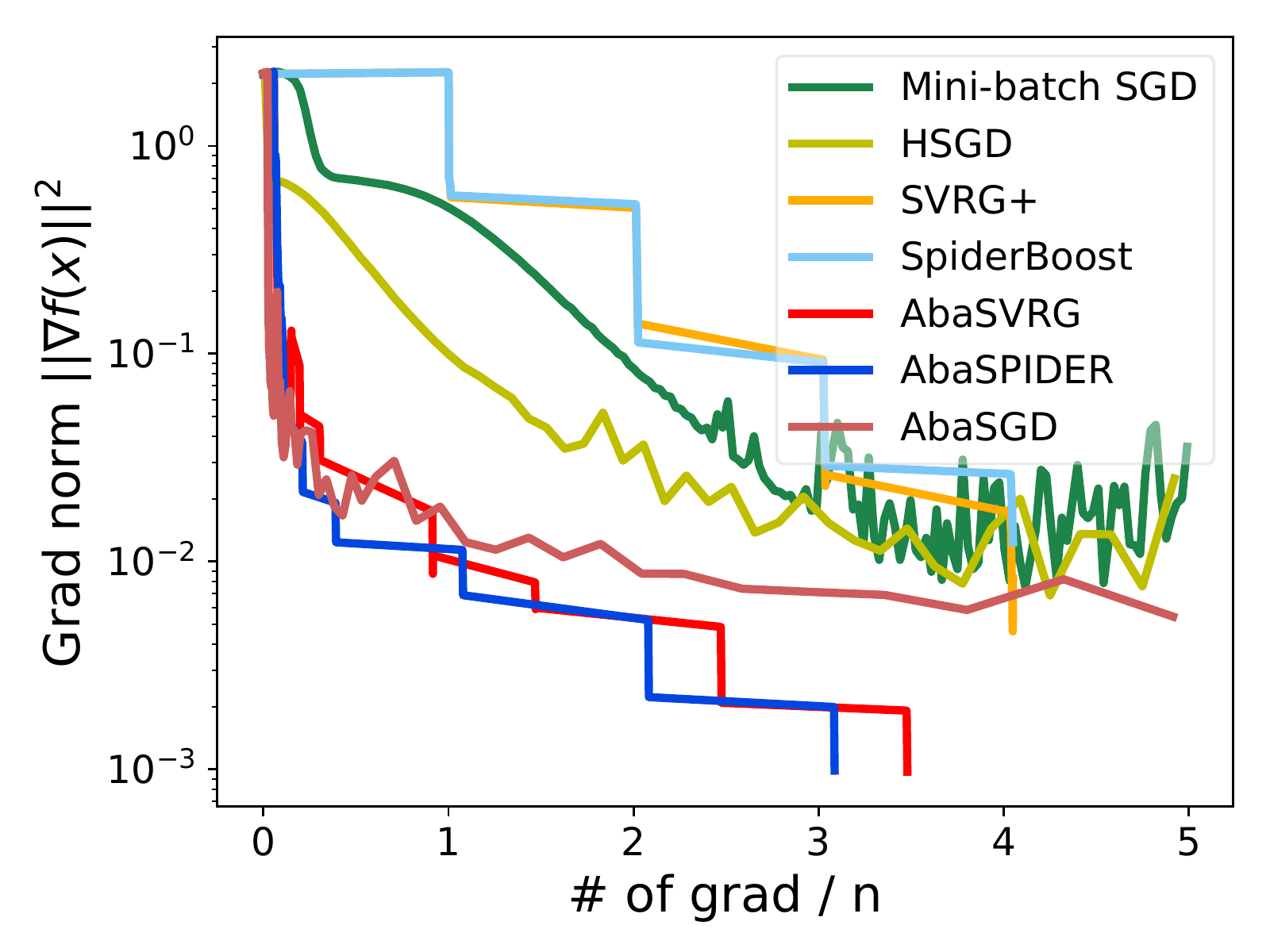}} 
	\subfigure[dataset: ijcnn1 ]{\includegraphics[width=42mm]{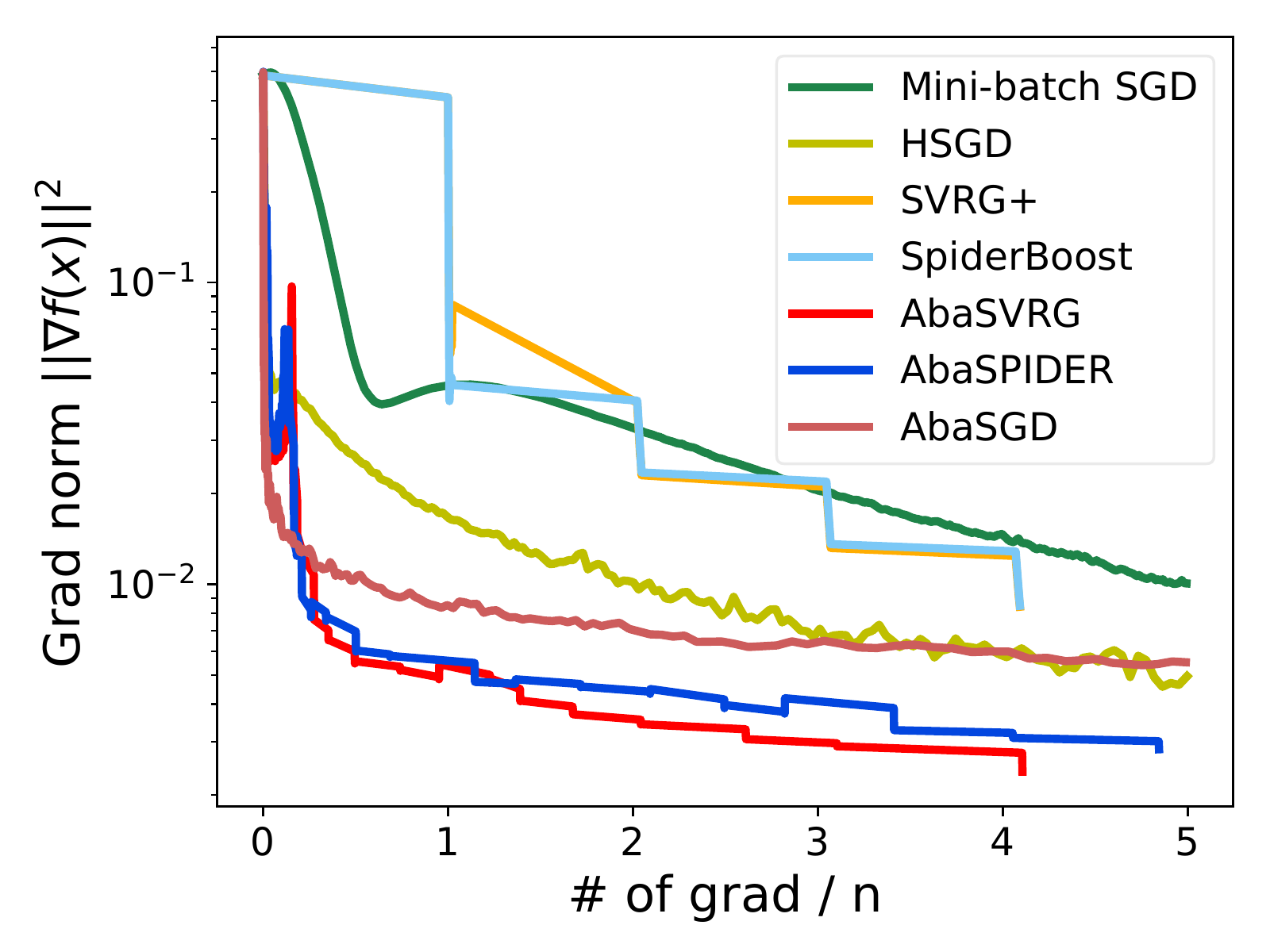}} 
	\subfigure[dataset: a9a ]{\includegraphics[width=42mm]{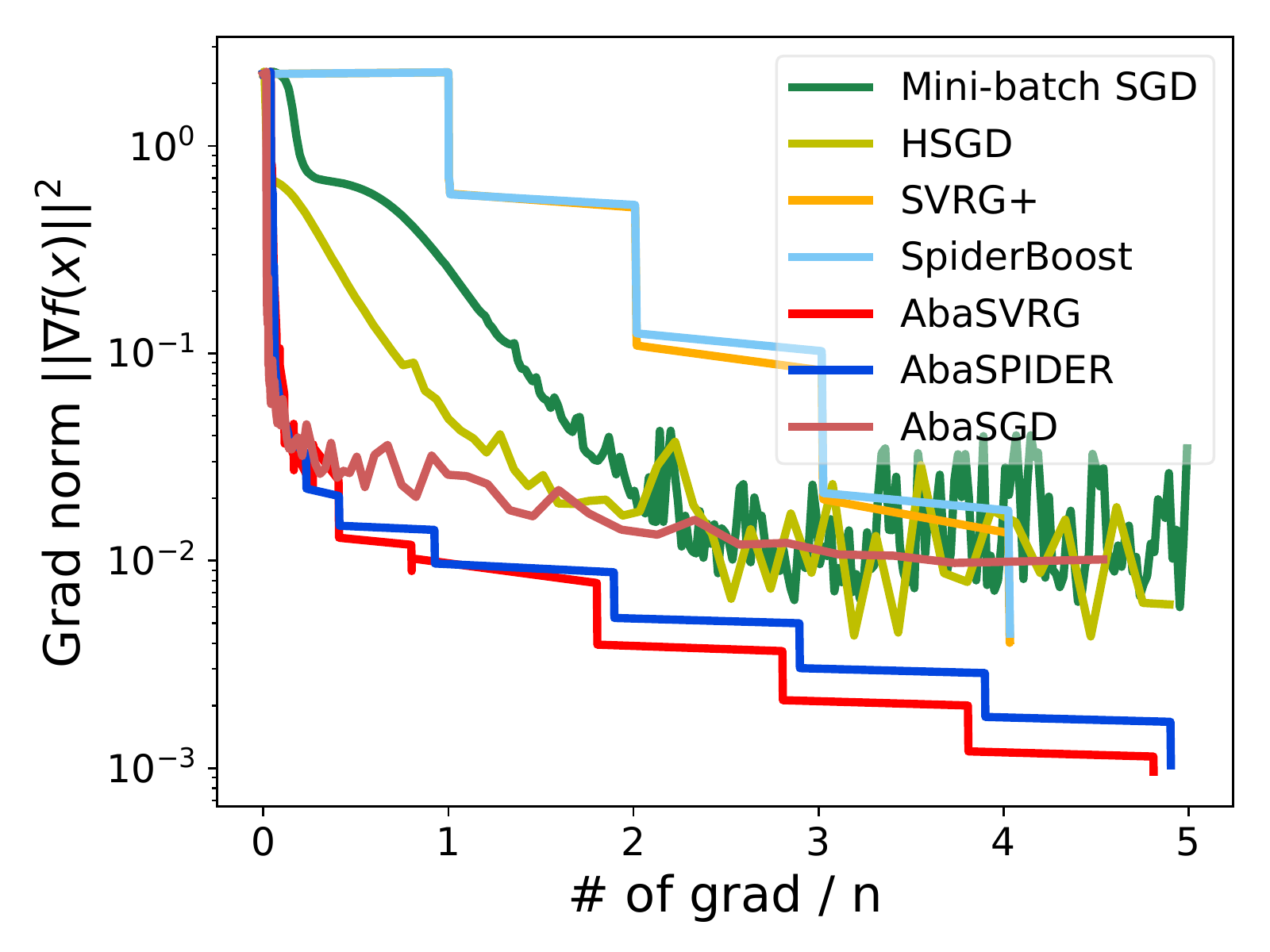}}  
	\subfigure[dataset: w8a ]{\includegraphics[width=42mm]{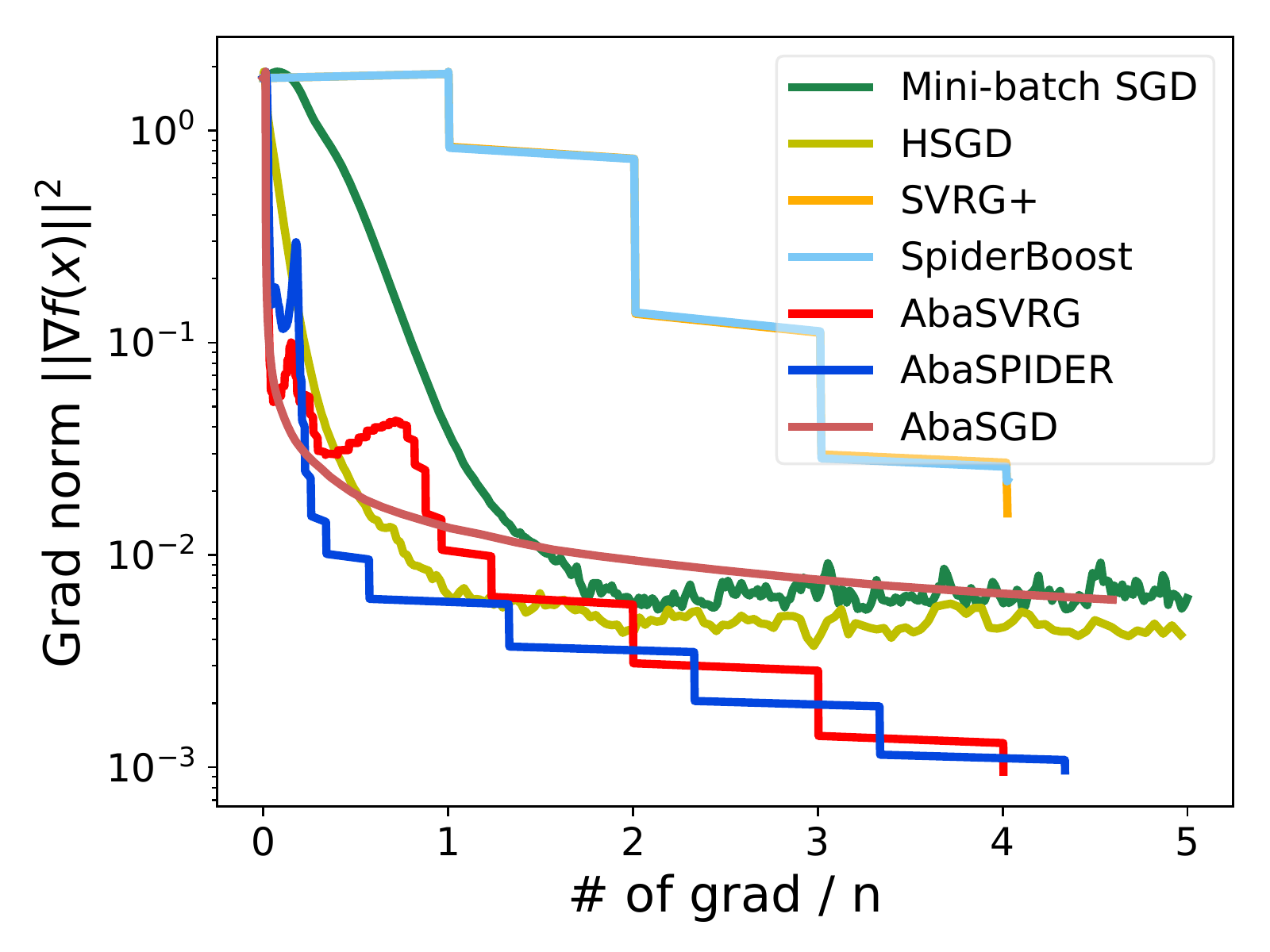}} 
	\vspace{-0.2cm}
	\caption{Comparison of different algorithms for  logistic regression problem on four datasets. All figures plot gradient norm v.s. $\#$ of gradient evaluations.}\label{figure:result11}
\end{figure*}

As shown in Fig.~\ref{figure:result11},  AbaSVRG and AbaSPIDER  converge much faster than all other algorithms in terms of the total number of gradient evaluations  on all four datasets. It can be seen that both of them take the advantage of sample-efficient SGD-like updates (due to the small batch size) at the initial stage and attain high accuracy provided by variance-reduced methods at the final stage. This is consistent with the choice of our batch-size adaptation  scheme. 


\subsection{Results for Training Multi-Layer Neural Networks}\label{appen:trainNeural}
In this subsection, we compare our proposed algorithms with other competitive algorithms as specified in \Cref{diaodehen} for training a three-layer ReLU neural network with a cross entropy loss on the dataset of MNIST ($n=60000, d=780$). The neural network has a size of $(d_{\text{in}},100,100,d_{\text{out}})$, where $d_{\text{in}}$ and $d_{\text{out}}$ are the input and output dimensions and $100$ is the number of neurons in the two hidden layers. 
We select the stepsize $\eta$ from $\{10^{-4} k, k=1,2,...,15 \}$ and the mini-batch size $B$ from $\{64, 96, 128, 256, 512\}$ for all algorithms, and we present the best performance among these parameters. For all variance-reduced algorithms, we set $c_1=1$ and select the best $c_2$ from $\{10^3,5 \times 10^3,10^4 \}$. For HSGD algorithm, we select $c_b$  from $\{1, 10, 50, 100, 500, 1000\}$, and present the best performance among these parameters. 
 \vspace{-0.1cm}
   \begin{figure*}[h]
	\centering     
	\subfigure{\includegraphics[width=48mm]{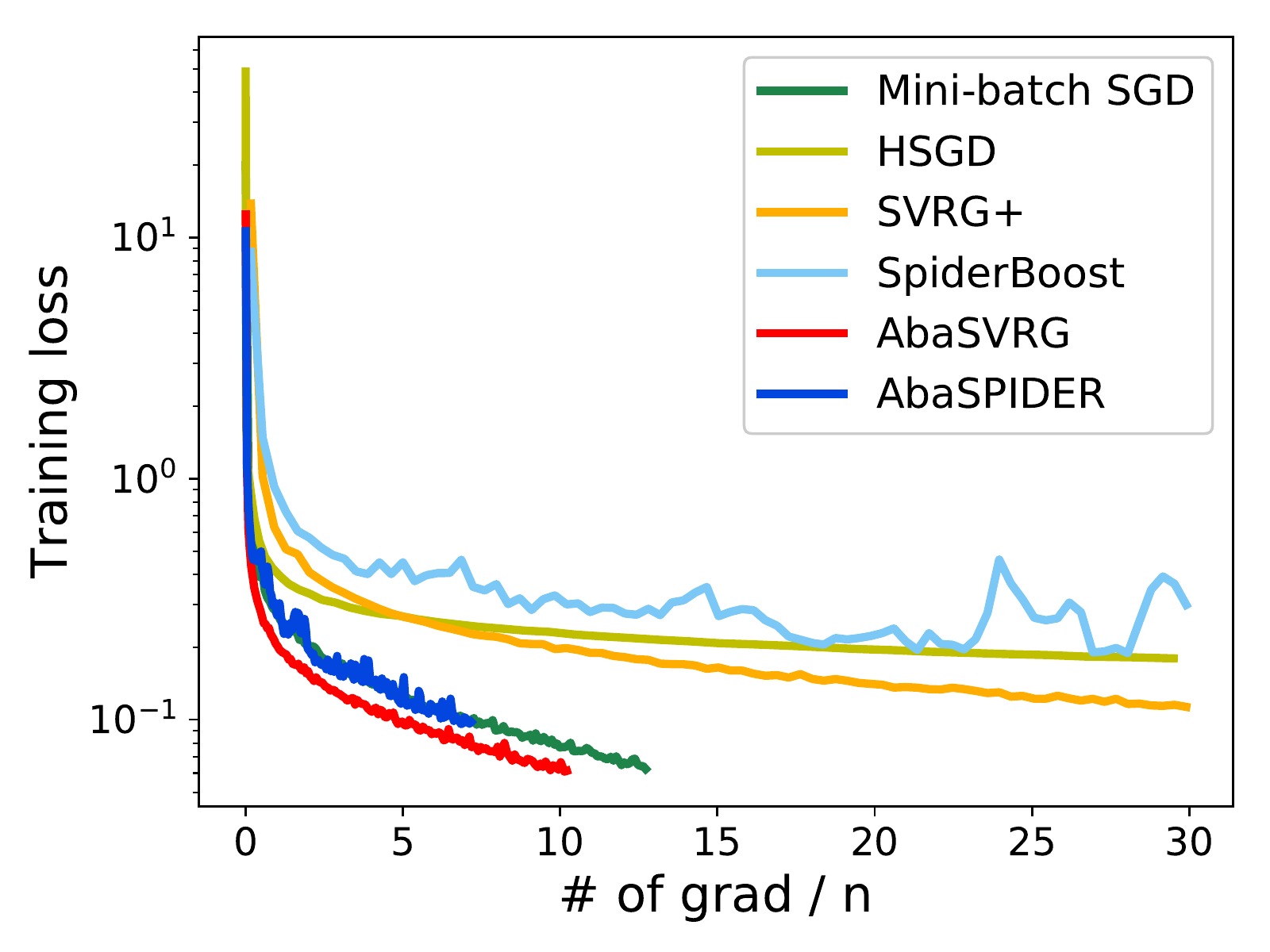}} 
        \subfigure{\includegraphics[width=48mm]{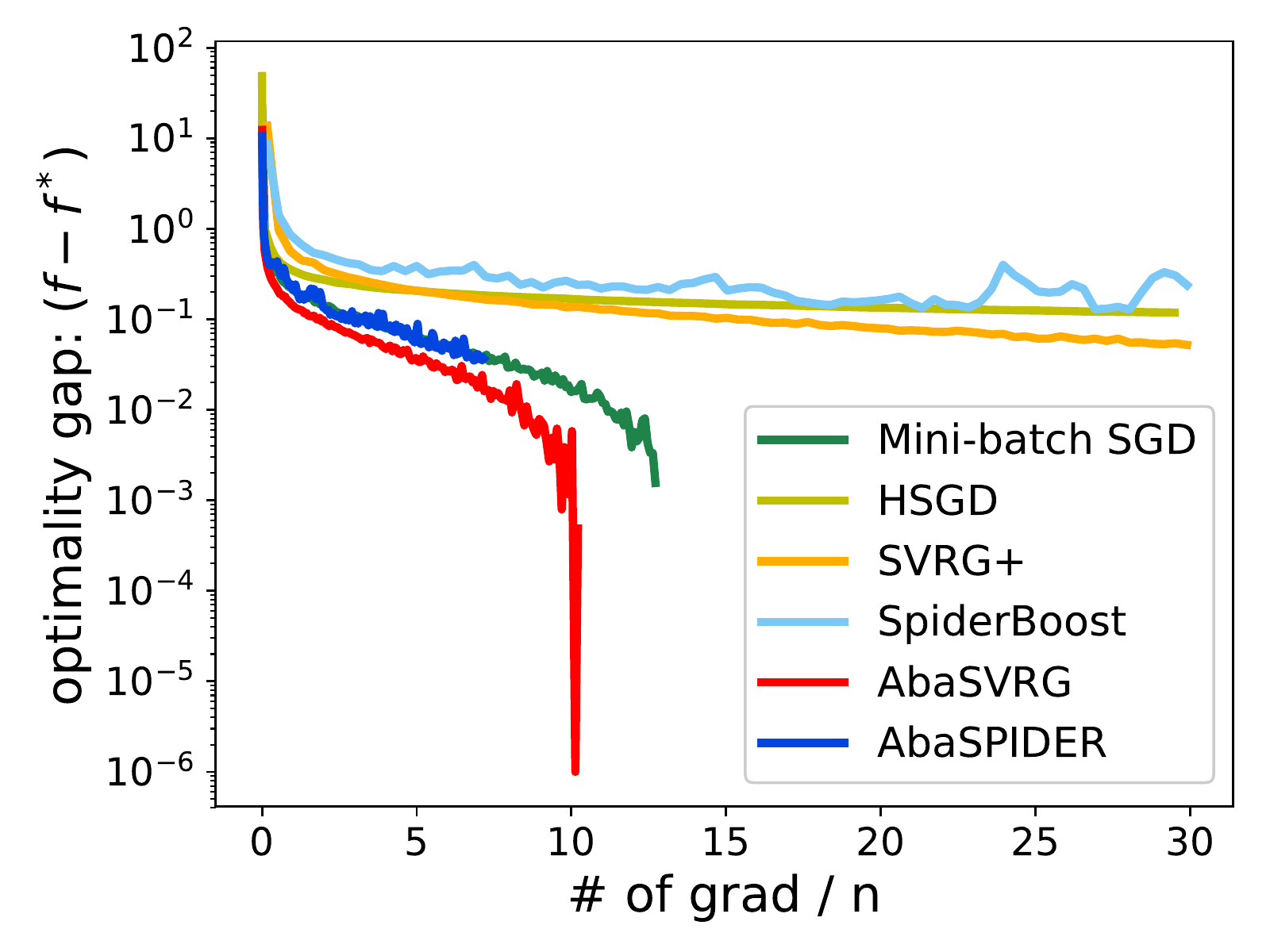}}
        \subfigure{\includegraphics[width=48mm]{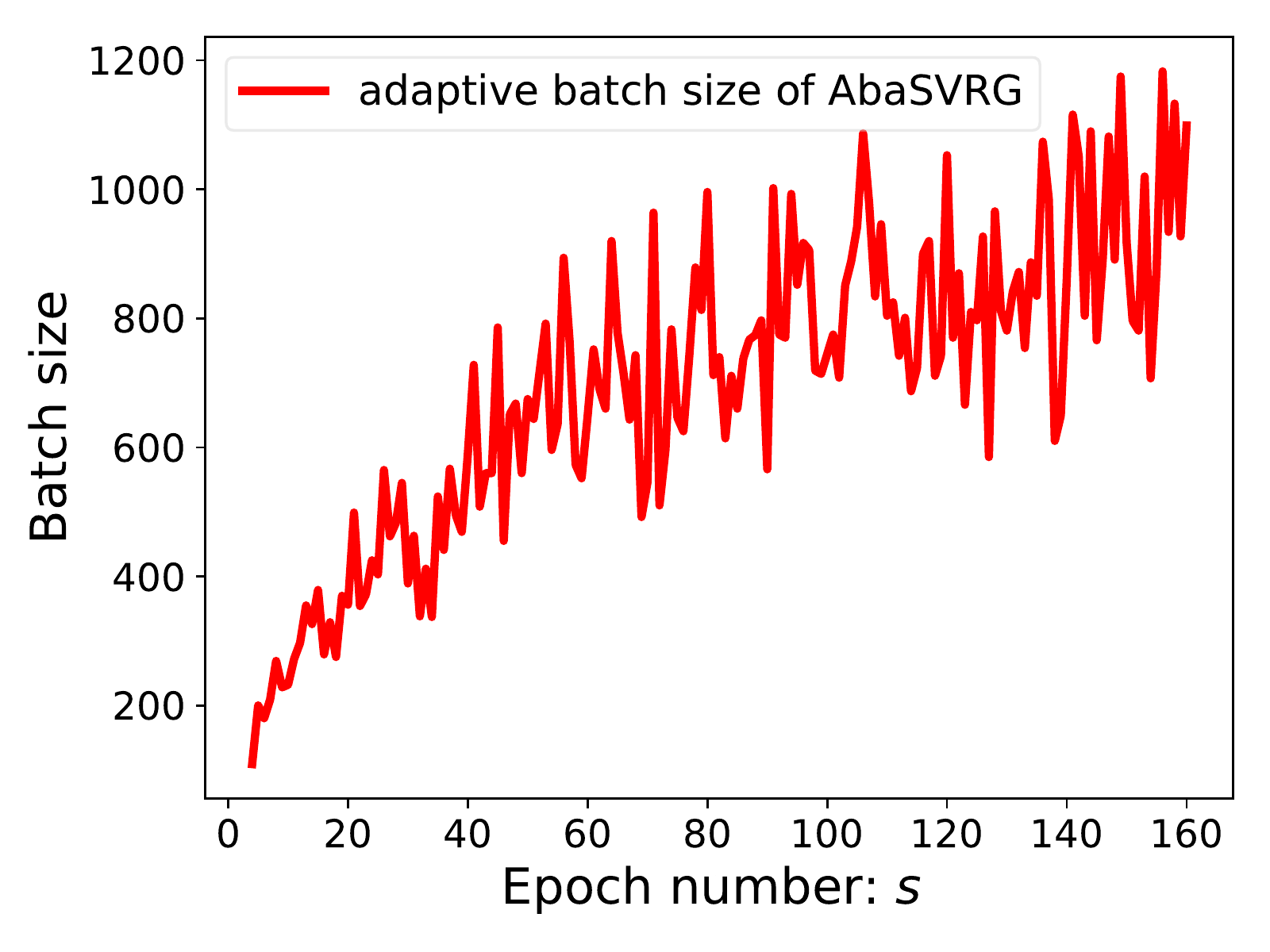}} 
	\vspace{-0.1cm}
	\caption{Comparison of various algorithms for  training a three-layer neural network on MNIST.}\label{figure:result3}
	 \vspace{-0.1cm}
\end{figure*}

 As shown in Fig.~\ref{figure:result3}, our AbaSVRG achieves the best performance among all competing algorithms, and AbaSPIDER performs similarly to mini-batch SGD for decreasing training loss, but converges faster in terms of gradient norm. Interestingly, the batch-size adaptation used by AbaSVRG increases the batch size slower than both exponential and linear increase of the batch size, and its scaling is close to the \textit{logarithmical} increase as shown in the right-most plot in \Cref{figure:result3}. Such an observation further demonstrates that our gradient-based batch-size adaptation scheme can also adapt to the neural network landscape with a differently (i.e., more slowly) increased batch size from that for nonconvex regression problem over a9a and w8a datasets.      

\subsection{Experimental Details for Reinforcement Learning}\label{append:rl}


 The hyper-parameters listed in Table~\ref{tab:rl experiment param} are the same among all methods on
each task. For the proposed AbaSVRPG and AbaSPIDER-PG, we adopt the same hyper-parameter of $\alpha\sigma^2=1$ and $ \beta=1000$ in all experiments.

\begin{table}[H]
	\centering
	\small
	\caption{Parameters used in the RL experiments}
	\label{tab:rl experiment param}
	\begin{tabular}{|l||l|l|l|l|}
		\hline
		Task                         & InvertedPendulum & InvertedDoublePendulum     & Swimmer           & Hopper            \\ \hline
		Horizon                      & 500              & 500                        & 500               & 500               \\ \hline
		Discount Factor $\gamma$     & 0.99             & 0.99                       & 0.99              & 0.99              \\ \hline
		$q$                          & 10               & 10                         & 10                & 10                \\ \hline
		$N$                          & 100              & 100                        & 50                & 50                \\ \hline
		$B$                          & 20               & 20                         & 20                & 20                 \\ \hline
		$\epsilon$                & 0.01            & 0.01                    & 0.01               & 0.01     \\ \hline
		Step Size                    & 0.001            & 0.001                      & 0.0001            & 0.001            \\ \hline
		NN Hidden Weights            & $16 \times 16$   & $16 \times 16$             & $32 \times 32$    & $64 \times 64$    \\ \hline
		NN Activation                & tanh             & tanh                       & tanh              & tanh              \\ \hline
		Baseline                     & No               & No                         & Yes               & Yes               \\ \hline
	\end{tabular}
\end{table}

  \begin{figure*}[h]
	\centering     
	\subfigure[Hopper ]{\includegraphics[width=30.5mm]{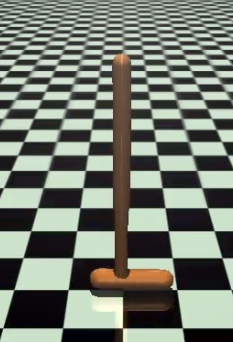}} 
	\subfigure[Inverteddoublependulum]{\includegraphics[width=33mm]{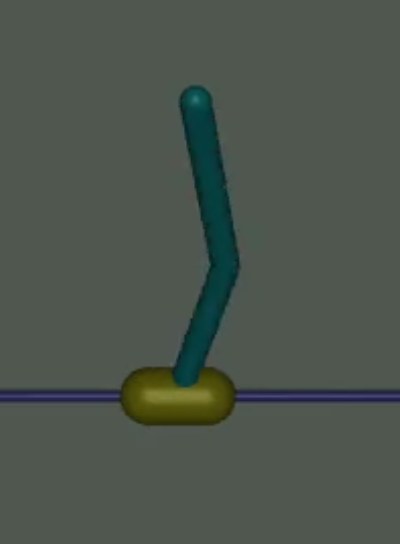}} 
	\subfigure[Invertedpendulum]{\includegraphics[width=38mm]{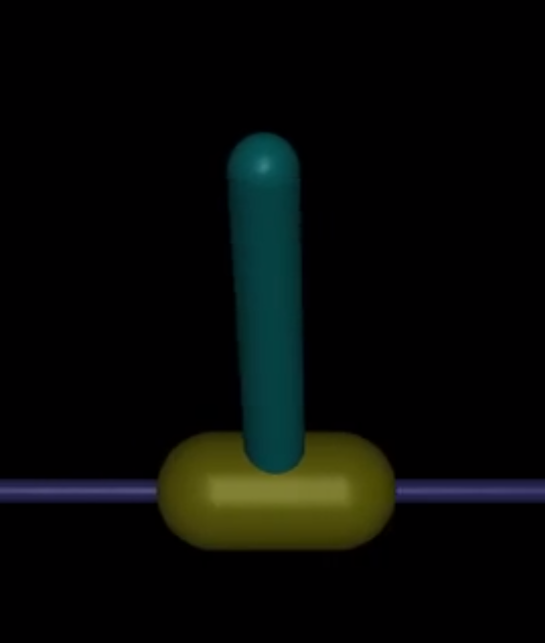}}  
	\subfigure[Swimmer]{\includegraphics[width=38mm]{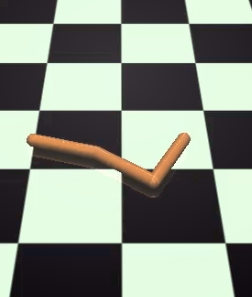}} 
	\vspace{-0.2cm}
	\caption{Task Environments.}
	\label{fig:envs}
\end{figure*}

Figure \ref{fig:envs} illustrates all task environments. The problem setup regarding each task is summarized as follows:
\begin{enumerate}
	\item \textit{InvertedPendulum}: A cart is moving along a track with zero friction and a pole is attached through an un-actuated joint. The pendulum is balanced by controlling the velocity of the cart. The action space is continuous with $a \in [-1,1]$ (with $-1$ for pushing cart to the left and $1$ for pushing cart to the right). For a single episode with time step $h$ enumerated from 1 to 500, the episode is terminated when the pole angle $\theta_h > 0.2 rad$, and otherwise a reward of value $1$ is awarded.
	
	\item \textit{InvertedDoublePendulum}: The setup of this task is similar to that at the InvertedPendulum. The only difference is that another pendulum is added to the end of the previous pendulum through an unactuated rotational joint.
	
	\item \textit{Swimmer}: The agent is a 3-link robot defined in Mujoco with the state-space dimension of 13. It is actuated by two joints to swim in a viscous fluid. For a single episode with time step $h$ enumerated from 1 to 500, the reward function encourages the agent to move forward as fast as possible while maintaining energy efficiency. That is, given forward velocity $v_x$ and joint action $a$, $r(v_x, a) = v_x^2 - 10^{-4} \|a\|_2^2$.
	
	\item \textit{Hopper}: A two-dimensional single-legged robot is trained to hop forward. The system has the state-space dimension of $11$ and action space dimension of 3. For a single episode with time step $h$ enumerated from 1 to 600, we have forward velocity $v_x$ and commanded action $a$. The episode terminates early (before $h$ reaches $500$) when the tilting angle of upper body or the height position for center of mass drops below a certain preset threshold. The reward function encourages the agent to move forward as fast as possible in an energy efficient manner. It also gets one alive bonus for every step it survives without triggering any of the termination threshold.
\end{enumerate}

For the tasks of Swimmer and Hopper, we also include the linear baseline for value function approximation~\cite{duan2016benchmarking}.

\section{Convergence of AbaSVRG and AbaSPIDER  under Local PL Geometry}\label{se:5}
Many nonconvex machine learning problems (e.g., phase retrieval~\cite{zhou2016geometrical}) and deep learning (e.g., neural networks~\cite{zhong2017recovery,zhou2017characterization}) problems have been shown to satisfy the following Polyak-{\L}ojasiewicz  (PL) (i.e., gradient dominance) condition in  local regions  near local or  global minimizers.  
\begin{Definition}[\cite{polyak1963gradient,nesterov2006cubic}]\label{d1}
	Let ${x}^*=\arg\min_{{x}\in\mathbb{R}^d}f({x})$. Then, the function $f$ is said to be $\tau$-gradient dominated if for any ${x}\in\mathbb{R}^d$, 
	$f({x})-f({x}^*)\leq \tau \|\nabla f({x})\|^2.$
\end{Definition}
In this section, we  explore whether our proposed AbaSVRG and AbaSPIDER with batch size adaptation can attain  a faster linear convergence rate if the iterate enters the local PL regions.   All the proofs are provided in Appendix~\ref{youdiandou}.  
\vspace{-0.1cm}
\subsection{AbaSVRG: Convergence  under PL Geometry without Restart}
\vspace{-0.1cm}
The following theorem provides the convergence and complexity  for AbaSVRG under the PL condition. 
\begin{Theorem}\label{svrg:pl}
	Let $\eta =\frac{1}{c_\eta L}$, $B=m^2$ with $\frac{8L\tau}{c_\eta -2}\leq m <4L\tau$, $\beta_1 \leq \epsilon\big(\frac{1}{\gamma}\big)^{m(S-1)} $, and $c_\beta=c_\epsilon= \Big(2\tau+ \frac{2\tau}{1-\exp(\frac{-4}{c_\eta(c_\eta -2)})}\Big)\vee \frac{16c_\eta L\tau }{m} $, where constants $c_\eta>4$ and $\gamma = 1-\frac{1}{8L\tau}<1$. Then under the PL condition, the final iterate ${\tilde x}^S$ of AbaSVRG satisfies 
	$$\mathbb{E}(f({\tilde x}^S) -f({x}^*) ) \leq \gamma^{K}(f({x}_{0}) - f({x}^*))+\frac{\epsilon}{2}.$$ 
	To obtain an $\epsilon$-accurate solution ${\tilde x}^S$,  the total number of SFO calls is given by 
	\begin{align}\label{eq:smin}
	\sum_{s=1}^S\min\left\{ \frac{c_\beta\sigma^2 }{\sum_{t=1}^m \|{v}_{t-1}^{s-1}\|^2/m}, c_\epsilon \sigma^2\epsilon^{-1},n\right\} + KB \leq  \mathcal{O}\Big( \left(\frac{\tau}{\epsilon}\wedge n\right) \log \frac{1}{\epsilon} + \tau^{3} \log \frac{1}{\epsilon}\Big).
	\end{align}
\end{Theorem}
Our proof of Theorem~\ref{svrg:pl} is different from and more challenging than the previous techniques developed  in~\citealt{reddi2016stochastic,Reddi2016,li2018simple} for SVRG-type algorithms, because we need to handle the adaptive batch size $N_s$ with the dependencies on the  iterations at the previous epoch. In addition, we do not need \textit{extra} assumptions for proving the convergence under PL condition, whereas~\citealt{Reddi2016} and~\citealt{li2018simple} require $\tau \geq n^{1/3}$ and $\tau\geq n^{1/2}$, respectively. As a result, Theorem~\ref{svrg:pl} can be applied to any condition number regime.  For the small condition number regime where $1\leq \tau\leq\Theta(n^{1/3})$, the worst-case complexity of AbaSVRG outperforms the result achieved by SVRG~\cite{Reddi2016}.
Furthermore, the actual complexity of our AbaSVRG can be much lower than the worst-case complexity due to the  adaptive batch size. 


\subsection{AbaSPIDER: Convergence  under PL Geometry without Restart}
The following theorem shows that AbaSPIDER  achieves a linear convergence rate under the PL condition without restart. 
Our analysis can be of independent interest for other SPIDER-type methods.
\begin{Theorem}\label{spider_pl}
	Let $\eta =\frac{1}{c_\eta L}$, $B=m$ with $\frac{8L\tau}{c_\eta -2}\leq m <4L\tau$, $\beta_1 \leq \epsilon\big(\frac{1}{\gamma}\big)^{m(S-1)} $, and $c_\beta=c_\epsilon=\Big(2\tau+ \frac{2\tau}{1-\exp(\frac{-4}{c_\eta(c_\eta -2)})}\Big)\vee \frac{16c_\eta L\tau}{m} $, where  constants $c_\eta>4$ and $\gamma = 1-\frac{1}{8L\tau}$. Then under the PL condition, the final iterate ${\tilde x}^S$ of AbaSPIDER satisfies 
	$$\mathbb{E}(f({\tilde x}^S) -f({x}^*) ) \leq \gamma^{K}\mathbb{E}(f({x}_{0}) - f({x}^*))  +\frac{\epsilon}{2}.$$
	To obtain an $\epsilon$-accurate solution ${\tilde x}^S$,  the total number of SFO calls is given by 
	\begin{align*}
	\sum_{s=1}^S\min\left\{ \frac{c_\beta\sigma^2 }{\sum_{t=1}^m \|{v}_{t-1}^{s-1}\|^2/m} , c_\epsilon \sigma^2\epsilon^{-1},n\right\} + KB \leq\mathcal{O}\left( \left(\frac{\tau}{\epsilon}\wedge n\right) \log \frac{1}{\epsilon} + \tau^{2} \log \frac{1}{\epsilon}\right).
	\end{align*}
\end{Theorem}
As shown in Theorem~\ref{spider_pl}, AbaSPIDER achieves a lower worst-case SFO complexity than AbaSVRG by a factor of $\Theta(\tau)$, and matches the best result provided by Prox-SpiderBoost-gd~\cite{wang2019spiderboost}. 
However, Prox-SpiderBoost-gd is a variant of Prox-SpiderBoost with algorithmic modification, and has not been shown to achieve the near-optimal complexity for general nonconvex optimization. In addition, AbaSPIDER has a much lower complexity in practice due to the adaptive batch size.

\section{An analysis for  SGD with Adaptive Mini-Batch Size}\label{apen:sgd}
Recently, \citealt{sievert2019improving} proposed an improved SGD algorithm by adapting the batch size to the gradient norms in preceding steps. However, they do not show performance guarantee for their proposed algorithm. In this section, we aim to fill this gap by providing an analysis for adaptive batch size SGD (AbaSGD) with mini-batch size depending on the stochastic gradients in the preceding $m$ steps.
as shown in Algorithm~\ref{alg:AbaSGD}. 
To simplify notations, we set norms of the stochastic gradients before the algorithm starts to be $\|\mathbf{v}_{-1}\| = \|\mathbf{v}_{-2}\| = \cdots  =  \|\mathbf{v}_{-m}\|=\alpha_0$ and let $\mathbb{E}_t (\cdot)= \mathbb{E}(\cdot \, |\, \mathbf{x}_0,...,\mathbf{x}_t )$.

\begin{algorithm}[H]
	\caption{AbaSGD}
	\label{alg:AbaSGD}
	\small
	\begin{algorithmic}[1]
		\STATE {\bfseries Input:} $\mathbf{x}_0$,  stepsize $\eta$,  $m >0$,  $\alpha_0>0$. 
		\FOR{$t=0, 1, ..., T$} 
		\STATE {Set  $\color{blue}{|B_t| =\min \left\{ \frac{2\sigma^2}{\sum_{i=1}^{m}\|\mathbf{v}_{t-i}\|^2/m},\frac{ 24\sigma^2}{\epsilon} , n\right\}     }.$  }
		\IF{$|B_t| = n$ }  
		\STATE{Compute  $\mathbf{v}_{t}  = \nabla f (\mathbf{x}_t)$}
		\ELSE \STATE{Sample $B_t$ from $[n]$ with  replacement,  and compute $\mathbf{v}_{t}  = \nabla f_{B_t} (\mathbf{x}_t)$}    
	    \ENDIF  
		\STATE{$\mathbf{x}_{t+1} =\mathbf{x}_{t} -\eta\mathbf{v}_{t} $} 
		\ENDFOR
		\STATE {\bfseries Output:} choose $\mathbf{x}_{\zeta}$ from $\{ \mathbf{x}_{i} \}_{i=0,...,T} $ uniformly at random
	\end{algorithmic}
\end{algorithm}
\begin{Theorem}\label{th_abasgd}
	Let Assumption~\ref{assum1} hold, $\epsilon>0$ and choose a stepsize $\eta$ such that  
	\begin{align*}
		\phi = \eta-\frac{L\eta^2}{2} > 0. 
	\end{align*}
	Then, the output $\mathbf{x}_\zeta$ returned by AbaSGD satisfies 
	\begin{align*}
	\mathbb{E}\|\nabla f(\mathbf{x}_\zeta)\|^2   \leq  \frac{ 2\left( f(\mathbf{x}_{0 } )	 -   f^*  \right)   +  {\eta m}   \alpha_0^2}{2T\phi} + \frac{ \eta}{12 \phi}\epsilon, 
	\end{align*} 
	where $f^*=\inf_{\mathbf{x}\in\mathbb{R}^d} f(\mathbf{x})$, and $T$ is the total number of iterations.
\end{Theorem} 
 Theorem~\ref{th_abasgd} shows that AbaSGD achieves a $\mathcal{O}\big(\frac{1}{T}\big)$ convergence rate for nonconvex optimization by using the adaptive mini-batch size. In the following corollary, we derive the SFO complexity of AbaSGD. 
\begin{corollary}\label{co:adasgd}
	Under the setting of Theorem~\ref{th_abasgd}, we choose  
	the constant stepsize $\eta = \frac{1}{2L}$. Then, to obtain an $\epsilon$-accurate solution $\mathbf{x}_\zeta$, the total number of iterations  required by AbaSGD
	\begin{align*} 
		T = \frac{  16L\left(f(\mathbf{x}_{0 } )	 -   f^*  \right) +   4m   \alpha_0^2 }{\epsilon },
	\end{align*}
	and the total number of SFO calls required by AbaSGD is given by 
	\begin{align*}
		\sum_{t=0}^{T} |B_t| =\underbrace{ \sum_{t=0}^{T}  \min \left\{ {\color{blue}{\frac{2\sigma^2}{\sum_{i=1}^{m}\|\mathbf{v}_{t-i}\|^2/m}}},\frac{ 24\sigma^2}{\epsilon} , n\right\}}_{\text{\normalfont  complexity of AbaSGD}} \leq  \underbrace{T\min \left\{ \frac{ 24\sigma^2}{\epsilon} , n\right\}}_{\text{\normalfont  complexity of vanilla SGD}}  = \mathcal{O}  \left(\frac{1}{\epsilon^2} \wedge \frac{n}{\epsilon}\right). 
	\end{align*} 
\end{corollary}
Corollary~\ref{co:adasgd} shows that the worst-case complexity of AbaSGD is $\mathcal{O}  \left(\frac{1}{\epsilon^2} \wedge \frac{n}{\epsilon}\right)$, which 
is at least as good as those of SGD and GD. More importantly, the actual complexity of AbaSGD can be much lower than those of GD and SGD due to the adaptive batch size.

 \newpage
{\Large{\bf Technical Proofs}}
 \section{Proofs for Results in Section~\ref{se:4}}\label{apen:svrg}
 \subsection{Proof of Theorem~\ref{th_svrg}}
To prove Theorem~\ref{th_svrg}, we first establish the following lemma to upper-bound the estimation variance $\mathbb{E}_{0,s}\|\nabla f({x}_{t-1}^s)-{v}_{t-1}^s\|^2$ for $1\leq t\leq m$, where $\mathbb{E}_{t,s}(\cdot)$  denotes $\mathbb{E}(\cdot | {x}_0^1,{x}_0^2,...,{x}_2^1,...,{x}_{t}^s)$.
\begin{lemma}\label{le:variance}
Let Assumption~\ref{assum1} hold. Then, for $1\leq t \leq m$, we have 
\begin{align}
\mathbb{E}_{0,s}\|\nabla f({x}_{t-1}^s)-{v}_{t-1}^s\|^2\leq \frac{\eta^2L^2(t-1)}{B}\mathbb{E}_{0,s}\sum_{i=0}^{t-2}\|{v}_i^s\|^2 +\frac{I_{(N_s<n)}}{N_s} \sigma^2
\end{align}
where $I_{(A)}=1$ if the event $A$ occurs and $0$ otherwise, and $\sum_{i=0}^{-1}\|{v}_i^s\|^2=0$.
\end{lemma}
\begin{proof}[Proof of \Cref{le:variance}]
Based on line 10 in Algorithm~\ref{alg:svrg}, we have, for $1\leq t \leq m$,
\begin{align*}
\|{v}_{t-1}^s  -& \nabla f({x}_{t-1}^s)\|^2 =\|\nabla f_{\mathcal{B}} ({x}_{t-1}^s)  -   \nabla f_{\mathcal{B}} ({\tilde x}^{s-1}) - \nabla f({x}_{t-1}^s) + \nabla f({\tilde x}^{s-1})  +{g}^s - \nabla f({\tilde x}^{s-1})\|^2.
\end{align*}
Taking the expectation $\mathbb{E}_{0,s}(\cdot)$  over the above equality yields
\begin{align}
\mathbb{E}_{0,s}&\|{v}_{t-1}^s  - \nabla f({x}_{t-1}^s)\|^2 = \mathbb{E}_{0,s} \|\nabla f_{\mathcal{B}} ({x}_{t-1}^s)  -   \nabla f_{\mathcal{B}} ({\tilde x}^{s-1}) - \nabla f({x}_{t-1}^s) + \nabla f({\tilde x}^{s-1}) \|^2 \nonumber 
\\&+ 2\underbrace{\mathbb{E}_{0,s}\langle \nabla f_{\mathcal{B}} ({x}_{t-1}^s)  -   \nabla f_{\mathcal{B}} ({\tilde x}^{s-1}) - \nabla f({x}_{t-1}^s) + \nabla f({\tilde x}^{s-1}) , {g}^s - \nabla f({\tilde x}^{s-1})  \rangle }_{(*)}\nonumber
\\ &+\mathbb{E}_{0,s} \|{g}^s - \nabla f({\tilde x}^{s-1})\|^2,
\end{align}
which, in conjunction with the fact that 
\begin{align}
(*) = \mathbb{E}_{{x}_1^s,..,{x}_{t-1}^s}\mathbb{E}_{t-1,s}\langle \nabla f_{\mathcal{B}} ({x}_{t-1}^s)  -   \nabla f_{\mathcal{B}} ({\tilde x}^{s-1}) - \nabla f({x}_{t-1}^s) + \nabla f({\tilde x}^{s-1}) , {g}^s - \nabla f({\tilde x}^{s-1})  \rangle = 0\nonumber
\end{align}
and letting $F_i := \nabla f_{i} ({x}_{t-1}^s)  -   \nabla f_{i} ({\tilde x}^{s-1}) - \nabla f({x}_{t-1}^s) + \nabla f({\tilde x}^{s-1})$,  implies that 
\begin{align}\label{eq:val}
\mathbb{E}_{0,s}&\|{v}_{t-1}^s  - \nabla f({x}_{t-1}^s)\|^2  \nonumber
\\=& \mathbb{E}_{0,s} \|\nabla f_{\mathcal{B}} ({x}_{t-1}^s)  -   \nabla f_{\mathcal{B}} ({\tilde x}^{s-1}) - \nabla f({x}_{t-1}^s) + \nabla f({\tilde x}^{s-1}) \|^2 +\mathbb{E}_{0,s} \|{g}^s - \nabla f({\tilde x}^{s-1})\|^2 \nonumber
\\=& \frac{1}{B^2} \mathbb{E}_{0,s}  \sum_{i\in \mathcal{B}} \|\nabla f_{i} ({x}_{t-1}^s)  -   \nabla f_{i} ({\tilde x}^{s-1}) - \nabla f({x}_{t-1}^s) + \nabla f({\tilde x}^{s-1})\|^2 + \mathbb{E}_{0,s} \|{g}^s - \nabla f({\tilde x}^{s-1})\|^2 \nonumber
\\&+ \frac{2}{B^2}\sum_{i<j,\, i,j\in \mathcal{B}} \mathbb{E}_{0,s} \langle F_i, F_j \rangle \nonumber
\\\overset{(i)}=&\frac{1}{B}\mathbb{E}_{0,s} \|\nabla f_{i} ({x}_{t-1}^s)  -   \nabla f_{i} ({\tilde x}^{s-1}) - \nabla f({x}_{t-1}^s) + \nabla f({\tilde x}^{s-1}) \|^2 +\mathbb{E}_{0,s} \|{g}^s - \nabla f({\tilde x}^{s-1})\|^2  \nonumber
\\ \overset{(ii)}\leq & \frac{1}{B} \mathbb{E}_{0,s}\| \nabla f_{i} ({x}_{t-1}^s)  -   \nabla f_{i} ({\tilde x}^{s-1}) \|^2+\mathbb{E}_{0,s} \|{g}^s - \nabla f({\tilde x}^{s-1})\|^2  \nonumber
\\ \overset{(iii)}\leq & \frac{1}{B} \mathbb{E}_{0,s}\| \nabla f_{i} ({x}_{t-1}^s)  -   \nabla f_{i} ({\tilde x}^{s-1}) \|^2+\frac{I_{(N_s<n)}}{N_s}\sigma^2
\end{align} 
where (i) follows from the fact that 
\begin{align}
\mathbb{E}_{0,s} \langle F_i,F_j \rangle = \mathbb{E}_{{x}_1^s,...,{x}_{t-1}^s}\left( \mathbb{E}_{t-1,s} \langle F_i,F_j \rangle\right) = \mathbb{E}_{{x}_1^s,...,{x}_{t-1}^s}\left( \langle  \mathbb{E}_{t-1,s}(F_i),\mathbb{E}_{t-1,s}(F_j) \rangle\right) =0, \nonumber
\end{align}
(ii) follows from the fact that $\mathbb{E}\|{y}-\mathbb{E}({y})\|^2\leq \mathbb{E}\|{y}\|^2$ for any ${y}\in\mathbb{R}^d$, and (iii) follows by combining Lemma B.2 in~\citealt{lei2017non} and the fact that $N_s$ is fixed given ${x}_0^1,...,{x}_{0}^{s}$. Then, we obtain from~\eqref{eq:val} that 
\begin{align}
\mathbb{E}_{0,s}\|{v}_{t-1}^s  - \nabla f({x}_{t-1}^s)\|^2 \leq& \frac{L^2}{B} \mathbb{E}_{0,s}\|{x}_{t-1}^s - {\tilde x}^{s-1}\| + \frac{I_{(N_s<n)}}{N_s}\sigma^2 \nonumber
\\=& \frac{L^2}{B}\mathbb{E}_{0,s} \Big\|\sum_{i=0}^{t-2}({x}_{i+1}^s-{x}_{i}^s)\Big\|^2 + \frac{I_{(N_s<n)}}{N_s}\sigma^2 \nonumber
\\=& \frac{\eta^2L^2}{B}\mathbb{E}_{0,s} \Big\|\sum_{i=0}^{t-2}{v}_{i}^s\Big\|^2 + \frac{I_{(N_s<n)}}{N_s}\sigma^2 \nonumber
\\\overset{(i)}\leq & \frac{\eta^2L^2(t-1)}{B}\mathbb{E}_{0,s}\sum_{i=0}^{t-2}\Big\|{v}_{i}^s\Big\|^2 + \frac{I_{(N_s<n)}}{N_s}\sigma^2,
\end{align}
where (i) follows from the  Cauchy–Schwartz inequality that $\|\sum_{i=1}^k{a}_i\|^2 \leq k\sum_{i=1}^k\|{a}_i\|^2$.
\end{proof}
\begin{proof}[{\bf Proof of Theorem~\ref{th_svrg}}]
Based on Lemma~\ref{le:variance}, we next prove Theorem~\ref{th_svrg}. 

Since the objective function $f(\cdot)$ has a $L$-Lipschitz continuous gradient, we obtain that  for $1 \leq t \leq m$, 
\begin{align}
f({x}_t^s) &\leq  f({x}_{t-1}^s) + \langle \nabla f({x}_{t-1}^s), {x}_{t}^s-{x}_{t-1}^s\rangle +\frac{L\eta^2}{2} \|{v}_{t-1}^s\|^2  \nonumber
\\= & f({x}_{t-1}^s) + \langle \nabla f({x}_{t-1}^s)-{v}_{t-1}^s, -\eta {v}_{t-1}^s\rangle - \eta\|{v}_{t-1}^s\|^2 +\frac{L\eta^2}{2} \|{v}_{t-1}^s\|^2 \nonumber
\\\overset{(i)}\leq& f({x}_{t-1}^s) + \frac{\eta}{2}\| \nabla f({x}_{t-1}^s)-{v}_{t-1}^s \|^2  +\frac{\eta}{2} \|{v}_{t-1}^s\|^2
-\Big (\eta-\frac{L\eta^2}{2} \Big)\|{v}_{t-1}^s\|^2.  \nonumber
\end{align}
where (i) follows from the inequality that $\langle {a}, {b}\rangle \leq \frac{1}{2}(\|{a}\|^2+\|{b}\|^2)$. 
Then, taking expectation $\mathbb{E}_{0,s}(\cdot)$ over the above inequality yields
\begin{align}\label{eq:s1}
\mathbb{E}_{0,s} f({x}_t^s) \leq& \mathbb{E}_{0,s} f({x}_{t-1}^s) +\frac{\eta}{2} \mathbb{E}_{0,s}\| \nabla f({x}_{t-1}^s)-{v}_{t-1}^s \|^2 -\Big (\frac{\eta}{2}-\frac{L\eta^2}{2} \Big)\mathbb{E}_{0,s}\|{v}_{t-1}^s\|^2. 
\end{align}
Combining~\eqref{eq:s1} and Lemma~\ref{le:variance} yields, for $1\leq t\leq m$  
\begin{align*}
\mathbb{E}_{0,s} f({x}_t^s) \leq & \mathbb{E}_{0,s} f({x}_{t-1}^s) +\frac{\eta^3L^2(t-1)}{2B}\mathbb{E}_{0,s}\sum_{i=0}^{t-2}\|{v}_i^s\|^2 +\frac{\eta I_{(N_s<n)}}{2N_s} \sigma^2
\\&-\Big (\frac{\eta}{2}-\frac{L\eta^2}{2} \Big)\mathbb{E}_{0,s}\|{v}_{t-1}^s\|^2.
\end{align*}
Telescoping the above inequality over $t$ from $1$ to $m$ yields
\begin{align}\label{eq:s2}
\mathbb{E}_{0,s} f({x}_m^s) \leq& \mathbb{E}_{0,s} f({x}_0^s) + \sum_{t=1}^m\frac{\eta^3L^2(t-1)}{2B}  \mathbb{E}_{0,s}  \sum_{i=0}^{t-2} \|{v}_i^s\|^2  + \frac{\eta \sigma^2 m I_{(N_s<n)}}{2N_s}  \nonumber
\\&-  \Big (\frac{\eta}{2}-\frac{L\eta^2}{2} \Big)\sum_{t=0}^{m-1}\mathbb{E}_{0,s}\|{v}_{t}^s\|^2  \nonumber
\\\overset{(i)}\leq & \mathbb{E}_{0,s} f({x}_0^s) + \frac{\eta^3L^2m^2}{2B}  \mathbb{E}_{0,s}  \sum_{i=0}^{m-1} \|{v}_i^s\|^2  + \frac{\eta \sigma^2 m I_{(N_s<n)}}{2N_s}  \nonumber
\\&- \Big (\frac{\eta}{2}-\frac{L\eta^2}{2} \Big)\sum_{t=0}^{m-1}\mathbb{E}_{0,s}\|{v}_{t}^s\|^2 ,
\end{align}
where (i) follows from the fact that $\frac{\eta^3L^2(t-1)}{2B}  \mathbb{E}_{0,s}  \sum_{i=0}^{t-2} \|{v}_i^s\|^2\leq \frac{\eta^3L^2m}{2B}  \mathbb{E}_{0,s}  \sum_{i=0}^{m-1} \|{v}_i^s\|^2$.
Recall that $N_s =\min\{c_\beta \sigma^2 \beta_s^{-1},c_\epsilon\sigma^2  \epsilon^{-1},n\}$ and $c_\beta, c_\epsilon \geq \alpha$.
 Then, we have 
\begin{align}\label{eq:bsn}
 \frac{ I_{(N_s<n)}}{N_s} \leq \frac{1}{\min\{c_\beta \sigma^2\beta_s^{-1}, c_\epsilon\sigma^2\epsilon^{-1}\}} = \max\left\{\frac{\beta_s}{c_\beta\sigma^2}, \frac{\epsilon}{c_\epsilon\sigma^2}\right\}\leq \max\left\{\frac{\beta_s}{\alpha\sigma^2}, \frac{\epsilon}{\alpha\sigma^2}\right\},
\end{align}
To explain the first inequality in~\eqref{eq:bsn}, we denote $N_s = \min(n_1, n_2, n)$ for simplicty. If $N_s \geq n$, then the indicator function $I_{(\cdot)} = 0$ and hence $I_{(\cdot)}/N_s = 0 < 1/\min(n_1, n_2)$. If $N_s < n$, then $I_{(\cdot)} = 1$ and $N_s = \min(n_1, n_2)$ and hence $I_{(\cdot)}/N_s = 1/\min(n_1,n_2)$. Combining the above two cases yields the first inequality in~\eqref{eq:bsn}. 
Combining~\eqref{eq:bsn} and~\eqref{eq:s2} yields  
\begin{align}\label{eq:s3}
\mathbb{E}_{0,s} f({x}_m^s) \overset{(i)}\leq & \mathbb{E}_{0,s} f({x}_0^s) + \frac{\eta^3L^2m^2}{2B}  \mathbb{E}_{0,s}  \sum_{i=0}^{m-1} \|{v}_i^s\|^2  + \eta m\left(  \frac{\beta_s}{2\alpha} + \frac{\epsilon}{2\alpha} \right) \nonumber
\\&-  \Big (\frac{\eta}{2}-\frac{L\eta^2}{2} \Big)\sum_{t=0}^{m-1}\mathbb{E}_{0,s}\|{v}_{t}^s\|^2 . 
\end{align}
where (i) follows from the fact that $\max(a,b)\leq a+b$. Taking the expectation of~\eqref{eq:s3} over ${x}_{0}^1,...,{x}_0^s$, we obtain
\begin{align*}
\mathbb{E} f({x}_m^s)\leq & \mathbb{E}f({x}_0^s) + \frac{\eta m}{2\alpha} \mathbb{E}\beta_s+  \frac{\eta m\epsilon}{2\alpha}-  \Big (\frac{\eta}{2}-\frac{L\eta^2}{2}- \frac{\eta^3L^2m^2}{2B}  \Big)\sum_{t=0}^{m-1}\mathbb{E}\|{v}_{t}^s\|^2 . 
\end{align*}
Recall that $\beta_1 \leq \epsilon S$ and $\beta_s = \frac{1}{m}\sum_{t=1}^m \|{v}_{t-1}^{s-1}\|^2$ for $s=2,...,S$.  Then, telescoping the above inequality over $s$ from $1$ to $S$ and noting that ${x}_m^s={x}_0^{s+1}$, we obtain 
\begin{align}\label{eq:ss}
\mathbb{E} f({x}_m^S)\leq & \mathbb{E}f({x}_0) -  \Big (\frac{\eta}{2}-\frac{L\eta^2}{2}- \frac{\eta^3L^2m^2}{2B}  \Big)\sum_{s=1}^S\sum_{t=0}^{m-1}\mathbb{E}\|{v}_{t}^s\|^2 +  \frac{\eta mS\epsilon}{2\alpha} \nonumber
\\&+ \frac{\eta mS\epsilon}{2\alpha} + \sum_{s=2}^S\frac{\eta m}{2\alpha} \mathbb{E}\left( \frac{1}{m}\sum_{t=1}^m \|{v}_{t-1}^{s-1}\|^2\right) \nonumber
\\\leq& \mathbb{E}f({x}_0) -  \Big (\frac{\eta}{2}-\frac{L\eta^2}{2}- \frac{\eta^3L^2m^2}{2B} -\frac{\eta}{2\alpha} \Big)\sum_{s=1}^S\sum_{t=0}^{m-1}\mathbb{E}\|{v}_{t}^s\|^2 + \frac{\eta mS\epsilon}{\alpha}.
\end{align}
Dividing the both sides of~\eqref{eq:ss} by $\eta Sm$ and rearranging the terms, we obtain
\begin{align}\label{eq:ss1}
\Big (\frac{1}{2}-\frac{1}{2\alpha}-\frac{L\eta}{2}- \frac{\eta^2L^2m^2}{2B} \Big)\frac{1}{Sm}\sum_{s=1}^S\sum_{t=0}^{m-1}\mathbb{E}\|{v}_{t}^s\|^2 \leq \frac{f({x}_0)- f^*}{\eta Sm} + \frac{\epsilon}{\alpha},
\end{align} 
where $f^*=\inf_{{x}\in\mathbb{R}^d} f({x})>-\infty$. Since the output ${x}_{\zeta}$ is chosen  from $\{ {x}_{t}^s\}_{t=0,...,m-1, s=1,...,S}$ uniformly at random, we have 
\begin{align}\label{eq:ss2}
Sm\, \mathbb{E}\|\nabla f({x}_\zeta)\|^2 = &\sum_{s=1}^S\sum_{t=0}^{m-1}\mathbb{E}\|\nabla f({x}_t^s)\|^2 \nonumber
\\\leq&2\sum_{s=1}^S\sum_{t=0}^{m-1}\mathbb{E}\|\nabla f({x}_t^s)-{v}_t^s\|^2 + 2\sum_{s=1}^S\sum_{t=0}^{m-1}\mathbb{E}\|{v}_t^s\|^2  \nonumber
\\= & 2\sum_{s=1}^S\sum_{t=0}^{m-1}\mathbb{E}_{{x}_0^1,...,{x}_0^s}\left(\mathbb{E}_{0,s}\|\nabla f({x}_t^s)-{v}_t^s\|^2\right) + 2\sum_{s=1}^S\sum_{t=0}^{m-1}\mathbb{E}\|{v}_t^s\|^2 \nonumber
\\\overset{(i)}\leq &  2\sum_{s=1}^S\sum_{t=0}^{m-1}\mathbb{E}_{{x}_0^1,...,{x}_0^s}\left(\frac{\eta^2 L^2 m}{B}\mathbb{E}_{0,s}\sum_{i=0}^{m-1}\|{v}_i^s\|^2+\frac{\beta_s}{\alpha}+\frac{\epsilon}{\alpha}\right) + 2\sum_{s=1}^S\sum_{t=0}^{m-1}\mathbb{E}\|{v}_t^s\|^2 \nonumber
\\\leq & 2\sum_{s=1}^S \left(\frac{\eta^2 L^2 m^2}{B}\mathbb{E}\sum_{i=0}^{m-1}\|{v}_i^s\|^2+\frac{ m \beta_s}{\alpha}+\frac{m\epsilon}{\alpha}\right) + 2\sum_{s=1}^S\sum_{t=0}^{m-1}\mathbb{E}\|{v}_t^s\|^2  \nonumber
\\ \overset{(ii)}\leq &  \left(\frac{2\eta^2 L^2 m^2}{B} + \frac{2}{\alpha}+2\right)\sum_{s=1}^S \sum_{t=0}^{m-1}\mathbb{E}\|{v}_t^s\|^2+\frac{4Sm\epsilon}{\alpha} 
\end{align}
where (i) follows from Lemma~\ref{le:variance} and~\eqref{eq:bsn}, and (ii) follows from the definition of $\beta_s$ for $s=1,...,S$.  Combining~\eqref{eq:ss1} and~\eqref{eq:ss2}  and letting $\phi = \frac{1}{2}-\frac{1}{2\alpha}-\frac{L\eta}{2}- \frac{\eta^2L^2m^2}{2B}$ and $\psi=\frac{2\eta^2 L^2 m^2}{B} + \frac{2}{\alpha}+2 $, we have 
\begin{align}
\mathbb{E}\|\nabla f({x}_\zeta)\|^2   \leq  \frac{\psi(f({x}_0)- f^*)}{\phi \eta Sm} + \frac{\psi\epsilon}{\phi \alpha} + \frac{4\epsilon}{\alpha},
\end{align} 
which finishes the proof. 
\end{proof}
\subsection{Proof of Corollary~\ref{co:svrg} }
Recall that $\eta = \frac{1}{4L}$, $B = m^2$ and $c_\beta, c_\epsilon\geq 16$. Then, we have $\alpha=16$, $\phi\geq\frac{5}{16}>\frac{1}{4}$ and $\psi\leq \frac{9}{4} $ in Theorem~\ref{th_svrg}, and thus 
\begin{align*}
\mathbb{E}\|\nabla f({x}_\zeta)\|^2   \leq  \frac{36L(f({x}_0)- f^*)}{ K} + \frac{13}{16}\epsilon.
\end{align*}
Thus, to achieve $\mathbb{E}\|\nabla f({x}_\zeta)\|^2 <\epsilon$, AbaSVRG requires at most $192L(f({x}_0)- f^*)\epsilon^{-1}=\Theta (\epsilon^{-1})$ iterations. Then, the total number of SFO calls is given by 
\begin{align*}
\sum_{s=1}^S\min\{c_\beta\sigma^2\beta_s^{-1}, c_\epsilon\sigma^2\epsilon^{-1},n\} + KB \leq &S (c_\epsilon\sigma^2\epsilon^{-1} \wedge n) + KB \leq \mathcal{O} \left( \frac{\epsilon^{-1} \wedge n }{\epsilon\sqrt{B}}+ \frac{B}{\epsilon}\right).
\end{align*}
Furthermore, if we choose  $B=n^{2/3} \wedge \epsilon^{-2/3}$, then SFO complexity of AbaSVRG becomes 
\begin{align}
\mathcal{O} \left( \frac{\epsilon^{-2/3} \wedge n^{2/3}}{\epsilon}\right) \leq \mathcal{O}\left(\frac{1}{\epsilon}\left(n \wedge \frac{1}{\epsilon}\right)^{2/3}\right).
\end{align}
\subsection{Complexity under $B=m$}\label{apen:a3}
\begin{corollary}
 Let stepsize $\eta = \frac{1}{4L\sqrt{m}}$, mini-batch size $B = m$ and $c_\beta, c_\epsilon \geq 16$. Then, to obtain an $\epsilon$-accurate solution ${x}_\zeta$,  the total number of SFO calls required by AbaSVRG is given by 
\begin{align*}
\sum_{s=1}^S\min\left\{\frac{c_\beta\sigma^2}{\beta_s}, \frac{c_\epsilon\sigma^2}{\epsilon},n\right\} + KB
\leq\mathcal{O}\left(\frac{n \wedge \epsilon^{-1}}{\sqrt{B}\epsilon} + \frac{B^{3/2}}{\epsilon}\right).
\end{align*}
If we specially choose $B=n^{1/2}\wedge \epsilon^{-1/2}$, then the worst-case complexity is $\mathcal{O}\left(\frac{1}{\epsilon}(n \wedge \frac{1}{\epsilon})^{3/4}\right).$
\end{corollary}
\begin{proof}
Since $\eta = \frac{1}{4L\sqrt{m}}$, $B = m$ and $c_\beta, c_\epsilon \geq 16$, we obtain $\alpha=16$, $\phi=\frac{7}{16}-\frac{1}{8\sqrt{m}}\geq \frac{5}{16}>\frac{1}{4}$ and $\psi \leq \frac{9}{4}$ in Theorem~\ref{th_svrg}, and thus 
\begin{align*}
\mathbb{E}\|\nabla f({x}_\zeta)\|^2   \leq  \frac{36L\sqrt{m}(f({x}_0)- f^*)}{ K} + \frac{13}{16}\epsilon.
\end{align*}
To achieve $\mathbb{E}\|\nabla f({x}_\zeta)\|^2 <\epsilon$, AbaSVRG requires at most $192L\sqrt{m}(f({x}_0)- f^*)\epsilon^{-1}=\Theta (\sqrt{m}\epsilon^{-1})$ iterations. Then, the total number of SFO calls is given by
\begin{align*}
\sum_{s=1}^S\min\{c_\beta\sigma^2\beta_s^{-1}, c_\epsilon\sigma^2\epsilon^{-1},n\} + KB \leq &S (c_\epsilon\sigma^2\epsilon^{-1} \wedge n) + KB \leq \mathcal{O} \left( \frac{\epsilon^{-1} \wedge n }{\epsilon\sqrt{B}}+ \frac{B^{3/2}}{\epsilon}\right).
\end{align*}
Furthermore, if we choose  $B=n^{1/2} \wedge \epsilon^{-1/2}$,  then the SFO complexity is $\mathcal{O}\left(\frac{1}{\epsilon}(n \wedge \frac{1}{\epsilon})^{3/4}\right).$
\end{proof}
\subsection{Proof of Theorem~\ref{th_spider}}
In order to prove Theorem~\ref{th_spider}, we first use the following lemma to provide an upper bound on the estimation variance $\mathbb{E}_{0,s}\|\nabla f({x}_{t}^s)-{v}_{t}^s\|^2$ for $0 \leq t \leq m-1$, where $\mathbb{E}_{t,s}(\cdot)$ denotes $\mathbb{E}(\cdot | {x}_0^1,{x}_0^2,...,{x}_2^1,...,{x}_{t}^s)$. 
\begin{lemma}[Adapted from~\citealt{fang2018spider}]\label{le:spider}
Let Assumption~\ref{assum1} hold. Then, for $0\leq t \leq m-1$, 
\begin{align}
\mathbb{E}_{0,s} \|\nabla f({x}_{t}^s)-{v}_{t}^s\|^2 \leq \frac{\eta ^2 L^2}{B}	\sum_{i=0}^{t-1}\mathbb{E}_{0,s}\|{v}_{i}^s\|^2 + 	\frac{I_{(N_s<n)}}{N_s}\sigma^2.
\end{align}
where we define the stochastic gradients before the algorithm starts to satisfy $\sum_{i=0}^{-1}\mathbb{E}_{0,s}\|{v}_{i}^s\|^2 =0$ for easy presentation. 
\end{lemma}
\begin{proof}[Proof of \Cref{le:spider}]
	Combining A.3 and A.4 in~\citealt{fang2018spider} yields, for $1\leq i \leq m-1$,
	\begin{align}
	\mathbb{E}_{i,s} \|\nabla f({x}_{i}^s)-{v}_{i}^s\|^2 \leq& \frac{L^2}{B}\|{x}_{i}^s-{x}_{i-1}^s\|^2 + \|\nabla f({x}_{i-1}^s)-{v}_{i-1}^s\|^2  \nonumber
	\\=& \frac{\eta ^2 L^2}{B}\|{v}_{i-1}^s\|^2 + \|\nabla f({x}_{i-1}^s)-{v}_{i-1}^s\|^2. \nonumber
	\end{align}
	Taking the expectation of the above inequality over ${x}_1^s,...,{x}_{i}^s$, we have 
	\begin{align*}
	\mathbb{E}_{0,s} \|\nabla f({x}_{i}^s)-{v}_{i}^s\|^2 \leq \frac{\eta ^2 L^2}{B}	\mathbb{E}_{0,s}\|{v}_{i-1}^s\|^2 + 	\mathbb{E}_{0,s}\|\nabla f({x}_{i-1}^s)-{v}_{i-1}^s\|^2.
	\end{align*}
Then, telescoping the above inequality over $i$ from $1$ to $t$ yields 	
\begin{align}\label{eq:dodo}
	\mathbb{E}_{0,s} \|\nabla f({x}_{t}^s)-{v}_{t}^s\|^2 \leq \frac{\eta ^2 L^2}{B}	\sum_{i=0}^{t-1}\mathbb{E}_{0,s}\|{v}_{i}^s\|^2 + 	\mathbb{E}_{0,s}\|\nabla f({x}_{0}^s)-{v}_{0}^s\|^2.
\end{align}
Based on Lemma B.2 in~\citealt{lei2017non}, we have 
\begin{align*}
	\mathbb{E}_{0,s}\|\nabla f({x}_{0}^s)-{v}_{0}^s\|^2 \leq \frac{I_{(N_s<n)}}{N_s}\sigma^2,
\end{align*}
	which, combined with~\eqref{eq:dodo}, finishes the proof.
\end{proof}
\begin{proof}[{\bf Proof of Theorem~\ref{th_spider}}]
Based on Lemma~\ref{le:spider}, we now prove Theorem~\ref{th_spider}.

Since the objective function $f(\cdot)$ has a $L$-Lipschitz continuous gradient, we obtain that  for $1 \leq t \leq m$, 
\begin{align*}
f({x}_t^s) \leq&  f({x}_{t-1}^s) + \langle \nabla f({x}_{t-1}^s), {x}_{t}^s-{x}_{t-1}^s\rangle +\frac{L\eta^2}{2} \|{v}_{t-1}^s\|^2  \nonumber
\\= & f({x}_{t-1}^s) + \langle \nabla f({x}_{t-1}^s)-{v}_{t-1}^s, -\eta {v}_{t-1}^s\rangle - \eta\|{v}_{t-1}^s\|^2 +\frac{L\eta^2}{2} \|{v}_{t-1}^s\|^2 \nonumber
\\\overset{(i)}\leq&f({x}_{t-1}^s) + \frac{\eta}{2}\|\nabla f({x}_{t-1}^s)-{v}_{t-1}^s\|^2 +\frac{\eta}{2}\| {v}_{t-1}^s\|^2 - \eta\|{v}_{t-1}^s\|^2 +\frac{L\eta^2}{2} \|{v}_{t-1}^s\|^2 \nonumber
\\\leq &f({x}_{t-1}^s) + \frac{\eta}{2}\|\nabla f({x}_{t-1}^s)-{v}_{t-1}^s\|^2  - \left(\frac{\eta}{2}-\frac{L\eta^2}{2}\right)\|{v}_{t-1}^s\|^2, 
\end{align*}
where (i) follows from the inequality that $\langle {a}, {b}\rangle \leq \frac{1}{2}(\|{a}\|^2+\|{b}\|^2)$. 
Then, taking expectation $\mathbb{E}_{0,s}$ over the above inequality and applying Lemma~\ref{le:spider}, we have, for $1\leq t \leq m$,
\begin{align}
\mathbb{E}_{0,s} f({x}_t^s) \leq& \mathbb{E}_{0,s} f({x}_{t-1}^s) + \frac{\eta}{2}\mathbb{E}_{0,s}\|\nabla f({x}_{t-1}^s)-{v}_{t-1}^s\|^2  - \left(\frac{\eta}{2}-\frac{L\eta^2}{2}\right)\mathbb{E}_{0,s}\|{v}_{t-1}^s\|^2 \nonumber
\\\leq&  \mathbb{E}_{0,s} f({x}_{t-1}^s)
+\frac{\eta ^3 L^2}{2B}	\sum_{i=0}^{t-2}\mathbb{E}_{0,s}\|{v}_{i}^s\|^2 + 	\frac{I_{(N_s<n)}}{N_s}\frac{\eta\sigma^2}{2}
- \left(\frac{\eta}{2}-\frac{L\eta^2}{2}\right)\mathbb{E}_{0,s}\|{v}_{t-1}^s\|^2 \nonumber
\\\overset{(i)}\leq&   \mathbb{E}_{0,s} f({x}_{t-1}^s)
+\frac{\eta ^3 L^2}{2B}	\sum_{i=0}^{m-1}\mathbb{E}_{0,s}\|{v}_{i}^s\|^2 +\max\left\{\frac{\eta\beta_s}{2\alpha}, \frac{\eta\epsilon}{2\alpha}\right\}
- \left(\frac{\eta}{2}-\frac{L\eta^2}{2}\right)\mathbb{E}_{0,s}\|{v}_{t-1}^s\|^2 \nonumber
\end{align}
where (i) follows from $t-2<m-1$ and~\eqref{eq:bsn}. Telescoping the above inequality over $t$ from $1$ to $m$ and using $\max(a,b)\leq a+b$ yield 
\begin{align*}
\mathbb{E}_{0,s} f({x}_m^s) \leq \mathbb{E}_{0,s} f({x}_0^s) -\left( \frac{\eta}{2}-\frac{L\eta^2}{2}- \frac{\eta ^3 L^2 m}{2B} \right)	\sum_{t=0}^{m-1}\mathbb{E}_{0,s}\|{v}_{t}^s\|^2 + \frac{\eta m\beta_s}{2\alpha} + \frac{\eta m\epsilon}{2\alpha}. 
\end{align*}
Taking the expectation of the above inequality over  ${x}_0^1,...,{x}_0^s$, we obtain
\begin{align}
\mathbb{E} f({x}_m^s) \leq \mathbb{E}f({x}_0^s) -\left( \frac{\eta}{2}-\frac{L\eta^2}{2}- \frac{\eta ^3 L^2 m}{2B} \right)	\sum_{t=0}^{m-1}\mathbb{E}\|{v}_{t}^s\|^2 + \frac{\eta m}{2\alpha}\mathbb{E}(\beta_s) + \frac{\eta m\epsilon}{2\alpha}. \nonumber
\end{align}
Recall that $\beta_1 \leq \epsilon S$ and $\beta_s = \frac{1}{m}\sum_{t=0}^{m-1} \|{v}_{t}^{s-1}\|^2$ for $s=2,...,S$. Then, telescoping the above inequality over $s$ from $1$ to $S$ and noting that ${x}_m^s = {x}_0^{s+1}={\tilde x}^{s}$, we have 
\begin{align*}
\mathbb{E} f({\tilde x}^S) \leq& \mathbb{E}f({x}_0) -\left( \frac{\eta}{2}-\frac{L\eta^2}{2}- \frac{\eta ^3 L^2 m}{2B} \right)\sum_{s=1}^S	\sum_{t=0}^{m-1}\mathbb{E}\|{v}_{t}^s\|^2  + \frac{\eta mS\epsilon}{2\alpha} \nonumber
\\&+\frac{\eta}{2\alpha}\sum_{s=1}^{S-1}	\sum_{t=0}^{m-1}\mathbb{E}\|{v}_{t}^s\|^2  \nonumber
\\\leq& \mathbb{E}f({x}_0) -\left( \frac{\eta}{2}-\frac{\eta}{2\alpha}-\frac{L\eta^2}{2}- \frac{\eta ^3 L^2 m}{2B} \right)\sum_{s=1}^S	\sum_{t=0}^{m-1}\mathbb{E}\|{v}_{t}^s\|^2  + \frac{\eta mS\epsilon}{2\alpha}. 
\end{align*}
Dividing the both sides of the above inequality by $\eta Sm$ and rearranging the terms, we obtain
\begin{align}\label{eq:aps}
\left( \frac{1}{2}-\frac{1}{2\alpha}-\frac{L\eta}{2}- \frac{\eta ^2 L^2 m}{2B} \right)\frac{1}{Sm}\sum_{s=1}^S	\sum_{t=0}^{m-1}\mathbb{E}\|{v}_{t}^s\|^2  \leq \frac{f({x}_0)- f^*}{\eta Sm} + \frac{\epsilon}{2\alpha}.
\end{align}
 Since the output ${x}_{\zeta}$ is chosen  from $\{ {x}_{t}^s\}_{t=0,...,m-1, s=1,...,S}$ uniformly at random,  we have 
\begin{align}\label{eq:gg1}
Sm\, \mathbb{E}\|\nabla f({x}_\zeta)\|^2 = &\sum_{s=1}^S\sum_{t=0}^{m-1}\mathbb{E}\|\nabla f({x}_t^s)\|^2 \nonumber
\\\leq&2\sum_{s=1}^S\sum_{t=0}^{m-1}\mathbb{E}\|\nabla f({x}_t^s)-{v}_t^s\|^2 + 2\sum_{s=1}^S\sum_{t=0}^{m-1}\mathbb{E}\|{v}_t^s\|^2  \nonumber
\\= & 2\sum_{s=1}^S\sum_{t=0}^{m-1}\mathbb{E}_{{x}_0^1,...,{x}_0^s}\left(\mathbb{E}_{0,s}\|\nabla f({x}_t^s)-{v}_t^s\|^2\right) + 2\sum_{s=1}^S\sum_{t=0}^{m-1}\mathbb{E}\|{v}_t^s\|^2 \nonumber
\\\overset{(i)}\leq &  2\sum_{s=1}^S\sum_{t=0}^{m-1}\mathbb{E}_{{x}_0^1,...,{x}_0^s}\left(\frac{\eta^2 L^2 }{B}\mathbb{E}_{0,s}\sum_{i=0}^{m-1}\|{v}_i^s\|^2+\frac{\beta_s}{\alpha}+\frac{\epsilon}{\alpha}\right) + 2\sum_{s=1}^S\sum_{t=0}^{m-1}\mathbb{E}\|{v}_t^s\|^2 \nonumber
\\\leq & 2\sum_{s=1}^S \left(\frac{\eta^2 L^2 m}{B}\mathbb{E}\sum_{i=0}^{m-1}\|{v}_i^s\|^2+\frac{ m \beta_s}{\alpha}+\frac{m\epsilon}{\alpha}\right) + 2\sum_{s=1}^S\sum_{t=0}^{m-1}\mathbb{E}\|{v}_t^s\|^2  \nonumber
\\ \overset{(ii)}\leq &  \left(\frac{2\eta^2 L^2 m}{B} + \frac{2}{\alpha}+2\right)\sum_{s=1}^S \sum_{t=0}^{m-1}\mathbb{E}\|{v}_t^s\|^2+\frac{4Sm\epsilon}{\alpha} 
\end{align}
where (i) follows from Lemma~\ref{le:spider} and~\eqref{eq:bsn} and (ii) follows from the definition of $\beta_s$ for $s=1,...,S$. 
Let $\phi = \frac{1}{2}-\frac{1}{2\alpha}-\frac{L\eta}{2}- \frac{\eta ^2 L^2 m}{2B}$, $\psi=\frac{2\eta^2 L^2 m}{B} + \frac{2}{\alpha}+2 $ and $K=Sm$. 
 Then, combining~\eqref{eq:gg1} and~\eqref{eq:aps}, we finish the proof. 
 
\subsection{Proof of Corollary~\ref{co:spider}}
Recall that $1\leq B\leq n^{1/2}\wedge \epsilon^{-1/2}$, $m = (n \wedge \frac{1}{\epsilon})B^{-1}$, 
$\eta =\frac{1}{4L}\sqrt{\frac{B}{m}}$ and $c_\beta, c_\epsilon\geq 16$. Then, we have $\alpha=16$, 
$m \geq  n^{1/2}\wedge \epsilon^{-1/2}\geq B$ and  $\eta \leq \frac{1}{4L}$. Thus, we obtain
\begin{align*}
\phi = \frac{1}{2}-\frac{1}{2\alpha}-\frac{L\eta}{2}- \frac{\eta ^2 L^2 m}{2B} \geq \frac{5}{16}>\frac{1}{4} \text{ and } \psi \leq \frac{9}{4},
\end{align*}
which, in conjunction with Theorem~\ref{th_spider}, implies that 
\begin{align*}
\mathbb{E}\|\nabla f({x}_\zeta)\|^2\leq  \frac{36L\sqrt{m}(f({x}_0)- f^*)}{ \sqrt{B}K} + \frac{17}{32}\epsilon.
\end{align*}
Thus, to achieve $\mathbb{E}\|\nabla f({x}_\zeta)\|^2 <\epsilon$, AbaSPIDER requires at most $\frac{384L\sqrt{m}(f({x}_0)- f^*)}{5\sqrt{B}\epsilon}=\Theta \big(\frac{\sqrt{m}}{\sqrt{B}\epsilon}\big)$ iterations. Then, the total number of SFO calls is given by 
\begin{align*}
\sum_{s=1}^S\min\{c_\beta\sigma^2\beta_s^{-1}, c_\epsilon\sigma^2\epsilon^{-1},n\} + KB \leq &S (c_\epsilon \sigma^2 \epsilon^{-1} \wedge n) + KB \leq  \mathcal{O}\left( \frac{\epsilon^{-1}\wedge n }{\epsilon\sqrt{mB}} + \frac{\sqrt{mB}}{\epsilon}   \right), 
\end{align*}
which, in conjunction with $mB = n \wedge \frac{1}{\epsilon}$, finishes the proof. 
\end{proof}

\section{Proofs for Results in \Cref{rlyoudiandiao}}\label{em_verify}
\subsection{Useful Lemmas}
In this section, we provide some useful lemmas. The following two lemmas follow directly from Assumptions in Subsection~\ref{assum:rl}.
\begin{lemma}[\cite{Papini2018}] \label{bounded_from_papini}
	Under Assumptions \ref{g_assumption} and \ref{assumption}, the following holds:
	\begin{itemize} 
		\item[(i)]
		$\nabla J$ is $L$- Lipschitz, i.e., for any $\theta_1, \theta_2 \in \mathbb{R}^d$:
		$\norml{\nabla J(\theta_{1}) - \nabla J(\theta_{2})} \leq L \norml{\theta_{1} - \theta_{2}}$.
		\item[(ii)]
		$g(\tau|\theta)$ is Lipschitz continuous with Lipschitz constant $L_g$, i.e., for any trajectory $\tau \in \mathcal{T}$:
		\begin{align*} 
		\norml{g(\tau|\theta_1) - g(\tau|\theta_2)} \leq L_g \norml{\theta_1 - \theta_2} .
		\end{align*}
		\item [(iii)]
		$g(\tau|\theta)$ and $\nabla \log(p(\tau|\theta))$ are  bounded, i.e., there exist positive constants $0 \leq \Gamma, M < \infty$ such that for any $\tau \in \mathcal{T}$ and $\theta \in \Theta$:
		\begin{align*}
		\norml{\nabla \log(p(\tau|\theta))}^2 \leq M  \quad \text{ and } \quad \norml{g(\tau|\theta)}^2 \leq \Gamma.
		\end{align*} 
	\end{itemize}	
\end{lemma} 

\begin{lemma}[\cite{xu2019sample,xu2019improved} Lemma A.1]\label{var_assumption}
	For any $\theta_1,\theta_2\in\mathcal{R}^d$, let $\omega \big(\tau|\theta_1,\theta_2\big)=p(\tau |\theta_1)/p(\tau |\theta_2)$. Under  Assumptions \ref{assumption} and \ref{assumption_bound_variance}, it holds that
	\begin{align*}
	\E_{\tau \sim p(\cdot|\theta_1)} \norml{1 - \frac{p(\tau|\theta_2)}{p(\tau|\theta_1)}}^2  = \Var\big(\omega \big(\tau|\theta_1,\theta_2\big)\big)\leq \alpha \|\theta_1-\theta_2\|_2^2,
	\end{align*}
	where $\alpha$ is a positive constant.
\end{lemma}

The following lemma captures an important property for the trajectory gradients, and its proof follows directly from Lemma~\ref{var_assumption}.  
%
\begin{lemma} \label{lispchitz_g}
	Under Assumptions \ref{g_assumption}, \ref{assumption}, and \ref{assumption_bound_variance} the following inequality holds for any $\theta_1, \theta_2 \in \mathcal{R}^d$,
	\begin{align*}
	\E_{\tau \sim p(\cdot|\theta_1) } \norml{g(\tau|\theta_1) - \omega(\tau|\theta_1, \theta_2)g(\tau|\theta_2)  }^2 \leq Q \norml{\theta_1 - \theta_2}^2,
	\end{align*}
where the importance sampling function $\omega(\tau | \theta_{1}, \theta_{2}): =  {p(\tau|\theta_2)}/{p(\tau|\theta_1)}$, and the constant $Q \defeq  2(L_g^2+ \Gamma \alpha)$ with constants $L_g, \Gamma$ and $\alpha$ given in Lemmas~\ref{bounded_from_papini} and~\ref{var_assumption}. 
\end{lemma}

\subsection{Proof of \Cref{SVRPG_conv}}
In this section, we provide the convergence analysis for AbaSVRPG. 
To simplify notations, 
we use $\E_k[\cdot]$ to denote
the expectation operation conditioned on all the randomness before $\theta_{k}$, i.e.,
$\E [\cdot | \theta_{0}, \cdots, \theta_{k}]$ and $n_k = \floor{k/m}\times m$. 


To prove the convergence of AbaSVRPG, we first present a general iteration analysis for an algorithm with the update rule taking the form of $\theta_{k+1} = \theta_{k} + \eta v_k$, for $k=0, 1, \cdots$. The proof of \Cref{beginWith} can be found in \Cref{proof_of_lemmas}. 
\begin{lemma} \label{beginWith}
	Let $\nabla J$ be $L$-Lipschitz, and $\theta_{k+1} = \theta_{k+1} + \eta v_k$. Then, the following inequality holds:
	\begin{align*}
	\E{J(\theta_{k+1}) } - \E J(\theta_{k}) \geq   {\left(\frac{\eta}{2} - \frac{L\eta^2}{2}\right) \E \norml{v_k}^2}  - \frac{\eta}{2} \E \norml{v_k - \nabla J( \theta_{k})}^2.
	\end{align*}
\end{lemma} 

Since, we do not specify the exact form of $v_k$,   \Cref{beginWith} is applicable to various algorithms such as AbaSVRPG and AbaSPIDER-PG with the same type of update rules.

We next present the variance bound of AbaSVRPG.
\begin{proposition} \label{variance_bound_SVRPG}
	Let Assuptions~\ref{g_assumption}, \ref{assumption}, and \ref{assumption_bound_variance}  hold. Then,  for $k = 0, \cdots K$, the variance of the gradient estimator $v_k$ of AbaSVRPG can be bounded as 
	\begin{align*} 
	\E  \norml{v_{k} - \nabla J( \theta_{k})}^2  \leq &(k-n_k) \frac{Q\eta^2}{B}  \sum_{i = n_k}^{k } \E \norml{ v_i }^2 +  \E  \norml{v_{n_k} -\nabla J(\theta_{n_k})}^2,
	\end{align*}
\end{proposition} 
where $\norml{v_{i}} = 0$ for $i = -1, \cdots, -m$ for simple notations. 

\begin{proof}[Proof of \Cref{variance_bound_SVRPG}]
	To  bound the variance $\E \norml{v_k - \nabla J( \theta_{k})}^2$, it is sufficient to bound ${\E}_k \norml{v_k - \nabla J( \theta_{k})}^2 $ since  by the tower property of expectation we have $\E \norml{v_k - \nabla J( \theta_{k})}^2 = \E {\E}_k \norml{v_k - \nabla J( \theta_{k})}^2 $. Thus, we first bound ${\E}_k \norml{v_k - \nabla J( \theta_{k})}^2 $ for the case with $\mod(k,m) \neq 0$, and then generalize it to  the case with $\mod(k,m) = 0$.
	\begin{align*}
	{\E}_k &\norml{v_k - \nabla J( \theta_{k})}^2\\
	&\numequ{i}  {\E}_k \norml{\frac{1}{B} \sum_{i=1}^{B} g(\tau_i|\theta_k) - \frac{1}{B} \sum_{i=1}^{B} \omega(\tau_i| \theta_{k}, \tilde{\theta}) g(\tau_i|\tilde{\theta} ) + \tilde{v}- \nabla J( \theta_{k})}^2 \\
	&= {\E}_k \norml{\frac{1}{B} \sum_{i=1}^{B} g(\tau_i|\theta_k) - \frac{1}{B} \sum_{i=1}^{B} \omega(\tau_i| \theta_{k}, \tilde{\theta}) g(\tau_i|\tilde{\theta} ) + \nabla J( \tilde{\theta}) - \nabla J( \theta_{k}) + \tilde{v} -\nabla J( \tilde{\theta}) }^2 \\
	&={\E}_k \norml{\frac{1}{B} \sum_{i=1}^{B} g(\tau_i|\theta_k) - \frac{1}{B} \sum_{i=1}^{B} \omega(\tau_i| \theta_{k}, \tilde{\theta}) g(\tau_i|\tilde{\theta} ) + \nabla J( \tilde{\theta}) - \nabla J( \theta_{k})  }^2  + {\E}_k \norml{\tilde{v} -\nabla J( \tilde{\theta})}^2 \\
	&\qquad \quad + 2 {\E}_k\inner{\frac{1}{B} \sum_{i=1}^{B} g(\tau_i|\theta_k) - \frac{1}{B} \sum_{i=1}^{B} \omega(\tau_i| \theta_{k}, \tilde{\theta}) g(\tau_i|\tilde{\theta} ) + \nabla J( \tilde{\theta}) - \nabla J( \theta_{k})  }{\tilde{v} -\nabla J( \tilde{\theta})} \\
	&\numequ{ii}{\E}_k \norml{\frac{1}{B} \sum_{i=1}^{B} g(\tau_i|\theta_k) - \frac{1}{B} \sum_{i=1}^{B} \omega(\tau_i| \theta_{k}, \tilde{\theta}) g(\tau_i|\tilde{\theta} ) + \nabla J( \tilde{\theta}) - \nabla J( \theta_{k})  }^2  + {\E}_k \norml{\tilde{v} -\nabla J( \tilde{\theta})}^2 \\
	&\numleq{iii} \underset{\tau \sim p(\cdot|\theta_k)}{{\E}_k }\frac{1}{B}  \norml{  g(\tau |\theta_k) -     \omega(\tau | \theta_{k}, \tilde{\theta}) g(\tau |\tilde{\theta} )   + \nabla J( \tilde{\theta}) - \nabla J( \theta_{k})}^2  + {\E}_k \norml{\tilde{v} -\nabla J( \tilde{\theta})}^2 \\
	&\numleq{iv} \underset{\tau \sim p(\cdot|\theta_k)}{{\E}_k }\frac{1}{B}  \norml{  g(\tau |\theta_k) -     \omega(\tau | \theta_{k}, \tilde{\theta}) g(\tau |\tilde{\theta} )  }^2  + {\E}_k \norml{\tilde{v} -\nabla J( \tilde{\theta})}^2 \\ 
	&\numleq{v} \frac{Q}{B} \norml{\theta_{k} - \tilde{\theta}}^2 + {\E}_k \norml{\tilde{v} -\nabla J( \tilde{\theta})}^2 = \frac{Q}{B} \norml{\theta_{k} -\theta_{k-1} + \theta_{k-1} \cdots \theta_{n_k}}^2 + {\E}_k \norml{\tilde{v} -\nabla J( \tilde{\theta})}^2 \\
	&\numleq{vi} \frac{Q}{B} (k-n_k) \sum_{i = n_k}^{k-1} \norml{\theta_{i+1} -\theta_{i } }^2 + {\E}_k \norml{\tilde{v} -\nabla J( \tilde{\theta})}^2, 
	\end{align*}
	where (i) follows from the definition of $v_k$ in \Cref{SVRPG}, (ii) follows from the fact that $\E_k\left[ \frac{1}{B} \sum_{i=1}^{B} g(\tau_i|\theta_k) - \frac{1}{B} \sum_{i=1}^{B} \omega(\tau_i| \theta_{k}, \tilde{\theta}) g(\tau_i|\tilde{\theta} ) + \nabla J( \tilde{\theta}) - \nabla J( \theta_{k}) \right]= 0$, and thus given $\theta_{k}, \cdots, \tilde{\theta}$, the expectation of the inner product is $0$, (iii) follows from \Cref{batch_variance}, (iv) follows from the fact that $\Var(X) \leq \E \norml{X}^2$, (v) follows from \Cref{lispchitz_g} we provide in \Cref{proof_of_lemmas}, and (vi) follows from the vector inequality that $\norml{\sum_{i=1}^{m} \theta_{i}}^2 \leq m \sum_{i=1}^{m} \norml{\theta_i}^2$.
	
	Therefore, we have
	\begin{align*}
	\E  \norml{v_k - \nabla J( \theta_{k})}^2 &= \E {\E}_k \norml{v_k - \nabla J( \theta_{k})}^2 \\
	& \leq \frac{Q}{B} (k-n_k) \sum_{i = n_k}^{k-1} \E \norml{\theta_{i+1} -\theta_{i } }^2 + \E \norml{\tilde{v} -\nabla J( \tilde{\theta})}^2 \\
	& \leq \frac{Q}{B} (k-n_k) \sum_{i = n_k}^{k } \E \norml{\theta_{i+1} -\theta_{i } }^2 + \E \norml{\tilde{v} -\nabla J( \tilde{\theta})}^2  \\
	& \leq \frac{Q}{B} (k-n_k) \eta^2 \sum_{i = n_k}^{k } \E \norml{ v_i }^2 + \E \norml{\tilde{v} -\nabla J( \tilde{\theta})}^2  \\ 
	&\numequ{i} (k-n_k)  \frac{Q\eta^2 }{B} \sum_{i = n_k}^{k } \E \norml{ v_i }^2 + \E \norml{v_{n_k} -\nabla J( \theta_{n_k})}^2, 
	\end{align*}
	where (i) follows from the fact that at iteration $k$, $\tilde{v} = v_{n_k}$ and $\tilde{\theta} = \theta_{n_k}$. It is also straightforward to check that the above inequality  holds for any $k$ with $\mod(k,m) = 0$.
\end{proof}

\subsection*{Proof of \Cref{SVRPG_conv}} \label{C_SVRPG} 
	Since in \Cref{SVRPG}, $\nabla J$ is $L$-Lipschitz, and $\theta_{k+1} = \theta_{k+1} + \eta v_k$, we obtain the following inequality directly from \Cref{beginWith}:
	\begin{align} \label{svrpg_iteration}
	\E{J(\theta_{k+1}) } - \E J(\theta_{k}) \geq   {\left(\frac{\eta}{2} - \frac{L\eta^2}{2}\right) \E \norml{v_k}^2}  -  {\frac{\eta}{2} \E \norml{v_k - \nabla J( \theta_{k})}^2} .
	\end{align}

	By \Cref{variance_bound_SVRPG}, we have following variance bound:
	\begin{align}
	\E  \norml{v_k - \nabla J( \theta_{k})}^2  & \leq (k-n_k)\frac{Q\eta^2}{B}   \sum_{i = n_k}^{k } \E \norml{ v_i }^2 + \E \norml{v_{n_k} -\nabla J( \theta_{n_k})}^2 \label{svrpg_var_bound_0}
	\end{align}

	Moreover,  for $\mod(k,m) = 0$, we obtain
	\begin{align*}
	\mathbb{E} \|v_k - \nabla f(x_k)\|^2 &= \mathbb{E} \left\| \frac{1}{N}\sum_{i=1}^{N} \nabla g(\tau_i|\theta_{k}) - \nabla J(\theta_{k}) \right\|^2\\
	& \numequ{i}  \frac{1}{N } \E_{\tau \sim p(\cdot|\theta_k) }\left\|   \nabla g(\tau |\theta_{k}) - \nabla J(\theta_{k}) \right\|^2\overset{(ii )}{\leqslant} \frac{\sigma^2}{N} \\
	&\overset{(iii)}{\leqslant}  \frac{  \beta  }{ \alpha m } \sum_{i=n_k-m}^{n_k-1} \norml{v_i}^2   + \frac{\epsilon}{\alpha}.  \numberthis \label{SVRPG_para_1}
	\end{align*}
	where (i) follows from \Cref{batch_variance}, (ii) follows from \Cref{assumption_bound_variance}, and (iii) follows from the fact that $$N = \frac{\alpha \sigma^2}{\frac{\beta}{m} \sum_{i=n_k-m}^{n_k-1} \norml{v_i}^2 + \epsilon},$$  where $\alpha > 0$ and $\beta \geqslant 0$.
	
	Plugging \eqref{SVRPG_para_1} into \eqref{svrpg_var_bound_0}, we obtain
	\begin{align*}
	\E  \norml{v_k - \nabla J( \theta_{k})}^2  & \leq  (k-n_k)  \frac{Q \eta^2}{B} \sum_{i = n_k}^{k } \E \norml{ v_i }^2 + \frac{  \beta   }{ \alpha m }\sum_{i=n_k-m}^{n_k-1} \norml{v_i}^2   + \frac{\epsilon}{\alpha}. \numberthis \label{svrpg_var_bound} 
	\end{align*}
	
	Plugging \eqref{svrpg_var_bound} into \eqref{svrpg_iteration}, we obtain
	{ \small
	\begin{align*}
	\E{J(\theta_{k+1}) } - \E J(\theta_{k}) \geq   {\left(\frac{\eta}{2} - \frac{L\eta^2}{2}\right) \E \norml{v_k}^2}  - \frac{Q\eta^3}{2B} (k-n_k)  \sum_{i = n_k}^{k } \E \norml{ v_i }^2  -  \frac{  \eta \beta   }{ 2\alpha m }\sum_{i=n_k-m}^{n_k-1} \norml{v_i}^2   - \frac{\epsilon \eta}{2 \alpha} .
	\end{align*}
	}

	We note that for a given $k$, any iteration $ n_k \leq i \leq k$ shares the same $\tilde{\theta}$, and all their corresponding $n_i$ satisfies $n_i = n_k$. Thus, take the summation of the above  inequality over $k$ from $n_k$ to $k$, we obtain
	\begin{align*}
	\E&{J(\theta_{k+1}) } -  \E J(\theta_{n_k})\\
	&\geq   {\left(\frac{\eta}{2} - \frac{L\eta^2}{2}\right) \sum_{i=n_k}^{k} \E \norml{v_i}^2}  - \frac{Q\eta^3}{2B} \sum_{i=n_k}^{k} (i-n_k)  \sum_{j = n_k}^{i } \E \norml{ v_j }^2  -   \sum_{i=n_k}^{k} \frac{  \eta \beta    }{ 2\alpha m }\sum_{j=n_k-m}^{n_k-1} \norml{v_j}^2  -   \sum_{i=n_k}^{k}  \frac{\epsilon \eta}{2 \alpha}\\ 
	&\geq   {\left(\frac{\eta}{2} - \frac{L\eta^2}{2}\right) \sum_{i=n_k}^{k} \E \norml{v_i}^2}  - \frac{Q\eta^3}{2B} \sum_{i=n_k}^{k} (k-n_k)  \sum_{j = n_k}^{k } \E \norml{ v_j }^2  -     \sum_{i=n_k}^{k} \frac{  \eta \beta    }{ 2\alpha m }\sum_{j=n_k-m}^{n_k-1} \norml{v_j}^2  -   \sum_{i=n_k}^{k}  \frac{\epsilon \eta}{2 \alpha}\\
	&=   {\left(\frac{\eta}{2} - \frac{L\eta^2}{2}\right) \sum_{i=n_k}^{k} \E \norml{v_i}^2}  - \frac{Q\eta^3 (k-n_k)  (k-n_k+1) }{2B}   \sum_{j = n_k}^{k } \E \norml{ v_j }^2  -  \sum_{i=n_k}^{k} \frac{  \eta \beta   }{ 2\alpha m }\sum_{j=n_k-m}^{n_k-1} \norml{v_j}^2   -   \sum_{i=n_k}^{k}  \frac{\epsilon \eta}{2 \alpha}\\
	&=    \left(\frac{\eta}{2} - \frac{L\eta^2}{2} - \frac{Q\eta^3 (k-n_k)  (k-n_k+1) }{2B} \right)     \sum_{i = n_k}^{k } \E \norml{ v_i }^2   -   \sum_{i=n_k}^{k} \frac{  \eta \beta   }{ 2\alpha m }  \sum_{j=n_k-m}^{n_k-1} \norml{v_j}^2 -   \sum_{i=n_k}^{k}  \frac{\epsilon \eta}{2 \alpha}\\
	&\numgeq{i}  \left(\frac{\eta}{2} - \frac{L\eta^2}{2} - \frac{Q\eta^3m^2}{2B} \right)     \sum_{i = n_k}^{k } \E \norml{ v_i }^2   -   \sum_{i=n_k}^{k} \frac{  \eta \beta    }{ 2\alpha m } \sum_{j=n_k-m}^{n_k-1} \norml{v_j}^2 -   \sum_{i=n_k}^{k}  \frac{\epsilon \eta}{2 \alpha} \\
	&\numequ{ii} \phi  \sum_{i = n_k}^{k } \E \norml{ v_i }^2   -  \frac{\eta \beta (k - n_k + 1)}{2 \alpha m}  \sum_{i=n_k-m}^{n_k-1} \norml{v_i}^2    -   \sum_{i=n_k}^{k}  \frac{\epsilon \eta}{2 \alpha} \\
	&\numgeq{iii} \phi  \sum_{i = n_k}^{k } \E \norml{ v_i }^2   -  \frac{\eta \beta  }{2 \alpha  }  \sum_{i=n_k-m}^{n_k-1} \norml{v_i}^2    -   \sum_{i=n_k}^{k}  \frac{\epsilon \eta}{2 \alpha},  \numberthis \label{SVRPG_final}
	\end{align*}
	where (i) follows from the fact that $k-n_k < k-n_k +1 \leq m$,   (ii) follows from the fact  that $\phi \defeq \left(\frac{\eta}{2} - \frac{L\eta^2}{2} - \frac{Q\eta^3m^2}{2B} \right)$, and (iii) follows because $(k - n_k + 1)/m \leqslant 1$.
	
	Now, we are ready to bound $J(\theta_{K+1}) - J(\theta_{0})$.
	\begin{align*}
	\E &J(\theta_{K+1})  - \E J(\theta_{0}) \\
	&= \E   J(\theta_{K+1}) -\E J(\theta_{n_K}) +  \E J(\theta_{n_K}) \cdots + \E J(\theta_{m}) - \E J(\theta_{0}) \\
	&\numgeq{i}   \phi  \sum_{i = n_K}^{K } \E \norml{ v_i }^2  -      \frac{\eta \beta  }{2 \alpha  }  \sum_{i=n_K-m}^{n_K-1} \norml{v_i}^2    -   \sum_{i=n_K}^{K}  \frac{\epsilon \eta}{2 \alpha} +  \cdots + \phi  \sum_{i = 0}^{m-1} \E \norml{ v_i }^2  -     \frac{\eta \beta  }{2 \alpha  }  \sum_{i= -m}^{-1} \norml{v_i}^2    -   \sum_{i=0}^{m}  \frac{\epsilon \eta}{2 \alpha} \\
	&\numgeq{ii}  \phi  \sum_{i = 0}^{K }\E \norml{ v_i }^2 -     \frac{\eta \beta  }{2 \alpha  }  \sum_{i=0}^{n_K-1} \norml{v_i}^2    -   \sum_{i=0}^{K}  \frac{\epsilon \eta}{2 \alpha} \\
	&\geqslant \left(\phi  -     \frac{\eta \beta  }{2 \alpha  }  \right) \sum_{i = 0}^{K }\E \norml{ v_i }^2  -   \sum_{i=0}^{K}  \frac{\epsilon \eta}{2 \alpha},
	\end{align*}
	where (i) follows from \eqref{SVRPG_final}, and (ii) follows from the fact that we define $\norml{v_{-1}} = \cdots =  \norml{v_{-m}} = 0$.
	
	Thus, we obtain
	\begin{align*}
	 \left(\phi  -     \frac{\eta \beta  }{2 \alpha  }  \right)  \sum_{i = 0}^{K }\E \norml{ v_i }^2 & \leq 	\E J(\theta_{K+1}) - \E J(\theta_{0}) +    \sum_{i=0}^{K}  \frac{\epsilon \eta}{2 \alpha}\\
	& \numleq{i}  J(\theta^*) -  J(\theta_{0}) +    \sum_{i=0}^{K}  \frac{\epsilon \eta}{2 \alpha}, 
	\end{align*}
	where (i) follows because $\theta^* \defeq \arg  \underset{\theta \in \mathbb{R}^d}{\max} J(\theta) $.  Here, we  assume  $\left(\phi  -     \frac{\eta \beta  }{2 \alpha  }  \right) >0$  to continue our proof. Such an assumption can be satisfied by parameter tuning  as shown in \eqref{SVRPG_para_0}. Therefore, we obtain
	\begin{align}
	\sum_{i = 0}^{K }\E \norml{ v_i }^2 \leq \left(\phi  -     \frac{\eta \beta  }{2 \alpha  }  \right)^{-1}\left( {J(\theta^*) -  J(\theta_{0})} +    \sum_{i=0}^{K}  \frac{\epsilon \eta}{2 \alpha}\right).   \label{SVRPG_finite_bound}
	\end{align}
	
	With \eqref{SVRPG_finite_bound}, we next bound the gradient norm, i.e., $\norml{\nabla J(\theta_{\xi})}$, of the output of AbaSVRPG. Observe that
	\begin{align}
	\mathbb{E}  \|\nabla J(\theta_\xi)\|^2  = \mathbb{E}  \|\nabla J(\theta_\xi) - v_\xi + v_\xi\|^2 \leq 2\mathbb{E}  \|\nabla J(\theta_\xi) - v_\xi \|^2 + 2\mathbb{E}  \| v_\xi \|^2 \label{SVRPG_final_0}.
	\end{align}
	Therefore, it is sufficient to bound the two terms on the right hand side of the above inequality. First, note that
	\begin{align}
	\mathbb{E} \|v_\xi\|^2  \numequ{i} \frac{1}{K+1}\sum_{i=0}^{K} \mathbb{E} \|v_i\|^2 \numleq{ii}  \left(\phi  -     \frac{\eta \beta  }{2 \alpha  }  \right)^{-1}\left( \frac{J(\theta^*) -  J(\theta_{0})}{K+1} +     \frac{\epsilon \eta}{2 \alpha}\right), \label{SVRPGeq: 5}
	\end{align}
	where (i) follows from the fact that $\xi$ is selected uniformly at random from $\{0,\ldots,K\}$, and (ii) follows from \eqref{SVRPG_finite_bound}. On the other hand, we observe that
	\begin{align*}
	&\quad \mathbb{E}   \|\nabla J(\theta_\xi)-v_\xi\|^2\\
	  &\numequ{i} \frac{1}{K+1} \sum_{k = 0}^{K} \mathbb{E}  \|\nabla J(\theta_k)-v_k\|^2 \\ 
	&\numleq{ii}  \frac{1}{K+1} \sum_{k = 0}^{K}\left( (k-n_k)  \frac{Q \eta^2}{B} \sum_{i = n_k}^{k } \E \norml{ v_i }^2 + \frac{  \beta  }{ \alpha m }\sum_{i=n_k-m}^{n_k-1} \norml{v_i}^2    + \frac{\epsilon}{\alpha}\right) \\
	&\numequ{iii}   \frac{Q \eta^2 m}{B(K+1)} \sum_{k = 0}^{K}    \sum_{i = n_k}^{k } \E \norml{ v_i }^2  + \frac{\beta}{\alpha m(K+1)} \sum_{k = 0}^{K} \sum_{i=n_k-m}^{n_k-1} \norml{v_i}^2  + \frac{\epsilon}{\alpha} \\
	&\numequ{iv}  \frac{Q \eta^2 m}{B(K+1)} \left(\sum_{k = 0}^{m-1}    \sum_{i = 0}^{k } \E \norml{ v_i }^2 + \cdots + \sum_{k = n_K}^{K}    \sum_{i = n_K}^{k } \E \norml{ v_i }^2 \right)  + \frac{\beta}{\alpha m(K+1)} \sum_{k = 0}^{K} \sum_{i=n_k-m}^{n_k-1} \norml{v_i}^2  + \frac{\epsilon}{\alpha} \\
	&\leqslant \frac{Q \eta^2 m}{B(K+1)} \left(\sum_{k = 0}^{m-1}    \sum_{i = 0}^{m-1 } \E \norml{ v_i }^2 + \cdots + \sum_{k = n_K}^{K}    \sum_{i = n_K}^{K } \E \norml{ v_i }^2 \right)  + \frac{\beta}{\alpha m(K+1)} \sum_{k = 0}^{K} \sum_{i=n_k-m}^{n_k-1} \norml{v_i}^2  + \frac{\epsilon}{\alpha} \\
	&\numleq{v} \frac{Q \eta^2 m^2}{B(K+1)}  \sum_{i = 0}^{K } \E \norml{ v_i }^2    + \frac{\beta}{\alpha m(K+1)} \sum_{k = 0}^{K} \sum_{i=n_k-m}^{n_k-1} \norml{v_i}^2  + \frac{\epsilon}{\alpha} \\
	&\numleq{vi} \frac{Q \eta^2 m^2}{B(K+1)}  \sum_{i = 0}^{K } \E \norml{ v_i }^2    + \frac{\beta}{\alpha m(K+1)} \left(\sum_{k = 0}^{m-1} \sum_{i=-m}^{-1} \norml{v_i}^2 + \cdots + \sum_{k = n_K}^{K} \sum_{i=n_K-m}^{n_K-1} \norml{v_i}^2 \right) + \frac{\epsilon}{\alpha} \\
	&\numleq{vii}\frac{Q \eta^2 m^2}{B(K+1)}  \sum_{i = 0}^{K } \E \norml{ v_i }^2    + \frac{\beta}{\alpha  (K+1)}  \sum_{i = -m}^{n_K-1}   \norml{v_i}^2   + \frac{\epsilon}{\alpha} \\
	&\numleq{viii}   \frac{1}{K+1} \left(\frac{Q \eta^2 m^2}{B} + \frac{\beta}{\alpha} \right) \sum_{i = 0}^{K } \E \norml{ v_i }^2    + \frac{\epsilon}{\alpha} \\ 
	&\numleq{viiii} \frac{1}{K+1}  \left(\frac{Q \eta^2 m^2}{B} + \frac{\beta}{\alpha} \right)  
	\left(\phi  -     \frac{\eta \beta  }{2 \alpha  }  \right)^{-1}\left( {J(\theta^*) -  J(\theta_{0})} +    \sum_{i=0}^{K}  \frac{\epsilon \eta}{2 \alpha}\right)    + \frac{\epsilon}{\alpha} \\ 
	&=  \left(\frac{Q \eta^2 m^2}{B} + \frac{\beta}{\alpha} \right)  
	\left(\phi  -     \frac{\eta \beta  }{2 \alpha  }  \right)^{-1} \frac{\left( {J(\theta^*) -  J(\theta_{0})} \right)  }{K+1}   + \frac{\epsilon}{\alpha}  \left(1 + \frac{\eta}{2} \left(\frac{Q \eta^2 m^2}{B} + \frac{\beta}{\alpha} \right)   \left(\phi  -     \frac{\eta \beta  }{2 \alpha  }  \right)^{-1}\right) \numberthis \label{SVRPG_final_2}
	\end{align*}
	where (i) follows from the fact that $\xi$ is selected uniformly at random from $\{0,\ldots,K\}$, (ii) follows from \eqref{svrpg_var_bound}, (iii) follows from the fact that $k-n_k \leq m$, (iv) follows from the fact that for $n_k \leqslant k \leqslant n_k + m - 1$, $n_i = n_k$. (v) follows from $\sum_{k = n_k}^{n_k + m-1}    \sum_{i = n_k}^{n_k + m-1 } \E \norml{ v_i }^2 = m \sum_{i = n_k}^{n_k + m-1 } \E \norml{ v_i }^2$, (vi) follows from the same reasoning as in (iv), 	(vii) follows from $\sum_{k = n_k}^{n_k + m-1}    \sum_{i = n_k-m}^{ n_k-1 } \E \norml{ v_i }^2 = m \sum_{i = n_k-m}^{ n_k-1 } \E \norml{ v_i }^2$, (viii) follows from $\norml{v_{-1}} = \cdots =  \norml{v_{-m}} = 0$, and (viiii) follows from \cref{SVRPG_finite_bound}.  
	
	Substituting \eqref{SVRPGeq: 5}, \eqref{SVRPG_final_2} into \eqref{SVRPG_final_0}, we obtain
	\begin{align*}
	\mathbb{E} \|\nabla J(\theta_\xi)\|^2  &\le \frac{2}{K+1} \left(1 + \frac{Q \eta^2 m^2}{B} + \frac{\beta}{\alpha}\right)  \left(\phi  -     \frac{\eta \beta  }{2 \alpha  }  \right)^{-1}\left( {J(\theta^*) -  J(\theta_{0})} \right) \\
	& \qquad \qquad + \frac{2\epsilon}{\alpha} \left(  1 + \frac{\eta}{2} \left(1 + \frac{Q \eta^2 m^2}{B} + \frac{\beta}{\alpha} \right)   \left(\phi  -     \frac{\eta \beta  }{2 \alpha  }  \right)^{-1}\right)  \numberthis \label{converge_SVRPG}
	\end{align*}
	
\subsection{Proof of  \Cref{cor:SVRPG_sto}}
	
	Based on the parameter setting in \Cref{SVRPG_conv} that 
	\begin{align}
	\eta = \frac{1}{2L}, m =  \left( \frac{L^2\sigma^2}{Q\epsilon} \right)^{\frac{1}{3}} , B= \left( \frac{Q\sigma^4}{L^2\epsilon^2} \right)^{\frac{1}{3}}, \alpha = 48, \text{ and }   \beta = 6,  \label{SVRPG_para_2}
	\end{align} 
	we obtain
	\begin{align}
	\phi - \frac{\eta \beta  }{2 \alpha  }= \left(\frac{1}{4L} - \frac{1}{8L} - \frac{1}{16L} \right) - \frac{1}{32L}= \frac{1}{32L} > 0. \label{SVRPG_para_0}
	\end{align}

	Plugging \eqref{SVRPG_para_2} and \eqref{SVRPG_para_0}   into \eqref{converge_SVRPG}, we obtain 
	\begin{align*}
	\E \norml{\nabla J(\theta_{\xi})}^2 \leq  \frac{88L}{K+1} \left( J(\theta^*) -  J(\theta_{0})\right)    +  \frac{\epsilon}{2}.
	\end{align*}
	Hence, AbaSVRPG converges at a rate of $\mathcal{O}(1/K)$. Next, we bound the STO complexity. To acheive $\epsilon$ accuracy, we need 
	\begin{align*}
	\frac{88L}{K+1} \left( J(\theta^*) -  J(\theta_{0})\right)  \leq \frac{\epsilon}{2},
	\end{align*}
	which gives
	\begin{align*}
	K = \frac{176L\left(J(\theta^*) -  J(\theta_{0})\right)}{\epsilon} .
	\end{align*}
	We note that for $\mod(k,m) =  0$, the outer loop batch size 
	$ N =  \frac{\alpha \sigma^2}{\frac{\beta}{m} \sum_{i=n_k-m}^{n_k-1} \norml{v_i}^2 + \epsilon} \leq \frac{\alpha \sigma^2}{\epsilon}$. Hence, the overall STO complexity is given by
	\begin{align*}
	K\times 2B + \sum_{k=0}^{n_K} \frac{\alpha \sigma^2}{\frac{\beta}{m} \sum_{i=km-m}^{km-1} \norml{v_i}^2 + \epsilon}   &\leqslant K\times 2B + \sum_{k=0}^{n_K}   \frac{\alpha \sigma^2}{\epsilon}  \leqslant K\times 2B + \ceil{\frac{K}{m}}\times \frac{\alpha \sigma^2}{\epsilon} \\
	  &\numleq{i} 2KB + \frac{K }{m} \frac{\alpha \sigma^2}{\epsilon}   +  \frac{\alpha \sigma^2}{\epsilon}   \\
	&\numequ{ii} \mathcal{O} \left(\left(\frac{L\left(J(\theta^*) -  J(\theta_{0})\right)}{\epsilon}\right) \left( \left( \frac{Q\sigma^4}{L^2\epsilon^2} \right)^{\frac{1}{3}} + \frac{\sigma^2}{\epsilon} \left( \frac{Q\epsilon}{L^2\sigma^2} \right)^{\frac{1}{3}}  \right) + \frac{\sigma^2}{\epsilon}  \right) \\
	&= \mathcal{O} \left(\left(\frac{L\left(J(\theta^*) -  J(\theta_{0})\right)}{\epsilon}\right)  \left( \frac{Q\sigma^4}{L^2\epsilon^2} \right)^{\frac{1}{3}}   + \frac{\sigma^2}{\epsilon}  \right) \\
	&= \mathcal{O} \left(  \epsilon^{-5/3} + \epsilon^{-1} \right),
	\end{align*}
	where (i) follows from the fact that $\ceil{\frac{K}{m}}\times N \leq \frac{KN}{m}  + N $, and (ii) follows from the parameters setting of $K,B, \text{ and }m$ in \eqref{SVRPG_para_2}.  

\subsection{Proof of \Cref{SpiderPG_conv}}\label{sec_variance_bound_SpiderPG} 
In this section, we provide the proof of AbaSPIDER-PG. We first bound the variance of AbaSPIDER-PG given in the following proposition. 

\begin{proposition} \label{variance_bound_SpiderPG}
	Let Assumptions~\ref{g_assumption}, \ref{assumption}, and \ref{assumption_bound_variance}  hold. For $k = 0,... K$, gradient estimator $v_k$ of AbaSPIDER-PG satisfies 
	\begin{align*} 
	\E  \norml{v_{k} - \nabla J( \theta_{k})}^2  \leq &   \frac{Q\eta^2}{B}  \sum_{i = n_k}^{k } \E \norml{ v_i }^2 
+  \E  \norml{v_{n_k} -\nabla J(\theta_{n_k})}^2.
	\end{align*}
\end{proposition}
Comparing \Cref{variance_bound_SpiderPG} and \Cref{variance_bound_SVRPG}, one can clearly see that AbaSPIDER-PG has a much smaller variance bound than AbaSVRPG, particularly as the inner loop iteration goes further (i.e., as $k$ increases). This is because AbaSVRPG uses the initial outer loop batch gradient to construct the gradient estimator in all inner loop iterations, so that the variance in the inner loop accumulates up as the iteration goes further. In contrast, AbaSPIDER-PG  avoids such a variance accumulation problem by continuously using the gradient information from the immediate preceding step, and hence has less variance during the inner loop iteration. 
\begin{proof}[Proof of \Cref{variance_bound_SpiderPG}]
	To  bound the variance $\E \norml{v_k - \nabla J( \theta_{k})}^2$, it is sufficient to bound ${\E}_k \norml{v_k - \nabla J( \theta_{k})}^2 $,  and then the tower property of expectation yields the desired result. Thus, we first bound ${\E}_k \norml{v_k - \nabla J( \theta_{k})}^2 $ for $\mod(k,m) \neq 0$, and then generalize it to   $\mod(k,m) = 0$.
	{
		\begin{align*}
		{\E}_k &\norml{v_k - \nabla J( \theta_{k})}^2\\
		&\numequ{i}  {\E}_k \norml{ \frac{1}{B} \sum_{i=1}^{B} g(\tau_i| \theta_k) - \frac{1}{B} \sum_{i=1}^{B} \omega(\tau_i| \theta_{k}, \theta_{k-1}) g(\tau_i| \theta_{k-1}) + v_{k-1} - \nabla J( \theta_{k})}^2 \\
		&= {\E}_k \norml{ \frac{1}{B} \sum_{i=1}^{B} g(\tau_i| \theta_k) - \frac{1}{B} \sum_{i=1}^{B} \omega(\tau_i| \theta_{k}, \theta_{k-1}) g(\tau_i| \theta_{k-1})  + \nabla J( \theta_{k-1}) - \nabla J( \theta_{k}) +   v_{k-1} -\nabla J(\theta_{k-1}) }^2 \\
		&={\E}_k \norml{\frac{1}{B} \sum_{i=1}^{B} g(\tau_i|\theta_k) - \frac{1}{B} \sum_{i=1}^{B} \omega(\tau_i| \theta_{k}, \theta_{k-1}) g(\tau_i|\theta_{k-1} ) + \nabla J(\theta_{k-1}) - \nabla J( \theta_{k})  }^2  + {\E}_k \norml{v_{k-1} -\nabla J( \theta_{k-1})}^2 \\
		&\qquad \quad + 2 {\E}_k\inner{\frac{1}{B} \sum_{i=1}^{B} g(\tau_i|\theta_k) - \frac{1}{B} \sum_{i=1}^{B} \omega(\tau_i| \theta_{k}, \theta_{k-1}) g(\tau_i|\theta_{k-1}) + \nabla J( \theta_{k-1}) - \nabla J( \theta_{k})  }{v_{k-1} -\nabla J(\theta_{k-1})} \\
		&\numequ{ii}{\E}_k \norml{\frac{1}{B} \sum_{i=1}^{B} g(\tau_i|\theta_k) - \frac{1}{B} \sum_{i=1}^{B} \omega(\tau_i| \theta_{k}, \theta_{k-1}) g(\tau_i|\theta_{k-1}) + \nabla J( \theta_{k-1}) - \nabla J( \theta_{k})  }^2  + {\E}_k \norml{v_{k-1} -\nabla J(\theta_{k-1})}^2 \\
		&\numleq{iii} \underset{\tau \sim p(\cdot|\theta_k)}{{\E}_k }\frac{1}{B}  \norml{  g(\tau |\theta_k) -     \omega(\tau | \theta_{k}, \theta_{k-1}) g(\tau |\theta_{k-1})   + \nabla J( \theta_{k-1}) - \nabla J( \theta_{k})}^2  + {\E}_k \norml{v_{k-1} -\nabla J( \theta_{k-1})}^2 \\
		&\numleq{iv} \underset{\tau \sim p(\cdot|\theta_k)}{{\E}_k }\frac{1}{B}  \norml{  g(\tau |\theta_k) -     \omega(\tau | \theta_{k}, \theta_{k-1}) g(\tau |\theta_{k-1} )  }^2  + {\E}_k \norml{v_{k-1} -\nabla J( \theta_{k-1})}^2 \\ 
		&\numleq{v} \frac{Q}{B} \norml{\theta_{k} - \theta_{k-1}}^2 + {\E}_k \norml{v_{k-1} -\nabla J( \theta_{k-1})}^2 \\
		&\numleq{vi} \frac{Q\eta^2}{B} \norml{v_{k-1}}^2 + {\E}_k \norml{v_{k-1} -\nabla J( \theta_{k-1})}^2 \numberthis \label{SpiderPG_per_iteration_0}
		\end{align*}
	}
	where (i) follows from the definition of $v_k$ in \Cref{SVRPG}, (ii) follows from the fact that $$\E_k\left[\frac{1}{B} \sum_{i=1}^{B} g(\tau_i|\theta_k) - \frac{1}{B} \sum_{i=1}^{B} \omega(\tau_i| \theta_{k}, \theta_{k-1}) g(\tau_i|\theta_{k-1}) + \nabla J( \theta_{k-1}) - \nabla J( \theta_{k})  \right]= 0,$$ thus given $\theta_{k}, \cdots, \theta_{0}$, the expectation of the inner product equals $0$, (iii) follows from \Cref{batch_variance}, (iv) follows from the fact that $\Var(X) \leq \E \norml{X}^2$, (v) follows from \Cref{lispchitz_g}, and (vi) follows because   $\theta_{k} = \theta_{k-1} + \eta v_{k-1}$.

	Therefore, we have
	\begin{align*}
	\E  \norml{v_k - \nabla J( \theta_{k})}^2 &\numequ{i} \E {\E}_k \norml{v_k - \nabla J( \theta_{k})}^2 \\
	&\numleq{ii} \frac{Q\eta^2}{B} \E \norml{v_{k-1}}^2 + \E \norml{v_{k-1} -\nabla J( \theta_{k-1})}^2, \numberthis \label{SpiderPG_var_bound_0} 
	\end{align*}
	where (i) follows from the tower property of expectation, and (ii) follows from \eqref{SpiderPG_per_iteration_0}.

	Telescoping \eqref{SpiderPG_var_bound_0} over $k$ from $n_k+1$ to $k$, we obtain
	\begin{align*}
	\E  \norml{v_k - \nabla J( \theta_{k})}^2 &= \sum_{i=n_k+1}^{k} \frac{Q\eta^2}{B} \E \norml{v_{i-1}}^2 + \E \norml{v_{n_k} -\nabla J( \theta_{n_k})}^2\\
	& \leq \sum_{i=n_k}^{k} \frac{Q\eta^2}{B} \E \norml{v_{i}}^2 + \E \norml{v_{n_k} -\nabla J( \theta_{n_k})}^2. 
	\end{align*}
	It is straightforward to check that the above inequality also holds for any $k$ with $\mod(k,m) = 0$.
\end{proof}

\subsection*{Proof of \Cref{SpiderPG_conv}} \label{G_SpiderPG}
	Since in \Cref{SpiderPG}, $\nabla J$ is $L$-Lipschitz, and $\theta_{k+1} = \theta_{k+1} + \eta v_k$, we obtain the following inequality directly from \Cref{beginWith}:
	\begin{align} \label{SpiderPG_iteration}
	\E{J(\theta_{k+1}) } - \E J(\theta_{k}) \geq   {\left(\frac{\eta}{2} - \frac{L\eta^2}{2}\right) \E \norml{v_k}^2} -  {\frac{\eta}{2} \E \norml{v_k - \nabla J( \theta_{k})}^2} .  
	\end{align}
	
	By \Cref{variance_bound_SpiderPG}, we obtain
	\begin{align*}
	\E  \norml{v_k - \nabla J( \theta_{k})}^2  	& \leq \sum_{i=n_k}^{k} \frac{Q\eta^2}{B} \E \norml{v_{i}}^2 + \E \norml{v_{n_k} -\nabla J( \theta_{n_k})}^2   \numberthis \label{SpiderPG_var_bound_0_0}
	\end{align*}
	
	Moreover,  for $\mod(k,m) = 0$, we obtain
	\begin{align*}
	\mathbb{E} \|v_k - \nabla f(x_k)\|^2 &= \mathbb{E} \left\| \frac{1}{N}\sum_{i=1}^{N} \nabla g(\tau_i|\theta_{k}) - \nabla J(\theta_{k}) \right\|^2\\
	& \numequ{i}  \frac{1}{N } \E_{\tau \sim p(\cdot|\theta_k) }\left\|   \nabla g(\tau |\theta_{k}) - \nabla J(\theta_{k}) \right\|^2\overset{(ii )}{\leqslant} \frac{\sigma^2}{N} \\
	&\overset{(iii)}{\leqslant}  \frac{  \beta  }{ \alpha m } \sum_{i=n_k-m}^{n_k-1} \norml{v_i}^2   + \frac{\epsilon}{\alpha},  \numberthis \label{SpiderPG_para_1}
	\end{align*}
	where (i) follows from \Cref{batch_variance}, (ii) follows from \Cref{assumption_bound_variance}, and (iii) follows from the fact that $$N = \frac{\alpha \sigma^2}{\frac{\beta}{m} \sum_{i=n_k-m}^{n_k-1} \norml{v_i}^2 + \epsilon},$$  where $\alpha > 0$ and $\beta \geqslant 0$.
	
  	Plugging \eqref{SpiderPG_para_1} into \eqref{SpiderPG_var_bound_0_0}, we obtain
  	\begin{align*}
  	\E  \norml{v_k - \nabla J( \theta_{k})}^2  & \leq  \frac{Q \eta^2}{B} \sum_{i = n_k}^{k } \E \norml{ v_i }^2 + \frac{  \beta   }{ \alpha m } \sum_{i=n_k-m}^{n_k-1} \norml{v_i}^2  + \frac{\epsilon}{\alpha}. \numberthis \label{SpiderPG_var_bound} 
  	\end{align*}
  	
  	Plugging \eqref{SpiderPG_var_bound} into \eqref{SpiderPG_iteration}, we obtain
  	{ \small
  		\begin{align*}
  		\E{J(\theta_{k+1}) } - \E J(\theta_{k}) \geq   {\left(\frac{\eta}{2} - \frac{L\eta^2}{2}\right) \E \norml{v_k}^2}  - \frac{Q\eta^3}{2B}   \sum_{i = n_k}^{k } \E \norml{ v_i }^2  -  \frac{  \eta \beta    }{ 2\alpha m } \sum_{i=n_k-m}^{n_k-1} \norml{v_i}^2  - \frac{\epsilon \eta}{2 \alpha} .
  		\end{align*}
  	}

	We note that for a given $k$, any iteration $ n_k \leq i \leq k$, all their corresponding $n_i$ satisfies $n_i = n_k$. Thus,
	telescoping the above inequality over $k$ from $n_k$ to $k$, we obtain
	\begin{align*}
	\E&{J(\theta_{k+1}) } -  \E J(\theta_{n_k})\\
	&\geq   {\left(\frac{\eta}{2} - \frac{L\eta^2}{2}\right) \sum_{i=n_k}^{k} \E \norml{v_i}^2}  - \frac{Q\eta^3}{2B} \sum_{i=n_k}^{k}   \sum_{j = n_k}^{i } \E \norml{ v_j }^2  -   \sum_{i=n_k}^{k} \frac{  \eta \beta   }{ 2\alpha m } \sum_{i=n_k-m}^{n_k-1} \norml{v_i}^2  -   \sum_{i=n_k}^{k}  \frac{\epsilon \eta}{2 \alpha}\\ 
	&\geq   {\left(\frac{\eta}{2} - \frac{L\eta^2}{2}\right) \sum_{i=n_k}^{k} \E \norml{v_i}^2}  - \frac{Q\eta^3}{2B} \sum_{i=n_k}^{k}   \sum_{j = n_k}^{k } \E \norml{ v_j }^2  -     \sum_{i=n_k}^{k} \frac{  \eta \beta  }{ 2\alpha m }  \sum_{i=n_k-m}^{n_k-1} \norml{v_i}^2 -   \sum_{i=n_k}^{k}  \frac{\epsilon \eta}{2 \alpha}\\
	&=   {\left(\frac{\eta}{2} - \frac{L\eta^2}{2}\right) \sum_{i=n_k}^{k} \E \norml{v_i}^2}  - \frac{Q\eta^3    (k-n_k+1) }{2B}   \sum_{j = n_k}^{k } \E \norml{ v_j }^2  -  \sum_{i=n_k}^{k} \frac{  \eta \beta  }{ 2\alpha m } \sum_{i=n_k-m}^{n_k-1} \norml{v_i}^2  -   \sum_{i=n_k}^{k}  \frac{\epsilon \eta}{2 \alpha}\\
	&=    \left(\frac{\eta}{2} - \frac{L\eta^2}{2} - \frac{Q\eta^3    (k-n_k+1) }{2B} \right)     \sum_{i = n_k}^{k } \E \norml{ v_i }^2   -   \sum_{i=n_k}^{k} \frac{  \eta \beta    }{ 2\alpha m }\sum_{j=n_k-m}^{n_k-1} \norml{v_j}^2  -   \sum_{i=n_k}^{k}  \frac{\epsilon \eta}{2 \alpha}\\
	&\numgeq{i}  \left(\frac{\eta}{2} - \frac{L\eta^2}{2} - \frac{Q\eta^3m }{2B} \right)     \sum_{i = n_k}^{k } \E \norml{ v_i }^2   -   \sum_{i=n_k}^{k} \frac{  \eta \beta    }{ 2\alpha m }\sum_{j=n_k-m}^{n_k-1} \norml{v_j}^2  -   \sum_{i=n_k}^{k}  \frac{\epsilon \eta}{2 \alpha} \\
	&\numequ{ii} \phi  \sum_{i = n_k}^{k } \E \norml{ v_i }^2   -  \frac{\eta \beta (k - n_k + 1)}{2 \alpha m}  \sum_{i=n_k-m}^{n_k-1} \norml{v_i}^2    -   \sum_{i=n_k}^{k}  \frac{\epsilon \eta}{2 \alpha} \\
	&\numgeq{iii} \phi  \sum_{i = n_k}^{k } \E \norml{ v_i }^2   -  \frac{\eta \beta  }{2 \alpha  }  \sum_{i=n_k-m}^{n_k-1} \norml{v_i}^2    -   \sum_{i=n_k}^{k}  \frac{\epsilon \eta}{2 \alpha},  \numberthis \label{SpiderPG_final}
	\end{align*}
	where (i) follows from the fact that $  k-n_k +1 \leq m$,   (ii) follows from the fact  that $\phi \defeq \left(\frac{\eta}{2} - \frac{L\eta^2}{2} - \frac{Q\eta^3m }{2B} \right)$, and (iii) follows because $(k - n_k + 1)/m \leqslant 1$.
	 
	Now, we are ready to bound $J(\theta_{K+1}) - J(\theta_{0})$.
	\begin{align*}
	\E &J(\theta_{K+1})  - \E J(\theta_{0}) \\
	&= \E   J(\theta_{K+1}) -\E J(\theta_{n_K}) +  \E J(\theta_{n_K}) \cdots + \E J(\theta_{m}) - \E J(\theta_{0}) \\
	&\numgeq{i}   \phi  \sum_{i = n_K}^{K } \E \norml{ v_i }^2  -      \frac{\eta \beta  }{2 \alpha  }  \sum_{i=n_K-m}^{n_K-1} \norml{v_i}^2    -   \sum_{i=n_K}^{K}  \frac{\epsilon \eta}{2 \alpha} +  \cdots + \phi  \sum_{i = 0}^{m-1} \E \norml{ v_i }^2  -     \frac{\eta \beta  }{2 \alpha  }  \sum_{i= -m}^{-1} \norml{v_i}^2    -   \sum_{i=0}^{m}  \frac{\epsilon \eta}{2 \alpha} \\
	&\numgeq{ii}  \phi  \sum_{i = 0}^{K }\E \norml{ v_i }^2 -     \frac{\eta \beta  }{2 \alpha  }  \sum_{i=0}^{n_K-1} \norml{v_i}^2    -   \sum_{i=0}^{K}  \frac{\epsilon \eta}{2 \alpha} \\
	&\geqslant \left(\phi  -     \frac{\eta \beta  }{2 \alpha  }  \right) \sum_{i = 0}^{K }\E \norml{ v_i }^2  -   \sum_{i=0}^{K}  \frac{\epsilon \eta}{2 \alpha},
	\end{align*}
	where (i) follows from \eqref{SpiderPG_final}, and (ii) follows from the fact that we define $\norml{v_{-1}} = \cdots =  \norml{v_{-m}} = 0$.  Thus, we obtain
	\begin{align*}  \left(\phi  -     \frac{\eta \beta  }{2 \alpha  }  \right)  \sum_{i = 0}^{K }\E \norml{ v_i }^2 & \leq 	\E J(\theta_{K+1}) - \E J(\theta_{0}) +    \sum_{i=0}^{K}  \frac{\epsilon \eta}{2 \alpha}\\
	& \numleq{i}  J(\theta^*) -  J(\theta_{0}) +    \sum_{i=0}^{K}  \frac{\epsilon \eta}{2 \alpha}, 
	\end{align*}
	where (i) follows because $\theta^* \defeq \arg  \underset{\theta \in \mathbb{R}^d}{\max} J(\theta) $.  Here, we  assume  $\left(\phi  -     \frac{\eta \beta  }{2 \alpha  }  \right) >0$  to continue our proof. Such an assumption will be satisfied by parameter tuning  as shown in \eqref{SVRPG_para_0}. Therefore, we obtain
	\begin{align}
	\sum_{i = 0}^{K }\E \norml{ v_i }^2 \leq \left(\phi  -     \frac{\eta \beta  }{2 \alpha  }  \right)^{-1}\left( {J(\theta^*) -  J(\theta_{0})} +    \sum_{i=0}^{K}  \frac{\epsilon \eta}{2 \alpha}\right).   \label{SpiderPG_finite_bound}
	\end{align}
 
	With \eqref{SpiderPG_finite_bound}, we are now able to bound the gradient norm, i.e., $\norml{\nabla J(\theta_{\xi})}$, of the output of AbaSPIDER-PG. Observe that
	\begin{align}
	\mathbb{E}  \|\nabla J(\theta_\xi)\|^2  = \mathbb{E}  \|\nabla J(\theta_\xi) - v_\xi + v_\xi\|^2 \leq 2\mathbb{E}  \|\nabla J(\theta_\xi) - v_\xi \|^2 + 2\mathbb{E}  \| v_\xi \|^2 \label{final_0}.
	\end{align}
	Therefore, it is sufficient to bound the two terms on the right hand side of the above inequality. First, note that
	\begin{align}
	\mathbb{E} \|v_\xi\|^2  \numequ{i} \frac{1}{K+1}\sum_{i=0}^{K} \mathbb{E} \|v_i\|^2 \numleq{ii}  \left(\phi  -     \frac{\eta \beta  }{2 \alpha  }  \right)^{-1}\left( \frac{J(\theta^*) -  J(\theta_{0})}{K+1} +     \frac{\epsilon \eta}{2 \alpha}\right),  \label{eq: 5}
	\end{align}
	where (i) follows from the fact that $\xi$ is selected uniformly at random from $\{0,\ldots,K\}$, and (ii) follows from \eqref{SpiderPG_finite_bound}. On the other hand, we observe that 
		\begin{align*}
	&\quad \mathbb{E}   \|\nabla J(\theta_\xi)-v_\xi\|^2\\
	&\numequ{i} \frac{1}{K+1} \sum_{k = 0}^{K} \mathbb{E}  \|\nabla J(\theta_k)-v_k\|^2 \\ 
	&\numleq{ii}  \frac{1}{K+1} \sum_{k = 0}^{K}\left(   \frac{Q \eta^2}{B} \sum_{i = n_k}^{k } \E \norml{ v_i }^2 + \frac{  \beta  }{ \alpha m }\sum_{i=n_k-m}^{n_k-1} \norml{v_i}^2    + \frac{\epsilon}{\alpha}\right) \\
	&\numequ{iii}   \frac{Q \eta^2  }{B(K+1)} \sum_{k = 0}^{K}    \sum_{i = n_k}^{k } \E \norml{ v_i }^2  + \frac{\beta}{\alpha m(K+1)} \sum_{k = 0}^{K} \sum_{i=n_k-m}^{n_k-1} \norml{v_i}^2  + \frac{\epsilon}{\alpha} \\
	&\numequ{iv}  \frac{Q \eta^2  }{B(K+1)} \left(\sum_{k = 0}^{m-1}    \sum_{i = 0}^{k } \E \norml{ v_i }^2 + \cdots + \sum_{k = n_K}^{K}    \sum_{i = n_K}^{k } \E \norml{ v_i }^2 \right)  + \frac{\beta}{\alpha m(K+1)} \sum_{k = 0}^{K} \sum_{i=n_k-m}^{n_k-1} \norml{v_i}^2  + \frac{\epsilon}{\alpha} \\
	&\leqslant \frac{Q \eta^2  }{B(K+1)} \left(\sum_{k = 0}^{m-1}    \sum_{i = 0}^{m-1 } \E \norml{ v_i }^2 + \cdots + \sum_{k = n_K}^{K}    \sum_{i = n_K}^{K } \E \norml{ v_i }^2 \right)  + \frac{\beta}{\alpha m(K+1)} \sum_{k = 0}^{K} \sum_{i=n_k-m}^{n_k-1} \norml{v_i}^2  + \frac{\epsilon}{\alpha} \\
	&\numleq{v} \frac{Q \eta^2 m }{B(K+1)}  \sum_{i = 0}^{K } \E \norml{ v_i }^2    + \frac{\beta}{\alpha m(K+1)} \sum_{k = 0}^{K} \sum_{i=n_k-m}^{n_k-1} \norml{v_i}^2  + \frac{\epsilon}{\alpha} \\
	&\numleq{vi} \frac{Q \eta^2 m }{B(K+1)}  \sum_{i = 0}^{K } \E \norml{ v_i }^2    + \frac{\beta}{\alpha m(K+1)} \left(\sum_{k = 0}^{m-1} \sum_{i=-m}^{-1} \norml{v_i}^2 + \cdots + \sum_{k = n_K}^{K} \sum_{i=n_K-m}^{n_K-1} \norml{v_i}^2 \right) + \frac{\epsilon}{\alpha} \\
	&\numleq{vii}\frac{Q \eta^2 m }{B(K+1)}  \sum_{i = 0}^{K } \E \norml{ v_i }^2    + \frac{\beta}{\alpha  (K+1)}  \sum_{i = -m}^{n_K-1}   \norml{v_i}^2   + \frac{\epsilon}{\alpha} \\
	&\numleq{viii}   \frac{1}{K+1} \left(\frac{Q \eta^2 m }{B} + \frac{\beta}{\alpha} \right) \sum_{i = 0}^{K } \E \norml{ v_i }^2    + \frac{\epsilon}{\alpha} \\ 
	&\numleq{viiii} \frac{1}{K+1}  \left(\frac{Q \eta^2 m }{B} + \frac{\beta}{\alpha} \right)  
	\left(\phi  -     \frac{\eta \beta  }{2 \alpha  }  \right)^{-1}\left( {J(\theta^*) -  J(\theta_{0})} +    \sum_{i=0}^{K}  \frac{\epsilon \eta}{2 \alpha}\right)    + \frac{\epsilon}{\alpha} \\ 
	&=  \left(\frac{Q \eta^2 m}{B} + \frac{\beta}{\alpha} \right)  
	\left(\phi  -     \frac{\eta \beta  }{2 \alpha  }  \right)^{-1} \frac{\left( {J(\theta^*) -  J(\theta_{0})} \right)  }{K+1}   + \frac{\epsilon}{\alpha}  \left(1 + \frac{\eta}{2} \left(\frac{Q \eta^2 m^2}{B} + \frac{\beta}{\alpha} \right)   \left(\phi  -     \frac{\eta \beta  }{2 \alpha  }  \right)^{-1}\right) \numberthis \label{final_2}
	\end{align*}
	where (i) follows from the fact that $\xi$ is selected uniformly at random from $\{0,\ldots,K\}$, (ii) follows from \eqref{SpiderPG_var_bound}, (iii) follows from the fact that $k-n_k \leq m$, (iv) follows from the fact that for $n_k \leqslant k \leqslant n_k + m - 1$, $n_i = n_k$. (v) follows from $\sum_{k = n_k}^{n_k + m-1}    \sum_{i = n_k}^{n_k + m-1 } \E \norml{ v_i }^2 = m \sum_{i = n_k}^{n_k + m-1 } \E \norml{ v_i }^2$, (vi) follows from the same reasoning as in (iv), 	(vii) follows from $\sum_{k = n_k}^{n_k + m-1}    \sum_{i = n_k-m}^{ n_k-1 } \E \norml{ v_i }^2 = m\sum_{i = n_k-m}^{ n_k-1 } \E \norml{ v_i }^2$, (viii) follows from $\norml{v_{-1}} = \cdots =  \norml{v_{-m}} = 0$, and (viiii) follows from \cref{SpiderPG_finite_bound}.  
	 
	Substituting \eqref{eq: 5}, \eqref{final_2} into \eqref{final_0}, we obtain
	\begin{align*}
	\mathbb{E} \|\nabla J(\theta_\xi)\|^2  &\le \frac{2}{K+1} \left(1 + \frac{Q \eta^2 m }{B} + \frac{\beta}{\alpha}\right)  \left(\phi  -     \frac{\eta \beta  }{2 \alpha  }  \right)^{-1}\left( {J(\theta^*) -  J(\theta_{0})} \right) \\
	& \qquad \qquad + \frac{2\epsilon}{\alpha} \left(  1 + \frac{\eta}{2} \left(1 + \frac{Q \eta^2 m }{B} + \frac{\beta}{\alpha} \right)   \left(\phi  -     \frac{\eta \beta  }{2 \alpha  }  \right)^{-1}\right)  \numberthis \label{converge_SpiderPG_0}
	\end{align*}
	
\subsection{Proof of \Cref{cor:SpiderPG_sto}}	
	 
	Based on the parameter setting in \Cref{SpiderPG_conv} that 
	\begin{align}
	\eta = \frac{1}{2L},   m = \frac{L\sigma}{\sqrt{ Q\epsilon}}  , B= \frac{\sigma \sqrt{Q}}{L\sqrt{\epsilon}}, \alpha = 48 \text{ and } \beta = 16, \label{SpiderPG_para_2}
	\end{align} 
	we obtain
	\begin{align}
	\phi - \frac{\eta \beta  }{2 \alpha  } = \left(\frac{1}{4L} - \frac{1}{8L} - \frac{1}{16L} \right) - \frac{1}{32L} = \frac{1}{32L} > 0. \label{SpiderPG_para_0}
	\end{align}

	Plugging \eqref{SpiderPG_para_2} and \eqref{SpiderPG_para_0}     into \eqref{converge_SpiderPG_0}, we obtain
	\begin{align*}
	\E \norml{\nabla J(\theta_{\xi})}^2 \leq  \frac{88L}{K+1} \left( J(\theta^*) -  J(\theta_{0})\right) + \frac{\epsilon}{2}  
	\end{align*}
	To obtain $\epsilon$ accuracy, we need
	\begin{align*}
	\frac{88L}{K+1} \left( J(\theta^*) -  J(\theta_{0})\right) \leq \frac{\epsilon}{2},
	\end{align*}
	which gives
	\begin{align*}
	K = \frac{176L\left(J(\theta^*) -  J(\theta_{0})\right)}{\epsilon}.
	\end{align*}
	We note that for $\mod(k,m) =  0$, the outer loop batch size 
	$ N =  \frac{\alpha \sigma^2}{\frac{\beta}{m} \sum_{i=n_k-m}^{n_k-1} \norml{v_i}^2 + \epsilon} \leq \frac{\alpha \sigma^2}{\epsilon}$. Hence, the overall STO complexity is given by
	\begin{align*}
	K\times 2B + \sum_{k=0}^{n_K} \frac{\alpha \sigma^2}{\frac{\beta}{m} \sum_{i=km-m}^{km-1} \norml{v_i}^2 + \epsilon}   &\leqslant K\times 2B + \sum_{k=0}^{n_K}   \frac{\alpha \sigma^2}{\epsilon}  \leqslant K\times 2B + \ceil{\frac{K}{m}}\times \frac{\alpha \sigma^2}{\epsilon} \\
	&\numleq{i} 2KB + \frac{K }{m} \frac{\alpha \sigma^2}{\epsilon}   +  \frac{\alpha \sigma^2}{\epsilon}   \\ 
	&\numequ{ii} \mathcal{O} \left(\left(\frac{L\left(J(\theta^*) -  J(\theta_{0})\right)}{\epsilon}\right) \left(\frac{\sigma \sqrt{Q}}{L\sqrt{\epsilon}} +   \frac{ \sigma^2}{\epsilon} \frac{\sqrt{ Q\epsilon}}{L\sigma}  \right) + \frac{\sigma^2}{\epsilon}  \right) \\
	&= \mathcal{O} \left(\left(\frac{L\left(J(\theta^*) -  J(\theta_{0})\right)}{\epsilon}\right) \left(\frac{\sigma \sqrt{Q}}{L\sqrt{\epsilon}}   \right) + \frac{\sigma^2}{\epsilon}  \right) \\
	&=\mathcal{O} \left(  \epsilon^{-3/2} + \epsilon^{-1} \right).
	\end{align*}
	where (i) follows from the fact that $\ceil{\frac{K}{m}}\times N \leq \frac{KN}{m}  + N $, and (ii) follows from the parameters setting of $K,B, \text{ and } m$ in \eqref{SpiderPG_para_2}.

\section{Proof of Technical Lemmas} \label{proof_of_lemmas}
\subsection{Proof of \Cref{bounded_from_papini}}
	(i), (ii), (iii) follow from Lemma B.2, Lemma B.3 and Lemma B.4 in \citealt{Papini2018}, respectively.

\subsection{Proof of \Cref{lispchitz_g}}
Note that 
	\begin{align*}
	\E_{\tau \sim p(\cdot|\theta_1) } &\norml{g(\tau|\theta_1) - \omega(\tau|\theta_1, \theta_2)g(\tau|\theta_2)  }^2 \\
	&= \E_{\tau \sim p(\cdot|\theta_1) } \norml{g(\tau|\theta_1) - g(\tau|\theta_2) + g(\tau|\theta_2) - \omega(\tau|\theta_1, \theta_2)g(\tau|\theta_2)  }^2 \\
	&\numequ{i}   \E_{\tau \sim p(\cdot|\theta_1) } 2\norml{g(\tau|\theta_1) - g(\tau|\theta_2)  }^2 +   \E_{\tau \sim p(\cdot|\theta_1) } 2\norml{  g(\tau|\theta_2) - \omega(\tau|\theta_1, \theta_2)g(\tau|\theta_2) }^2\\
	&\numleq{ii} 2L_g^2 \norml{\theta_1 - \theta_2}^2 +   \E_{\tau \sim p(\cdot|\theta_1) } 2\norml{g(\tau|\theta_2) }^2\norml{  1 - \omega(\tau|\theta_1, \theta_2)}^2\\
	&\numleq{iii} 2L_g^2 \norml{\theta_1 - \theta_2}^2 +    2\Gamma \alpha\norml{\theta_1 - \theta_2}^2 =  2(L_g^2+ \Gamma \alpha) \norml{\theta_1 - \theta_2}^2 \numequ{iv} Q\norml{\theta_1 - \theta_2}^2,\\
	\end{align*}
	where (i) follows from the fact that $\norml{x+y}^2 \leq 2\norml{x}^2 + 2\norml{y}^2$, (ii) follows from item (ii) in \Cref{bounded_from_papini}, and (iii) follows from item (iii) in \Cref{bounded_from_papini} and \Cref{var_assumption}.	Then, the proof is complete. 

\subsection{Proof of \Cref{beginWith}}
	We derive the following lower bound
	\begin{align*}
	J(\theta_{k+1}) - J(\theta_{k}) &\numgeq{i} \inner{\nabla J(\theta_{k})}{\theta_{k+1} - \theta_{k}} - \frac{L}{2} \norml{\theta_{k+1} - \theta_{k}}^2 \\
	&\numequ{ii} \eta \inner{\nabla J(\theta_{k})}{v_k} - \frac{L\eta^2}{2} \norml{v_k}^2 \\
	&=  \eta \inner{\nabla J(\theta_{k}) - v_k + v_k }{v_k} - \frac{L\eta^2}{2} \norml{v_k}^2 \\
	&= \eta \norml{v_k}^2 +  \eta \inner{\nabla J(\theta_{k}) - v_k }{v_k} - \frac{L\eta^2}{2} \norml{v_k}^2 \\
	&\numgeq{iii} \eta \norml{v_k}^2 -  \eta \frac{\norml{ v_k - \nabla J(\theta_{k})}^2 + \norml{v_k}^2}{2}   - \frac{L\eta^2}{2} \norml{v_k}^2 \\
	&= \left(\frac{\eta}{2} - \frac{L\eta^2}{2}\right)\norml{v_k}^2 -  \frac{\eta}{2} \norml{\nabla J(v_k) -\theta_{k}}^2,  
	\end{align*}
	where (i) follows from the fact  that $\nabla J$ is $L$-Lipschitz, (ii) follows from  the update rule   $x_{k+1} = x_k + \eta  v_k$, and (iii) follows from Young's inequality. Taking the expectation over the entire random process on both sides, we obtain the desired result.

\subsection{Proof of \Cref{batch_variance}}
\begin{lemma} \label{batch_variance}
	Let $X,X_1, \cdots, X_n$ be   independent and identically distributed (i.i.d.) random variables with mean $\E[X]$, then, the following equation holds:
	\begin{align*}
	\E \norml{ \frac{1}{n} \sum_{i=1}^{n} X_i - \E X}^2 = \frac{\E \norml{  X  - \E X}^2}{n} 
	\end{align*}
\end{lemma}
\begin{proof}
Standard calculation yields
	\begin{align*}
	\E \norml{ \frac{1}{n} \sum_{i=1}^{n} X_i - \E X}^2 &= \E \norml{ \frac{1}{n} \sum_{i=1}^{n} (X_i - \E X)}^2 = \frac{1}{n^2} \E \norml{   \sum_{i=1}^{n} (X_i - \E X)}^2\\
	&= \frac{1}{n^2}  \sum_{i=1}^{n}\sum_{j=1}^{n}  \E \inner{X_i - \E X}{X_j - \E X}   \\
	&\numequ{i} \frac{1}{n^2}  \sum_{i=1}^{n}  \E \inner{X_i - \E X}{X_i - \E X}   \\
	&=   \frac{1}{n^2}  \sum_{i=1}^{n} \E \norml{X_i - \E X}^2\\
	&\numequ{ii} \frac{\E \norml{X  - \E X}^2}{n},
	\end{align*} 
	where (i) follows from the fact that $X_1, \cdots, X_n$ are i.i.d. random variables such that if $i \neq j$,   $\E\inner{X_i - \E X}{X_j - \E X}  = 0 $ , and (ii) follows from the fact that for i.i.d. random variables $\E \norml{X  - \E X}^2 = \E \norml{X_1  - \E X}^2 \cdots = \E \norml{X_n  - \E X}^2 $.
\end{proof}

\section{Proofs for Results in Appendix~\ref{se:5}} \label{youdiandou}
\subsection{Proof for Theorem~\ref{svrg:pl}}
To simplify notations, we let $c_\beta=c_\epsilon=\alpha =  \Big(2\tau+ \frac{2\tau}{1-\exp(\frac{-4}{c_\eta(c_\eta -2)})}\Big)\vee \frac{16c_\eta L\tau }{m}$.

Since the objective function $f(\cdot)$ has a $L$-Lipschitz continuous gradient, we obtain that  for $1 \leq t \leq m$, 
\begin{align}
f({x}_t^s)  \leq &  f({x}_{t-1}^s) + \langle  \nabla f({x}_{t-1}^s), {x}_{t}^s-{x}_{t-1}^s  \rangle + \frac{L\eta^2}{2}\|  {v}_{t-1}^s \|^2 \nonumber
\\=& f({x}_{t-1}^s) + \frac{\eta}{2}\| \nabla f({x}_{t-1}^s) -{v}_{t-1}^s\|^2 - \frac{\eta}{2}\| \nabla f({x}_{t-1}^s)\|^2 - \frac{\eta}{2} \|{v}_{t-1}^s\|^2
 + \frac{L\eta^2}{2}\|  {v}_{t-1}^s \|^2, \nonumber
\end{align}
which, in conjunction with the PL condition that $\|\nabla f({x}_{t-1}^s)\|^2 \geq \frac{1}{\tau} (f({x}_{t-1}^s)-f({x}^*))$, implies that 
\begin{align}
f({x}_t^s) - f({x}^*) \leq \left(1- \frac{\eta}{2\tau} \right)(f({x}_{t-1}^s) - f({x}^*)) - \left(\frac{\eta}{2}  - \frac{L\eta^2}{2}\right) \|{v}_{t-1}^s \|^2 + \frac{\eta}{2} \|\nabla f({x}_{t-1}^s) -{v}_{t-1}^s\|^2. \nonumber
\end{align}
Recall that $\mathbb{E}_{t,s}(\cdot)$  denotes $\mathbb{E}(\cdot | {x}_0^1,{x}_0^2,...,{x}_2^1,...,{x}_{t}^s)$. Then, taking expectation $\mathbb{E}_{0,s}(\cdot)$  over the above inequality yields, for $1 \leq t \leq m$, 
\begin{align}\label{eq:usetwo}
\mathbb{E}_{0,s}(f({x}_t^s) - f({x}^*) )\leq& \left(1- \frac{\eta}{2\tau} \right)\mathbb{E}_{0,s}(f({x}_{t-1}^s) - f({x}^*)) - \left(\frac{\eta}{2}  - \frac{L\eta^2}{2}\right) \mathbb{E}_{0,s} \|{v}_{t-1}^s \|^2 \nonumber
\\&+ \frac{\eta}{2} \mathbb{E}_{0,s}\|\nabla f({x}_{t-1}^s) -{v}_{t-1}^s\|^2,	
\end{align}
which, in conjunction with Lemma~\ref{le:variance}, implies that 
\begin{align}
\mathbb{E}_{0,s}(f({x}_t^s) - f({x}^*) )\leq& \left(1- \frac{\eta}{2\tau} \right)\mathbb{E}_{0,s}(f({x}_{t-1}^s) - f({x}^*)) - \left(\frac{\eta}{2}  - \frac{L\eta^2}{2}\right) \mathbb{E}_{0,s} \|{v}_{t-1}^s \|^2 \nonumber
\\&+ \frac{\eta^3L^2(t-1)}{2B}\mathbb{E}_{0,s}\sum_{i=0}^{t-2}\|{v}_i^s\|^2 +\frac{\eta I_{(N_s<n)}}{2N_s} \sigma^2.\nonumber
\end{align}
Let $\gamma := 1-\frac{\eta}{2\tau}$. Then, telescoping the above inequality over $t$ from $1$ to $m$ and using the fact that $t-1<m$, we have
\begin{align}\label{eq:eos}
\mathbb{E}_{0,s}(f({x}_m^s) -& f({x}^*) )\leq \gamma^m\mathbb{E}_{0,s}(f({x}_{0}^s) - f({x}^*)) - \left(\frac{\eta}{2}  - \frac{L\eta^2}{2}\right)\sum_{t=0}^{m-1}\gamma^{m-1-t} \mathbb{E}_{0,s} \|{v}_{t}^s \|^2 \nonumber
\\&+ \frac{\eta^3L^2m}{2B}\sum_{t=0}^{m-2}\gamma^{m-2-t}\mathbb{E}_{0,s}\sum_{i=0}^{t}\|{v}_i^s\|^2 +\left( \sum_{t=0}^{m-1}\gamma^t  \right) \frac{\eta I_{(N_s<n)}}{2N_s} \sigma^2.
\end{align}
Note that $\gamma^{m-1-t}\geq \gamma^{m}$ for $0\leq t \leq m-1$ and $\sum_{t=0}^{m-1}\gamma^t =\frac{1-\gamma^m}{1-\gamma}\leq \frac{1}{1-\gamma}=\frac{2\tau}{\eta}$. Then, we obtain from~\eqref{eq:eos} that 
\begin{align}\label{eq:eesa}
\mathbb{E}_{0,s}(f({x}_m^s) -f({x}^*) )\leq &\gamma^m\mathbb{E}_{0,s}(f({x}_{0}^s) - f({x}^*)) - \left(\frac{\eta}{4}  - \frac{L\eta^2}{2}\right)\sum_{t=0}^{m-1}\gamma^{m-1-t} \mathbb{E}_{0,s} \|{v}_{t}^s \|^2 \nonumber
\\&-\frac{\eta}{4}\sum_{t=0}^{m-1}\gamma^{m-1-t} \mathbb{E}_{0,s} \|{v}_{t}^s \|^2 + \frac{\eta^3L^2m}{2B}\mathbb{E}_{0,s}\sum_{i=0}^{m-1}\|{v}_i^s\|^2\left(\sum_{t=0}^{m-2}\gamma^{m-2-t}\right) \nonumber
\\&+\frac{\tau I_{(N_s<n)}}{N_s} \sigma^2 \nonumber
\\\leq & \gamma^m\mathbb{E}_{0,s}(f({x}_{0}^s) - f({x}^*)) - \left(\frac{\eta}{4}  - \frac{L\eta^2}{2}\right)\gamma^{m}\sum_{t=0}^{m-1} \mathbb{E}_{0,s} \|{v}_{t}^s \|^2 \nonumber
\\&-\frac{\eta}{4}\sum_{t=0}^{m-1}\gamma^{m-1-t} \mathbb{E}_{0,s} \|{v}_{t}^s \|^2 + \frac{\eta^2L^2m\tau}{B}\mathbb{E}_{0,s}\sum_{i=0}^{m-1}\|{v}_i^s\|^2 +\frac{\tau I_{(N_s<n)}}{N_s} \sigma^2 \nonumber
\\\leq & \gamma^m\mathbb{E}_{0,s}(f({x}_{0}^s) - f({x}^*)) - \left(\left(\frac{\eta}{4}  - \frac{L\eta^2}{2}\right)\gamma^{m}-\frac{\eta^2L^2m\tau}{B}\right)\sum_{t=0}^{m-1} \mathbb{E}_{0,s} \|{v}_{t}^s \|^2 \nonumber
\\& +\frac{\tau I_{(N_s<n)}}{N_s} \sigma^2 -\frac{\eta}{4}\sum_{t=0}^{m-1}\gamma^{m-1-t} \mathbb{E}_{0,s} \|{v}_{t}^s \|^2.
\end{align}
Recall  $\eta =\frac{1}{c_\eta L}$ ($c_\eta>4$), $\frac{8L\tau}{c_\eta -2}\leq m <4L\tau$ and $B= m^2 $. Then, we have 
\begin{align}\label{eq:eta4}
\left(\frac{\eta}{4}  - \frac{L\eta^2}{2}\right)\gamma^{m} &=\eta \left(\frac{1}{4}-\frac{1}{2c_\eta}\right)\left( 1-\frac{1}{2c_\eta\tau L}\right)^m >\eta \left(\frac{1}{4}-\frac{1}{2c_\eta}\right) \left(1-\frac{1}{2m}\right)^{m} \nonumber \\&\overset{(i)}\geq \frac{\eta}{2}\left(\frac{1}{4}-\frac{1}{2c_\eta}\right)\geq \frac{\eta^2L^2\tau}{m} = \frac{\eta^2L^2m\tau}{B} ,
\end{align}
where (i) follows from the fact that  $\left(1-\frac{1}{2m}\right)^{m}\geq \frac{1}{2}$ for $m\geq 1$. Recall $c_\beta=c_\epsilon=\alpha$ and $N_s=\min\{c_\beta \sigma^2\beta_s^{-1}, c_\epsilon\sigma^2 \epsilon^{-1},n\}$. Then, combining~\eqref{eq:bsn},~\eqref{eq:eesa} and~\eqref{eq:eta4} yields   
\begin{align}
\mathbb{E}_{0,s}(f({x}_m^s) -f({x}^*) ) \leq & \gamma^m\mathbb{E}_{0,s}(f({x}_{0}^s) - f({x}^*))  +\frac{\tau I_{(N_s<n)}}{N_s} \sigma^2 -\frac{\eta}{4}\sum_{t=0}^{m-1}\gamma^{m-1-t} \mathbb{E}_{0,s} \|{v}_{t}^s \|^2. \nonumber
\\ \leq & \gamma^m\mathbb{E}_{0,s}(f({x}_{0}^s) - f({x}^*))  +\tau\left(\frac{\beta_s}{\alpha}+\frac{\epsilon}{\alpha}\right) -\frac{\eta}{4}\sum_{t=0}^{m-1}\gamma^{m-1-t} \mathbb{E}_{0,s} \|{v}_{t}^s \|^2. \nonumber
\end{align}
Further taking expectation of the above inequality over  ${x}_{0}^1,...,{x}_0^s$,  we obtain
\begin{align*}
\mathbb{E}(f({x}_m^s) -f({x}^*) ) \leq \gamma^m\mathbb{E}(f({x}_{0}^s) - f({x}^*))  + \frac{\tau}{\alpha}\mathbb{E}\beta_s+\frac{\tau\epsilon}{\alpha}-\frac{\eta}{4}\sum_{t=0}^{m-1}\gamma^{m-1-t} \mathbb{E} \|{v}_{t}^s \|^2
\end{align*}
Recall that $\beta_1 \leq \epsilon\big(\frac{1}{\gamma}\big)^{m(S-1)} $  and $\beta_s = \frac{1}{m}\sum_{t=1}^m \|{v}_{t-1}^{s-1}\|^2$ for $s\geq 2$. Then, telescoping the above inequality over $s$ from $1$ to $S$ yields 
\begin{align}\label{eq:mlop}
\mathbb{E}(f({x}_m^S) -f({x}^*) ) \leq &\gamma^{Sm}\mathbb{E}(f({x}_{0}) - f({x}^*))  +
\sum_{s=1}^{S-1} \gamma^{m(S-1-s)} \frac{\tau}{\alpha m}\sum_{t=0}^{m-1}\mathbb{E}\|{v}_t^s\|^2 \nonumber
\\&+\gamma^{m(S-1)}\frac{\tau \beta_1}{\alpha}+\sum_{s=1}^S \gamma^{m(S-s)}\frac{\tau\epsilon}{\alpha}-\frac{\eta}{4}\sum_{s=1}^S\gamma^{m(S-s)}\sum_{t=0}^{m-1}\gamma^{m-1-t} \mathbb{E} \|{v}_{t}^s \|^2 \nonumber
\\\overset{(i)}\leq &\gamma^{K}\mathbb{E}(f({x}_{0}) - f({x}^*))  -\left( \frac{\eta}{4}\gamma^{2m}- \frac{\tau}{\alpha m}\right)
\sum_{s=1}^{S-1} \gamma^{m(S-1-s)}\sum_{t=0}^{m-1}\mathbb{E}\|{v}_t^s\|^2 \nonumber
\\&+\bigg(1+ \frac{1}{1-\exp(-\frac{4}{c_\eta(c_\eta -2)})}\bigg)\frac{\tau\epsilon}{\alpha},
\end{align}
where (i) follows from the fact that $\gamma^{m-1-t}\geq \gamma^{m}$ for $0\leq t \leq m-1$, $\gamma^{m(S-1)}\leq 1$, $\sum_{s=1}^S \gamma^{m(S-s)}\leq \frac{1}{1-\gamma^m }$ and
\begin{align*}
\gamma^{m} =\left(1 - \frac{1}{2c_\eta \tau L}\right)^m \leq \left(1 - \frac{4}{c_\eta(c_\eta-2)m}\right)^m \leq \exp\left(-\frac{4}{c_\eta(c_\eta-2)}\right).
\end{align*}
 Since $\alpha=\Big(2\tau+ \frac{2\tau}{1-\exp(-\frac{4}{c_\eta(c_\eta -2)})}\Big)\vee \frac{16c_\eta L\tau }{m} $, we have 
\begin{align}\label{eq:taos}
\bigg(1+ \frac{1}{1-\exp(-\frac{4}{c_\eta(c_\eta -2)})}\bigg)\frac{\tau\epsilon}{\alpha}\leq \frac{1}{2},\quad\frac{\eta}{4}\gamma^{2m} \overset{(i)}>\frac{1}{16c_\eta L} \geq \frac{\tau}{\alpha m} 
\end{align}
where (i) follows from~\eqref{eq:eta4} that $\gamma^m\geq \left(1-\frac{1}{2m}\right)^{m}\geq \frac{1}{2}$. Note that  ${x}_m^S={\tilde x}^S$. Then, combining~\eqref{eq:taos} and~\eqref{eq:mlop} yields 
\begin{align}
\mathbb{E}(f({\tilde x}^S) -f({x}^*) ) \leq \gamma^{K}(f({x}_{0}) - f({x}^*))+\frac{\epsilon}{2}.
\end{align} 
Let $K = (2c_\eta\tau L - 1)\log \left(\frac{2(f({x}_0)-f({x}^*))}{\epsilon}\right)$. Then, we have 
\begin{align*}
\gamma^{K}(f({x}_{0}) - f({x}^*))=& \exp\left[(2c_\eta\tau L-1)\log \frac{1}{\gamma}\log\left( \frac{\epsilon} {2(f({x}_0)-f({x}^*))}\right) \right] (f({x}_{0}) - f({x}^*)) \overset{(i)}\leq \frac{\epsilon}{2},
\end{align*}
where (i) follows from the fact that $\log \frac{1}{\gamma} = \log \big(1+\frac{1}{2c_\eta\tau L-1}\big)\leq \frac{1}{2c_\eta\tau L -1}$. Thus, the total number of SFO calls is 
\begin{align*}
\sum_{s=1}^S\min\left\{\frac{c_\beta}{\beta_s}, \frac{c_\epsilon}{\epsilon},n\right\} + KB \leq& \mathcal{O}\left( \left(\frac{c_\epsilon}{\epsilon}\wedge n\right)\frac{\tau}{m} \log \frac{1}{\epsilon} + B\tau \log \frac{1}{\epsilon}\right)
\\\overset{(i)}\leq&\mathcal{O}\left( \left(\frac{\tau}{\epsilon}\wedge n\right) \log \frac{1}{\epsilon} + \tau^{3} \log \frac{1}{\epsilon}\right),
\end{align*}
where (i) follows from the fact that $m=\Theta(\tau)$ and $c_\epsilon=\Theta(\tau)$.

\subsection{Proof of Theorem~\ref{spider_pl}}
To simplify notations, we let $c_\beta=c_\epsilon=\alpha =  \Big(2\tau+ \frac{2\tau}{1-\exp(\frac{-4}{c_\eta(c_\eta -2)})}\Big)\vee \frac{16c_\eta L\tau }{m}$.

Using an approach similar to \eqref{eq:usetwo}, we have, for $1\leq t \leq m$  
\begin{align*}
\mathbb{E}_{0,s}(f({x}_t^s) - f({x}^*) )\leq& \left(1- \frac{\eta}{2\tau} \right)\mathbb{E}_{0,s}(f({x}_{t-1}^s) - f({x}^*)) - \left(\frac{\eta}{2}  - \frac{L\eta^2}{2}\right) \mathbb{E}_{0,s} \|{v}_{t-1}^s \|^2 \nonumber
\\&+ \frac{\eta}{2} \mathbb{E}_{0,s}\|\nabla f({x}_{t-1}^s) -{v}_{t-1}^s\|^2,	
\end{align*}
which, in conjunction with Lemma~\ref{le:spider}, implies that 
\begin{align*}
\mathbb{E}_{0,s}(f({x}_t^s) - f({x}^*) )\leq& \left(1- \frac{\eta}{2\tau} \right)\mathbb{E}_{0,s}(f({x}_{t-1}^s) - f({x}^*)) - \left(\frac{\eta}{2}  - \frac{L\eta^2}{2}\right) \mathbb{E}_{0,s} \|{v}_{t-1}^s \|^2 \nonumber
 \\&+\frac{\eta ^3 L^2}{2B}	\sum_{i=0}^{t-2}\mathbb{E}_{0,s}\|{v}_{i}^s\|^2 + 	\frac{\eta I_{(N_s<n)}}{2N_s}\sigma^2.
\end{align*}
Let $\gamma := 1-\frac{\eta}{2\tau}$. Then, telescoping the above inequality over $t$ from 
$1$ to $m$ yields
\begin{align*}
\mathbb{E}_{0,s}(f({x}_m^s) &- f({x}^*) )\leq\gamma^m\mathbb{E}_{0,s}(f({x}_{0}^s) - f({x}^*)) - \left(\frac{\eta}{2}  - \frac{L\eta^2}{2}\right) \sum_{t=0}^{m-1}\gamma^{m-1-t}\mathbb{E}_{0,s} \|{v}_{t}^s \|^2 \nonumber
\\&+\frac{\eta ^3 L^2}{2B}	\sum_{t=0}^{m-2}\gamma^{m-2-t}\sum_{i=0}^{t}\mathbb{E}_{0,s}\|{v}_{i}^s\|^2 + \left(\sum_{t=0}^{m-1}\gamma^t\right)	\frac{\eta I_{(N_s<n)}}{2N_s}\sigma^2,
\end{align*}
which, in conjunction with $\sum_{t=0}^{m-1}\gamma^t =\frac{1-\gamma^m}{1-\gamma}\leq \frac{1}{1-\gamma}=\frac{2\tau}{\eta}$ and $\gamma^{m-1-t}\geq \gamma^{m}$ for $0\leq t \leq m-1$, implies that 
\begin{align}\label{eq:easp1}
\mathbb{E}_{0,s}(f({x}_m^s) -f({x}^*) )\leq &\gamma^m\mathbb{E}_{0,s}(f({x}_{0}^s) - f({x}^*)) - \left(\frac{\eta}{4}  - \frac{L\eta^2}{2}\right)\sum_{t=0}^{m-1}\gamma^{m-1-t} \mathbb{E}_{0,s} \|{v}_{t}^s \|^2 \nonumber
\\&-\frac{\eta}{4}\sum_{t=0}^{m-1}\gamma^{m-1-t} \mathbb{E}_{0,s} \|{v}_{t}^s \|^2 + \frac{\eta^3L^2}{2B}\left(\sum_{t=0}^{m-2}\gamma^{m-2-t}\right)\mathbb{E}_{0,s}\sum_{i=0}^{m-1}\|{v}_i^s\|^2 \nonumber
\\&+\frac{\tau I_{(N_s<n)}}{N_s} \sigma^2 \nonumber
\\\leq & \gamma^m\mathbb{E}_{0,s}(f({x}_{0}^s) - f({x}^*)) - \left(\frac{\eta}{4}  - \frac{L\eta^2}{2}\right)\gamma^{m}\sum_{t=0}^{m-1} \mathbb{E}_{0,s} \|{v}_{t}^s \|^2 \nonumber
\\&-\frac{\eta}{4}\sum_{t=0}^{m-1}\gamma^{m-1-t} \mathbb{E}_{0,s} \|{v}_{t}^s \|^2 + \frac{\eta^2L^2\tau}{B}\mathbb{E}_{0,s}\sum_{i=0}^{m-1}\|{v}_i^s\|^2 +\frac{\tau I_{(N_s<n)}}{N_s} \sigma^2 \nonumber
\\\leq & \gamma^m\mathbb{E}_{0,s}(f({x}_{0}^s) - f({x}^*)) - \left(\left(\frac{\eta}{4}  - \frac{L\eta^2}{2}\right)\gamma^{m}-\frac{\eta^2L^2\tau}{B}\right)\sum_{t=0}^{m-1} \mathbb{E}_{0,s} \|{v}_{t}^s \|^2 \nonumber
\\& +\frac{\tau I_{(N_s<n)}}{N_s} \sigma^2 -\frac{\eta}{4}\sum_{t=0}^{m-1}\gamma^{m-1-t} \mathbb{E}_{0,s} \|{v}_{t}^s \|^2.
\end{align}
Recall that  $\eta =\frac{1}{c_\eta L}$ with $c_\eta>4$ and  $B=m$ with $\frac{8L\tau}{c_\eta -2}\leq m <4L\tau$. Then, we have 
\begin{align}\label{eq:pops}
\left(\frac{\eta}{4}  - \frac{L\eta^2}{2}\right)\gamma^{m} &=\eta \left(\frac{1}{4}-\frac{1}{2c_\eta}\right)\left( 1-\frac{1}{2c_\eta\tau L}\right)^m >\eta \left(\frac{1}{4}-\frac{1}{2c_\eta}\right) \left(1-\frac{1}{2m}\right)^{m} \nonumber \\&\geq \frac{\eta}{2}\left(\frac{1}{4}-\frac{1}{2c_\eta}\right)\geq \frac{\eta^2L^2\tau}{m} = \frac{\eta^2L^2\tau}{B} ,
\end{align} 
which, combined with~\eqref{eq:easp1} and~\eqref{eq:bsn} , implies that 
\begin{align*}
 \mathbb{E}_{0,s}(f({x}_m^s) -f({x}^*) )\leq  \gamma^m\mathbb{E}_{0,s}(f({x}_{0}^s) - f({x}^*)) +  \tau\left(\frac{\beta_s}{\alpha}+\frac{\epsilon}{\alpha}\right)-\frac{\eta}{4}\sum_{t=0}^{m-1}\gamma^{m-1-t} \mathbb{E}_{0,s} \|{v}_{t}^s \|^2.
\end{align*}
Taking the expectation of the above inequality ${x}_{0}^1,...,{x}_0^s$, we obtain 
\begin{align*}
\mathbb{E}(f({x}_m^s) -f({x}^*) )\leq  \gamma^m\mathbb{E}(f({x}_{0}^s) - f({x}^*)) +  \frac{\tau}{\alpha}\mathbb{E}\beta_s+\frac{\tau\epsilon}{\alpha}-\frac{\eta}{4}\sum_{t=0}^{m-1}\gamma^{m-1-t} \mathbb{E} \|{v}_{t}^s \|^2.
\end{align*}
Recall $\beta_1  \leq  \epsilon\big(\frac{1}{\gamma}\big)^{m(S-1)}  $ and $\beta_s = \frac{1}{m}\sum_{t=0}^{m-1} \|{v}_{t}^{s-1}\|^2$ for $s=2,...,S$. Then, telescoping the above inequality over $s$ from $1$ to $S$ and using an approach similar to~\eqref{eq:mlop}, we have 
\begin{align*}
\mathbb{E}(f({x}_m^S) -f({x}^*) ) \leq &\gamma^{K}\mathbb{E}(f({x}_{0}) - f({x}^*))  -\left( \frac{\eta}{4}\gamma^{2m}- \frac{\tau}{\alpha m}\right)
\sum_{s=1}^{S-1} \gamma^{m(S-1-s)}\sum_{t=0}^{m-1}\mathbb{E}\|{v}_t^s\|^2 \nonumber
\\&+\bigg(1+ \frac{1}{1-\exp(-\frac{4}{c_\eta(c_\eta -2)})}\bigg)\frac{\tau\epsilon}{\alpha},
\end{align*}
which, in conjunction with~\eqref{eq:taos}, yields
\begin{align}
\mathbb{E}(f({x}_m^S) -f({x}^*) ) \leq \gamma^{K}\mathbb{E}(f({x}_{0}) - f({x}^*))  +\frac{\epsilon}{2}.
\end{align}
Let $K = (2c_\eta\tau L - 1)\log \left(\frac{2(f({x}_0)-f({x}^*))}{\epsilon}\right)$. Then, we have 
$\gamma^{K}(f({x}_{0}) - f({x}^*))\leq \frac{\epsilon}{2}$.
Thus, the total number of SFO calls is given by 
\begin{align*}
\sum_{s=1}^S\min\left\{\frac{c_\beta}{\beta_s}, \frac{c_\epsilon}{\epsilon},n\right\} + KB \leq& \mathcal{O}\left( \left(\frac{c_\epsilon}{\epsilon}\wedge n\right)\frac{\tau}{m} \log \frac{1}{\epsilon} + B\tau \log \frac{1}{\epsilon}\right)
\\\overset{(i)}\leq&\mathcal{O}\left( \left(\frac{\tau}{\epsilon}\wedge n\right) \log \frac{1}{\epsilon} + \tau^{2} \log \frac{1}{\epsilon}\right),
\end{align*}
where (i) follows from the fact that $B= m=\Theta(\tau)$ and $c_\epsilon=\Theta(\tau)$.

\newpage
\section{Proofs for Results in \Cref{apen:sgd}}
\subsection{Proof of Theorem~\ref{th_abasgd} }
	Since the gradient $\nabla f $ is $L$-Lipschitz, we obtain that,  for $t \geq 0$,
	\begin{align}
	f({x}_{t+1})   \leq&  f({x}_{t} ) + \langle \nabla f({x}_{t} ), {x}_{t+1} -{x}_{t} \rangle +\frac{L }{2} \|{x}_{t+1} -{x}_{t} \|^2  \nonumber \\ \overset{(i)}=&  f({x}_{t} )  -\eta \langle \nabla f({x}_{t } ),  {v}_t \rangle +\frac{L\eta^2}{2} \|{v}_{t} \|^2  \nonumber
	 \\=&  f({x}_{t} )  -\eta \langle \nabla f({x}_{t } ) - {v}_t + {v}_t,  {v}_t \rangle +\frac{L\eta^2}{2} \|{v}_{t} \|^2  \nonumber
	\\= & f({x}_{t} ) -\eta \langle \nabla f({x}_{t} )-{v}_{t} ,   {v}_{t} \rangle - \eta\|{v}_{t} \|^2 +\frac{L\eta^2}{2} \|{v}_{t} \|^2 \nonumber 
	\\= & f({x}_{t} ) -\eta \langle \nabla f({x}_{t} )-{v}_{t} ,   {v}_{t} - \nabla f({x}_{t} ) + \nabla f({x}_{t} )\rangle 	-\Big (\eta-\frac{L\eta^2}{2} \Big)\|{v}_{t } \|^2 \nonumber 
	\\= & f({x}_{t} ) -\eta \langle \nabla f({x}_{t} )-{v}_{t} ,   {v}_{t} - \nabla f({x}_{t} ) \rangle  -\eta \langle \nabla f({x}_{t} )-{v}_{t} ,    \nabla f({x}_{t} ) \rangle	-\Big (\eta-\frac{L\eta^2}{2} \Big)\|{v}_{t } \|^2 \nonumber  
	\\= & f({x}_{t} ) + \eta \|  \nabla f({x}_{t} )-{v}_{t} \|^2   -\eta \langle \nabla f({x}_{t} )-{v}_{t} ,    \nabla f({x}_{t} ) \rangle	-\Big (\eta-\frac{L\eta^2}{2} \Big)\|{v}_{t } \|^2 \nonumber  
	\end{align}
	where (i) follows from the fact that ${x}_{t+1} ={x}_{t} -\eta{v}_{t} $.	Then,  	taking expectation $\mathbb{E}(\cdot)  $ over the above inequality yields
	\begin{align}\label{abaSGD_1}
	\mathbb{E}  f({x}_{t+1}) \leq& \mathbb{E} f({x}_{t } ) +\eta \mathbb{E}  \|  \nabla f({x}_{t} )-{v}_{t} \|^2   -\eta \mathbb{E}  \langle \nabla f({x}_{t} )-{v}_{t} ,    \nabla f({x}_{t} ) \rangle	-\Big (\eta-\frac{L\eta^2}{2} \Big) \mathbb{E} \|{v}_{t } \|^2   \nonumber  \\
	\overset{(i)}=& \mathbb{E}  f({x}_{t } ) +\eta \mathbb{E}  \|  \nabla f({x}_{t} )-{v}_{t} \|^2   	-\Big (\eta-\frac{L\eta^2}{2} \Big) \mathbb{E} \| {v}_{t } \|^2  \nonumber  \\
	=& \mathbb{E}  f({x}_{t } )	-\Big (\eta-\frac{L\eta^2}{2} \Big) \mathbb{E} \| {v}_{t } \|^2    +\eta \mathbb{E}  \|  \nabla f({x}_{t} )-{v}_{t} \|^2,   
	\end{align}
	where (i) follows from $ \mathbb{E} \langle \nabla f({x}_{t} )-{v}_{t} ,    \nabla f({x}_{t} ) \rangle =\mathbb{E}_{{x}_0,...,{x}_t}\left(\mathbb{E}_{t} \langle \nabla f({x}_{t} )-{v}_{t} ,    \nabla f({x}_{t} ) \rangle\right) = 0$.

	Next, we upper-bound $\mathbb{E}  \|  \nabla f({x}_{t} )-{v}_{t} \|^2 $. 
	For the case when $|B_t| < n$, we  have 
	\begin{align*}
		\mathbb{E}  \|  \nabla f({x}_{t} )-{v}_{t} \|^2 &= \mathbb{E}  \left\|  \nabla f({x}_{t} )- \frac{1}{|B_t|} \sum_{i \in B_t}  \nabla f_i({x}_{t}) \right\|^2 = \mathbb{E}   \left\|\frac{1}{|B_t| }\sum_{i \in B_t}  \left( \nabla f({x}_{t} )-  \nabla f_i({x}_{t})\right) \right\|^2
		\\ &= \mathbb{E}  \frac{1}{|B_t|^2 } \left\| \sum_{i \in B_t} \left(   \nabla f({x}_{t} )-  \nabla f_i({x}_{t})\right) \right\|^2  \\
		&=\mathbb{E}  \frac{1}{|B_t|^2 }  \sum_{i \in B_t} \sum_{j \in B_t} \left\langle   \nabla f({x}_{t} )-  \nabla f_i({x}_{t}), \nabla f({x}_{t} )-  \nabla f_j({x}_{t})\right\rangle  \\
		&=\mathbb{E}_{{x}_0,...,{x}_t} \bigg(\mathbb{E}_{t} \frac{1}{|B_t|^2 }  \sum_{i \in B_t} \sum_{j \in B_t} \left\langle   \nabla f({x}_{t} )-  \nabla f_i({x}_{t}), \nabla f({x}_{t} )-  \nabla f_j({x}_{t})\right\rangle \bigg) \\
		&= \mathbb{E}_{{x}_0,...,{x}_t}  \frac{1}{|B_t|^2 }  \sum_{i \in B_t} \sum_{j \in B_t} \mathbb{E}_{t} \left\langle   \nabla f({x}_{t} )-  \nabla f_i({x}_{t}), \nabla f({x}_{t} )-  \nabla f_j({x}_{t})\right\rangle  \\
		&\overset{(i)}=\mathbb{E}_{{x}_0,...,{x}_t}  \frac{1}{|B_t|^2 }  \sum_{i \in B_t}   \mathbb{E}_{t}  \norml{ \nabla f({x}_{t} )-  \nabla f_i({x}_{t})}^2 \overset{(ii)}\leq \mathbb{E}  \frac{\sigma^2}{|B_t|  } ,   
	\end{align*}
	where (i) follows from  $\mathbb{E}_{t} \nabla f_i({x}_{t}) = \nabla f({x}_{t}) $, and $\mathbb{E}_{t} \left\langle   \nabla f({x}_{t} )-  \nabla f_i({x}_{t}), \nabla f({x}_{t} )-  \nabla f_j({x}_{t})\right\rangle = 0$ for $i \neq j$, and (ii) follows from item (3) in Assumption~\ref{assum1}.
	For the case when  $|B_t| = n$, we have  ${v}_t = \nabla f({x}_t)$,  and thus $\mathbb{E}  \|  \nabla f({x}_{t} )-{v}_{t} \|^2 = 0$. Combining the above two cases, we have
		\begin{align}
		\mathbb{E}  \|  \nabla f({x}_{t} )-{v}_{t} \|^2 \leq  \mathbb{E} \left( \frac{I_{(|B_t|<n)}}{|B_t|} \sigma^2 \right)\label{abaSGD_2}.
	\end{align}
	Plugging \eqref{abaSGD_2} into \eqref{abaSGD_1}, we obtain
	\begin{align*}
		\Big (\eta-\frac{L\eta^2}{2} \Big) \mathbb{E} \| {v}_{t } \|^2  \leq  \mathbb{E}  f({x}_{t } )	 - \mathbb{E}  f({x}_{t+1})   + \mathbb{E}  \frac{I_{(|B_t|<n)}}{|B_t|} \eta \sigma^2.
	\end{align*}
	
	Telescoping the above inequality over $t$ from $0$ to $T $ yields
	\begin{align}
	\sum_{t=0}^{T}	\Big (\eta-\frac{L\eta^2}{2} \Big) \mathbb{E} \| {v}_{t } \|^2  \leq  \mathbb{E}  f({x}_{0 } )	 - \mathbb{E}  f({x}_{T+1})   + \sum_{t=0}^{T} \mathbb{E}  \frac{I_{(|B_t|<n)}}{|B_t|} \eta \sigma^2. \label{abaSGD_3} 
	\end{align}
	
	Next,  we upper-bound $\sum_{t=0}^{T} \mathbb{E}  \big(\frac{I_{(|B_t|<n)}}{|B_t|} \eta \sigma^2\big)$ in the above inequality through the following steps.
	\begin{align*}
		\sum_{t=0}^{T} \mathbb{E}  \frac{I_{(|B_t|<n)}}{|B_t|} \eta \sigma^2 &\overset{(i)}\leq \sum_{t=0}^{T} \mathbb{E}    \left(  \frac{\sum_{i=1}^{m}\|{v}_{t-i}\|^2 }{2m\sigma^2} + \frac{\epsilon}{ 24\sigma^2} \right){\eta \sigma^2}  \\
		&=  \frac{\eta}{2m } \sum_{t=0}^{T}      \sum_{i=1}^{m} \mathbb{E} \|{v}_{t-i}\|^2  +\sum_{t=0}^{T}    \frac{\eta\epsilon}{ 24 }  \\ 
		&= {   \frac{\eta}{2m } \sum_{t=1}^{T} \sum_{i=1}^{\min\{m,t\}} \mathbb{E} \|{v}_{t-i}\|^2 + \frac{\eta}{2m } \sum_{t=0}^{\min \{m-1, T\}} \sum_{i=t+1}^{m} \mathbb{E} \|{v}_{t-i}\|^2 +\sum_{t=0}^{T}    \frac{\eta\epsilon}{ 24 } }\\
		&\overset{(ii)}\leq \frac{\eta}{2m } \sum_{t=1}^{T} \sum_{i=1}^{\min\{m,t\}} \mathbb{E} \|{v}_{t-i}\|^2 + \frac{\eta}{2m } \sum_{t=0}^{m-1} \sum_{i=t+1}^{m} \mathbb{E} \|{v}_{-1}\|^2 +\sum_{t=0}^{T}    \frac{\eta\epsilon}{ 24 } \\
		&= \frac{\eta}{2m } \sum_{t=1}^{T} \sum_{i=1}^{\min\{m,t\}} \mathbb{E} \|{v}_{t-i}\|^2 + \frac{\eta m}{2  }   \|{v}_{-1}\|^2 +\sum_{t=0}^{T}    \frac{\eta\epsilon}{ 24 } \\
		&= \frac{\eta}{2m } \sum_{i=0}^{T-1} \mathbb{E} \|{v}_{ i}\|^2 \sum_{t=i+1}^{\min \{i+m, T\}} 1  + \frac{\eta m}{2  }   \|{v}_{-1}\|^2 +\sum_{t=0}^{T}    \frac{\eta\epsilon}{ 24 } \\
		&\leq\frac{\eta}{2  } \sum_{i=0}^{T } \mathbb{E} \|{v}_{ i}\|^2   + \frac{\eta m}{2  }   \alpha_0^2 +\sum_{t=0}^{T}    \frac{\eta\epsilon}{ 24 }   \numberthis \label{abaSGD_4}
	\end{align*}
	where (i) follows from the definition of $|B_t|$, (ii) follows from the fact that $\|{v}_{-1}\| = \|{v}_{-2}\| = \cdots  =  \|{v}_{-m}\|=\alpha_0$.
		
	Plugging \eqref{abaSGD_4} into \eqref{abaSGD_3}, we obtain
	\begin{align*}
		\sum_{t=0}^{T}	\Big (\eta-\frac{L\eta^2}{2} \Big) \mathbb{E} \| {v}_{t } \|^2  \leq  \mathbb{E}  f({x}_{0 } )	 - \mathbb{E}  f({x}_{T+1})   + \frac{\eta}{2  } \sum_{i=0}^{T } \mathbb{E} \|{v}_{ i}\|^2   + \frac{\eta m}{2  } \alpha_0^2 +\sum_{t=0}^{T}    \frac{\eta\epsilon}{ 24 },
	\end{align*}
	which further yields 
	\begin{align*}
		\sum_{t=0}^{T}	\Big (\frac{\eta}{2}-\frac{L\eta^2}{2} \Big) \mathbb{E} \| {v}_{t } \|^2  &\leq  \mathbb{E}  f({x}_{0 } )	 - \mathbb{E}  f({x}_{T+1})    + \frac{\eta m}{2  }  \alpha_0^2 +\sum_{t=0}^{T}    \frac{\eta\epsilon}{ 24 } \\
		&\leq   f({x}_{0 } )	 -   f^*   + \frac{\eta m}{2  }   \alpha_0^2 +\sum_{t=0}^{T}    \frac{\eta\epsilon}{ 24 }. \numberthis \label{abaSGD_5}
	\end{align*}
Recall that $\phi \defeq \Big (\eta-\frac{L\eta^2}{2} \Big) > 0$.  Then, we obtain  from \eqref{abaSGD_5} that 
	\begin{align}
		\sum_{t=0}^{T} \mathbb{E} \| {v}_{t } \|^2   \leq \frac{2\left( f({x}_{0 } )	 -   f^*  \right) +  {\eta m}  \alpha_0^2}{2\phi} + \frac{(T+1)\eta \epsilon}{24\phi}. \label{abaSGD_7}
	\end{align}
	
	Recall that the output ${x}_{\zeta}$ is chosen  from $\{ {x}_{t}\}_{t=0,...,T}$ uniformly at random. Then, based on \eqref{abaSGD_7}, we have
	\begin{align*}
	\mathbb{E} \norml{\nabla f({x}_\zeta)}^2 &= \frac{1}{T}   \sum_{t=1}^{T} \mathbb{E} \norml{\nabla f({x}_t)}^2 =  \frac{1}{T}   \sum_{t=1}^{T} \mathbb{E} \norml{ \mathbb{E}_{t} {v}_t}^2 \overset{(i)}\leq\frac{1}{T}   \sum_{t=1}^{T} \mathbb{E} \big(\mathbb{E}_{t} \norml{  {v}_t}^2\big) \\
	&= \frac{1}{T}   \sum_{t=1}^{T} \mathbb{E}  \norml{  {v}_t}^2 	 \overset{(ii)}\leq  \frac{2\left( f({x}_{0 } )	 -   f^*  \right)    +  {\eta m}   \alpha_0^2}{2T\phi} + \frac{(T+1)\eta \epsilon}{24T\phi}\\
	&\leq    \frac{2\left( f({x}_{0 } )	 -   f^*  \right)    +  {\eta m}  \alpha_0^2}{2T\phi} + \frac{ \eta \epsilon}{12 \phi},
	\end{align*}
	where (i) follows from the Jensen's inequality, and (ii) follows from \eqref{abaSGD_7}.
	
\subsection{Proof of Corollary~\ref{co:adasgd}}
	Since $\eta = \frac{1}{2L}$, have
	\begin{align*}
			\phi = \Big (\eta-\frac{L\eta^2}{2} \Big) =\frac{1}{8L} > 0.
	\end{align*}
	Then, plugging $\eta = \frac{1}{2L}, \phi =\frac{1}{8L}$  and $T =  \left(16L\left(f({x}_{0 } )	 -   f^*  \right) +  4m   \alpha_0^2\right)\epsilon^{-1}$  in  Theorem~\ref{th_abasgd}, we have 
	\begin{align*}
	\mathbb{E} \norml{\nabla f({x}_\zeta)}^2  	&\leq \frac{8L\left(f({x}_{0 } )	 -   f^*  \right) +  2m    \alpha_0^2}{T}   +   \frac{ \epsilon}{ 3 } \leq \frac{5}{6} \epsilon \leq \epsilon.
	\end{align*} 
	Thus, the total SFO calls required by AbaSGD is given by 
	\begin{align*}
	\sum_{t=0}^{T} |B_t| = \sum_{t=0}^{T}  \min \left\{ \frac{2\sigma^2}{\sum_{i=1}^{m}\|{v}_{t-i}\|^2/m},\frac{ 24\sigma^2}{\epsilon} , n\right\} \leq (T+1)\left(\frac{ 24\sigma^2}{\epsilon} \wedge n\right)= \mathcal{O}  \left(\frac{1}{\epsilon^2} \wedge \frac{n}{\epsilon}\right). 
	\end{align*}

\end{document}